\RequirePackage{fix-cm}

\documentclass[smallextended]{svjour3}       %
\usepackage{amsmath,indentfirst,amsfonts,dsfont,enumerate,mathrsfs,color}
\usepackage{bm,multirow,booktabs}
\usepackage{amsopn,amssymb}
\usepackage{algorithm,hyperref}               
\usepackage{algcompatible}             
\usepackage{rotating}
\usepackage[dvips]{epsfig,graphicx}
\usepackage{stmaryrd}
\usepackage{listings,url,verbatim}
\usepackage{lscape}
\usepackage{graphicx,subfig}
\smartqed
\begin{document}

\title{How many integrals should be evaluated at least in two-dimensional hyperinterpolation?\thanks{This work was supported by the National Natural Science Foundation of China (No. 12561095, 12371099 and U24A2001), the Special Posts of Guizhou University (No. [2025]06 and [2024]42), the Science and Technology Commission of Shanghai Municipality under grant 23JC1400501, the Hong Kong Research Grants Council (Project 11204821), the Hong Kong Innovation and Technology Commission (InnoHK Project CIMDA) and City University of Hong Kong (Project 9610460).}
}


\author{Maolin Che     \and
        Congpei An     \and
        Yimin Wei      \and
        Hong Yan
}


\institute{
        M. Che \at
            School of Mathematics and Statistics and State Key Laboratory of Public Big Data, Guizhou University, Guiyang, 550025, Guizhou, P. R. of China
              \email{mlche@gzu.edu.cn ; chncml@outlook.com}           
        \and
        C. An (Corresponding author) \at
            School of Mathematics and Statistics, Guizhou University, Guiyang, 550025, Guizhou, P. R. of China  \email{andbachcp@gmail.com;ancp@gzu.edu.cn}
        \and
        Y. Wei \at
            School of Mathematical Sciences and Key Laboratory of Mathematics for Nonlinear Sciences,  Fudan University, Shanghai, 200433, P. R. China \email{ymwei@fudan.edu.cn}
        \and
        H. Yan \at
           Department of Electrical Engineering and Center for Intelligent Multidimensional Data Analysis, City University of Hong Kong, 83 Tat Chee Avenue, Kowloon, Hong Kong (\email{h.yan@cityu.edu.hk})
}

\date{Received: date / Accepted: date}

\maketitle

\begin{abstract}
  This paper introduces a novel approach to approximating continuous functions over high-dimensional hypercubes by integrating matrix CUR decomposition with hyperinterpolation techniques. Traditional Fourier-based hyperinterpolation methods suffer from the curse of dimensionality, as the number of coefficients grows exponentially with the dimension. To address this challenge, we propose two efficient strategies for constructing low-rank matrix CUR decompositions of the coefficient matrix, significantly reducing computational complexity while preserving accuracy.

  The first method employs structured index selection to form a compressed representation of the tensor, while the second utilizes adaptive sampling to further optimize storage and computation. Theoretical error bounds are derived for both approaches, ensuring rigorous control over approximation quality. Additionally, practical algorithms---including randomized and adaptive decomposition techniques---are developed to efficiently compute the CUR decomposition. Numerical experiments demonstrate the effectiveness of our methods in drastically reducing the number of required coefficients without compromising precision.

  Our results bridge matrix/tensor decomposition and function approximation, offering a scalable solution for high-dimensional problems. This work advances the field of numerical analysis by providing a computationally efficient framework for hyperinterpolation, with potential applications in scientific computing, machine learning, and data-driven modeling.

\end{abstract}

\section{Introduction}
The efficient approximation of multivariate functions is a cornerstone of computational mathematics, with applications spanning scientific computing, uncertainty quantification, and data-driven modeling, which can be obtained by Fourier series, Chebyshev Series, Chebyshev interpolation and so on. For some functions, spectral methods based on these methods offer exponential convergence rates, but they face the curse of dimensionality (cf. \cite{mason1980near,trefethen2017multivariate}). To overcome this issue, many scholars introduce matrix/tensor decompositions for obtaining an efficient approximation of multivariate functions. Soley {\it et al.} \cite{soley2021functional} extended the capabilities of the Chebyshev propagation scheme to high-dimensional systems in terms of the functional tensor-train (see \cite{bigoni2016spectral}) Chebyshev method. Griebel {\it et al.} \cite{griebel2023low} presented a method to construct multivariate low-rank approximations of functions from Sobolev spaces with
dominating mixed smoothness based on the tensor train decomposition (see \cite{oseledets2011tensor}). Saibaba {\it et al.} \cite{saibaba2022efficient} studied the properties of the Tucker decomposition method for approximating the Chebyshev polynomial of any multivariate function: computational complexity analysis, the error of the resulting approximations, and the benefits of this algorithm. Dolgov {\it et al.} \cite{dolgov2021functional} proposed a novel algorithm (i.e., Chebfun3F) for approximating a trivariate function defined on a tensor-product domain via function evaluations, which is similar to Chebfun3 (see \cite{hashemi2017chebfun}) and utilizes univariate fibers instead of
bivariate slices to construct the Tucker decomposition. Griebel and Harbrecht \cite{griebel2023analysis} considered the continuous versions of the Tucker tensor format and of the tensor train format for the approximation of multivariate functions under certain conditions. As is well known, there are few literature considering how to efficiently and fast approximate the truncated Fourier series of multivariate functions without explicitly forming all Fourier coefficients in advance. Hence, we will provide some efficient strategies for this problem based on matrix/tensor decompositions. In this work, we focus on functions with two variables. The case of functions with three (or more) variables will be considered in the upcoming work.

Based on practical problems and application scenarios, where exist many types for matrix decompositions (cf. \cite{golub2013matrix}), such as, singular value decomposition (SVD), matrix CUR decomposition, interpolative decomposition, QR factorization with column pivoting and LU decomposition with complete pivoting. The proposed algorithm in this work is based on matrix CUR decomposition (cf. \cite{park2025accuracy,hamm2021perturbations,boutsidis2017optimal,mahoney2009cur,sorensen2016deim,voronin2017efficient}), which have emerged as powerful tools for analyzing and approximating high-dimensional data, offering a framework to mitigate the curse of dimensionality inherent in traditional methods. Among these, the matrix CUR decomposition stands out for its ability to represent matrices in a compressed format through a core tensor and factor matrices, enabling efficient storage and computation. This decomposition is particularly useful in applications such as signal processing, machine learning, and numerical analysis, where high-dimensional data is prevalent.

In this paper, we explore the application of matrix CUR decomposition to hyperinterpolation for continuous functions defined over a squared region. Hyperinterpolation, a generalization of polynomial interpolation, approximates functions by projecting them onto a finite-dimensional subspace spanned by orthogonal bases. While traditional hyperinterpolation methods rely on Fourier series expansions, the exponential growth of coefficients in high dimensions poses significant computational challenges. Our work addresses this issue by leveraging matrix CUR decomposition to reduce the number of coefficients required for accurate approximation.

We present two efficient strategies to compute the matrix CUR decomposition of the coefficient matrix derived from the Fourier series expansion of a continuous function. The first strategy involves selecting specific subsets of indices to construct low-rank approximations, while the second employs adaptive sampling techniques to further reduce computational costs. Theoretical bounds on the approximation error are provided, ensuring the reliability of our methods. Additionally, we discuss practical algorithms for implementing these strategies, including randomized and adaptive techniques, and demonstrate their effectiveness through numerical examples.

This study bridges the gap between matrix decomposition and function approximation, offering a scalable solution for high-dimensional problems. By combining the theoretical foundations of matrix CUR decomposition with practical algorithmic innovations, we provide a robust framework for hyper-interpolation that is both computationally efficient and mathematically sound. Our results contribute to the broader field of numerical analysis and open new avenues for research in high-dimensional function approximation.
\subsection{Organizations}

The remainder of this paper is organized as follows. In Section \ref{t-lasso:sec2-main}, we overview several basic definitions and notations: the matrix volume, two types of matrix norms, and the $\epsilon$-rank of a matrix. In Section \ref{t-lasso:sec3-main}, we provide the theoretical analysis for approximating the truncated Fourier series based on the fixed rank or precision problem of matrix decomposition. In Section \ref{t-lasso:sec4-main}, we first form an upper bound for $\|\widetilde{F}f(x_1,x_2)-f(x_1,x_2)\|_{L^2}$ when the corresponding matrix CUR decomposition is obtained based on the maximal volume sampling strategy, and then provide two adaptive deterministic greedy algorithms to compute the CUR decomposition of the coefficient matrix (in practice, this matrix is unknown). In Section \ref{t-lasso:sec5-main} several examples are used to illustrate the accuracy and efficiency of the proposed algorithms. We conclude this paper in Section \ref{t-lasso:sec6-main}.
\section{Preliminaries}
\label{t-lasso:sec2-main}
Let $\mathbb{I}=\{i_1,i_2,\dots,i_K\}$ with $1\leq i_1<i_2<\dots<i_K\leq I$, and $\mathbb{J}=\{j_1,j_2,\dots,j_L\}$ with $1\leq j_1<j_2<\dots<j_L\leq J$. The zero matrix in $\mathbb{R}^{I\times J}$ is denoted by $\mathbf{0}_{I,J}$. For any matrix $\mathbf{A}\in\mathbb{C}^{I\times J}$, we use $\mathbf{A}(:,\mathbb{J})$ to denote the $I\times L$ submatrix of $\mathbf{A}$ containing only those columns of $\mathbf{A}$ indexed by $\mathbb{J}$, and $\mathbf{A}(\mathbb{I},:)$ to denote the $K\times J$ submatrix of $\mathbf{A}$ containing only those rows of $\mathbf{A}$ indexed by $\mathbb{I}$. The matrix $\mathbf{Q}\in\mathbb{C}^{I\times R}$ is orthonormal if $\mathbf{Q}^*\mathbf{Q}$ is the identity matrix. The volume of $\mathbf{A}\in\mathbb{C}^{I_1\times I_2}$ is the product of its singular values $V(\mathbf{A})=\prod_{k=1}^{\min\{I_1,I_2\}}\sigma_k(\mathbf{A})$. For a given matrix $\mathbf{A}\in\mathbb{C}^{I_1\times I_2}$, its maximum absolute entry norm and Frobenius norm are, respectively, defined as
\begin{equation*}
    \|\mathbf{A}\|_{\max}=\max_{\substack{i_1=1,2,\dots,I_1\\ i_2=1,2,\dots,I_2}}|a_{i_1i_2}|,\quad \|\mathbf{A}\|_F=\sqrt{\sum_{i_1=1}^{I_1}\sum_{i_2=1}^{I_2}|a_{i_1i_2}|^2},
\end{equation*}

We introduce the definition of the $\epsilon$-rank of a matrix in the following definition. Several alternative definitions of $\epsilon$-rank are discussed in \cite{beckermann2019bounds}.
\begin{definition}{{\bf (\cite[Definition 2.1]{udell2019big})}}
\label{t-lasso:epsilonrank:definition-main}
    Let $\mathbf{X}\in\mathbb{C}^{I\times J}$ be a matrix and $0<\epsilon<1$ be a tolerance. The (absolute) $\epsilon$-rank of $\mathbf{X}$ is given by
    \begin{equation*}
        {\rm rank}_{\epsilon}(\mathbf{X})=\min\{{\rm rank}(\mathbf{A}):\mathbf{A}\in\mathbb{C}^{I\times J},\|\mathbf{X}-\mathbf{A}\|_{\max}\leq \epsilon\}.
    \end{equation*}
    That is, $R={\rm rank}_{\epsilon}(\mathbf{X})$ is the smallest integer for which $\mathbf{X}$ can be approximated by a rank-$R$ matrix, up to an accuracy of $\epsilon$.
\end{definition}

\section{Function approximation by truncating Fourier series}
\label{t-lasso:sec3-main}
We first introduce several notations used in this and the following sections. Let $\mathbb{Z}$ be the set of all integers and define $\mathbb{Z}_N=\{(k_1,k_2):k_n\in\mathbb{Z},n=1,2\}$. Let $f(\cdot,\cdot)$ be a continuous periodic function of two variables $x_1$ and $x_2$ over the square
\begin{equation*}
    \mathbb{H}_2^0:=[-\pi,\pi]^2=\{(x_1,x_2):x_n\in[-\pi,\pi],n=1,2\}.
\end{equation*}
Then for any $(k_1,k_2)\in\mathbb{Z}_2$, the Fourier transform of $f(x_1,x_2)$ is given by
\begin{equation*}
    \widehat{f}_{(k_1,k_2)}=\left(\frac{1}{2\pi}\right)^2\int_{\mathbb{H}_2^0}f(x_1,x_2)e^{-\iota\cdot (k_1x_1+k_2x_2)}dx_1dx_2
\end{equation*}
with $\iota=\sqrt{-1}$. Denote the Hilbert space on $\mathbb{H}_2^0$ as
\begin{equation*}
    L^2(\mathbb{H}_2^0)=\left\{f(x_1,x_2):\sum_{(k_1,k_2)\in\mathbb{Z}_2}|\widehat{f}_{(k_1,k_2)}|<+\infty\right\}
\end{equation*}
equipped with the inner product
\begin{equation*}
    \langle f_1(x_1,x_2),f_2(x_1,x_2)\rangle=\left(\frac{1}{2\pi}\right)^2\int_{\mathbb{H}_2^0}f_1(x_1,x_2)\bar{f}_2(x_1,x_2)dx_1dx_2,
\end{equation*}
which implies  that
\begin{equation*}
    \|f(x_1,x_2)\|_{L^2}^2=\langle f(x_1,x_2),f(x_1,x_2)\rangle=\sum_{(k_1,k_2)\in\mathbb{Z}_2}|\widehat{f}_{(k_1,k_2)}|^2.
\end{equation*}
For any integer $\alpha>0$, the $\alpha$-derivative Sobolev space on $\mathbb{H}_2^0$ is given as
\begin{equation*}
    H^{\alpha}(\mathbb{H}_2^0)=\left\{f(x_1,x_2)\in L^2(\mathbb{H}_2^0):\|f(x_1,x_2)\|_\alpha<+\infty\right\},
\end{equation*}
where $$\|f(x_1,x_2)\|_\alpha=\left(\sum_{(k_1,k_2)\in\mathbb{Z}_2}(1+\|(k_1,k_2)\|_2^{2\alpha})|\widehat{f}_{(k_1,k_2)}|^2\right)^{1/2}$$ with $\|(k_1,k_2)\|_2^2=|k_1|^2+|k_2|^2$. Meanwhile, the semi-norm of $H^{\alpha}(\mathbb{H}_2^0)$ can be defined as $$|f(x_1,x_2)|_\alpha=\left(\sum_{(k_1,k_2)\in\mathbb{Z}_2}\|(k_1,k_2)\|_2^{2\alpha}|\widehat{f}_{(k_1,k_2)}|^2\right)^{1/2}.$$

Let $F(\cdot,\cdot)$ denote the projection of the function $f(\cdot,\cdot)$ on the partial sum of orders $I_1,I_2$ in $x_1,x_2$, respectively, of its Fourier series expansion (cf. \cite{mason1980near,trefethen2017multivariate}). Then one has
\begin{align}
    Ff(x_1,x_2)&=\sum_{k_1=-I_1}^{I_1}\sum_{k_2=-I_2}^{I_2}\alpha_{k_1,k_2}e^{\iota\cdot k_1x_1}e^{\iota \cdot k_2x_2}
    \label{t-lasso:orginal-expression-main}
\end{align}
with $(x_1,x_2)^\top\in\mathbb{H}_2^0$, where each $\alpha_{k_1,k_2}$ is given as
\begin{equation*}
    \alpha_{k_1,k_2}=\left(\frac{1}{2\pi}\right)^2\int_{\mathbb{H}_2^0}f(x_1,x_2)e^{-\iota\cdot (k_1x_1+k_2x_2)}dx_1dx_2.
\end{equation*}

For clarity, we use several tensor-product quadrature methods (cf. \cite{shin2017randomized,wu2017randomized}) to approximate all double integrals in this paper. More details are discussed in Section \ref{t-lasso:sec5-main}. Under this case, $Ff(x_1,x_2)$ can be viewed as a special case for the hyperinterpolation of $f(x_1,x_2)$ onto the subspace generated by
$$\{e^{\iota\cdot k_1x_1}e^{\iota \cdot k_2x_2}:k_n=-I_n,-I_n+1,\dots,I_n,n=1,2\},$$
with $-\pi\leq x_1,x_2\leq \pi$. To make the paper as self-contained as possible, we report the concept of hyperinterpolation and its variants in Section \ref{t-lasso:app3-main}.

Now we introduce the upper bound for $\|Ff(x_1,x_2)-f(x_1,x_2)\|_{L^2}$ (see \cite{canuto2006spectral}), which is summarized in the following theorem.
\begin{theorem}
    For each $f(x_1,x_2)\in H^{\alpha}(\mathbb{H}_2^0)$, there exists a constant $C$, independent of $f(x_1,x_2)$ and $\{I_1,I_2\}$, such that
    \begin{equation*}
        \|Ff(x_1,x_2)-f(x_1,x_2)\|_{L^2}\leq C\min\{I_1,I_2\}^{-\alpha}|f(x_1,x_2)|_{\alpha}.
    \end{equation*}
    \label{t-lasso:general-theorem-main}
\end{theorem}
\begin{remark}
    Following Theorem {\rm\ref{t-lasso:general-theorem-main}}, for a given $0<\epsilon<1$, when we assume that $\min\{I_1,I_2\}\geq (C|f(x_1,x_2)|_{\alpha}/\epsilon)^{1/\alpha}$, one has that $\|Ff(x_1,x_2)-f(x_1,x_2)\|_{L^2}\leq \epsilon$. Similarly, as a special case of Theorem 2 in {\rm \cite{griebel2023analysis}}, for a given function $f(x_1,x_2)\in L^2([a,b]^2)$, when for each $n$, $I_n\geq (1/\epsilon)^{1/\alpha}$, one has $\|Ff(x_1,x_2)-f(x_1,x_2)\|_{L^2}\leq \sqrt{2}\epsilon$ with
    \begin{equation*}
        Ff(x_1,x_2)=\sum_{i_1=1}^{I_1}\sum_{i_2=1}^{I_2}\mathbf{A}(i_1,i_2)\phi_{i_1}^{(1)}(x_1)\phi_{i_2}^{(2)}(x_2),
    \end{equation*}
    where for each $n$, the left eigenfunctions $\{\phi_{i_n}^{(n)}\}_{i_n=1}^{+\infty}$ form a orthonormal basis in $L^2([a,b])$, and the coefficient $\mathbf{A}(i_1,i_2)$ is given by
    \begin{equation*}
        \mathbf{A}(i_1,i_2)=\int_{[a,b]^2}f(x_1,x_2)\phi_{i_1}^{(1)}(x_1)\phi_{i_2}^{(2)}(x_2)dx_1dx_2.
    \end{equation*}

    Here, for the given function $f$, the left eigenfunctions $\{\phi_{i_n}^{(n)}\}_{i_n=1}^{+\infty}$ can be obtained by applying the singular value decomposition (also called Karhunen-L\`{o}eve expansion or the Schmidt decomposition \cite{bigoni2016spectral,stewart1993early}) to separate $x_1\in[a,b]$ and $x_2\in[a,b]$.
\end{remark}
\subsection{General theory for approximating (\ref{t-lasso:orginal-expression-main})}
For clarity, we define a matrix $\mathbf{A}\in\mathbb{C}^{(2I_1+1)\times (2I_2+1)}$ as
\begin{equation}
\label{t-lasso:allcoefficients-main}
    \mathbf{A}(i_1,i_2)=\alpha_{i_1-I_1-1,i_2-I_2-1}
\end{equation}
with $i_n=1,2,\dots,2I_n+1$ and $n=1,2$. Note that (\ref{t-lasso:allcoefficients-main}) quickly becomes impractical as the parameter dimension $N$ increases, due to the exponential growth in the number of coefficients in $\mathbf{A}$. This growth is asymptote of the curse of dimensionality. In this work, we will consider giving an approximation for (\ref{t-lasso:orginal-expression-main}), by effectively reducing the number of coefficients that must be computed.

Following Definition \ref{t-lasso:epsilonrank:definition-main}, for a given tolerance $0<\epsilon<1$, the $\epsilon$-rank of $\mathbf{A}$ in (\ref{t-lasso:allcoefficients-main}) is presented in the following theorem.
\begin{theorem}{{\bf (see \cite[Theorem 1.0]{udell2019big})}}
    Let $\mathbf{A}\in\mathbb{C}^{(2I_1+1)\times (2I_2+1)}$ and $0<\epsilon<1$. Then, there exists a matrix $\widetilde{\mathbf{A}}\in\mathbb{C}^{(2I_1+1)\times (2I_2+1)}$ such that ${\rm rank}(\widetilde{\mathbf{A}})\leq R$ and
    \begin{equation*}
        \|\widetilde{\mathbf{A}}-\mathbf{A}\|_{\max}\leq \epsilon\|\mathbf{A}\|_2,
    \end{equation*}
    where $R=\lceil 72\ln(2(I_1+I_2)+3)/\epsilon^2\rceil$.
    \label{t-lasso:epsilonrank-main}
\end{theorem}

The following corollary is easily obtained from Theorem \ref{t-lasso:epsilonrank-main}.
\begin{corollary}
    Let $\mathbf{A}\in\mathbb{C}^{(2I_1+1)\times (2I_2+1)}$ and $0<\epsilon<1$. Then, there exists a matrix $\widetilde{\mathbf{A}}\in\mathbb{C}^{(2I_1+1)\times (2I_2+1)}$ such that ${\rm rank}(\widetilde{\mathbf{A}})\leq R$ and
    \begin{equation*}
        \|\widetilde{\mathbf{A}}-\mathbf{A}\|_{\max}\leq \epsilon\|\mathbf{A}\|_F,
    \end{equation*}
    where $R=\lceil 72\ln(2(I_1+I_2)+3)/\epsilon^2\rceil$.
    \label{t-lasso:epsilonrank-corollary-main}
\end{corollary}

Let $\widetilde{\mathbf{A}}=\mathbf{E}\mathbf{F}^*$ satisfy the conditions and conclusions of Theorem \ref{t-lasso:epsilonrank-main} or Corollary \ref{t-lasso:epsilonrank-corollary-main} with $0<\epsilon<1$, where $\mathbf{E}\in\mathbb{C}^{(2I_1+1)\times R}$ and $\mathbf{F}\in\mathbb{C}^{(2I_2+1)\times R}$ are two matrices with full column rank. Then we deduce an approximation to $Ff(x_1,x_2)$ in (\ref{t-lasso:orginal-expression-main}) as follows,
\begin{equation}
    \widetilde{F}f(x_1,x_2)=\sum_{j=1}^{R}\phi_{j}^{(1)}(x_1)\phi_{j}^{(2)}(x_2),
    \label{t-lasso:approximation-expression-general-main}
\end{equation}
where the functions $\phi_{j}^{(1)}(x_1)$ and $\phi_{j}^{(2)}(x_2)$ are given by
\begin{align*}
    \phi_{j}^{(1)}(x_1)&=\sum_{i_1=1}^{2I_1+1}\mathbf{E}(i_1,j)e^{\iota\cdot (i_1-I_1-1)x_1},\\
    \phi_{j}^{(2)}(x_2)&=\sum_{i_2=1}^{2I_2+1}\mathbf{F}(i_2,j)e^{\iota\cdot (i_2-I_2-1)x_2}.
\end{align*}

Hence, the upper bound for $\|\widetilde{F}f(x_1,x_2)-f(x_1,x_2)\|_{L^2}$ is obtained by combining Theorems \ref{t-lasso:general-theorem-main} with \ref{t-lasso:epsilonrank-main}, which is summarized in the following theorem.
\begin{theorem}
    Let $f(x_1,x_2)\in H^{\alpha}(\mathbb{H}_2^0)$. Suppose that $\mathbf{A}\in\mathbb{C}^{(2I_1+1)\times (2I_2+1)}$ is given by {\rm (\ref{t-lasso:allcoefficients-main})}. For a given $0<\epsilon<1$, there exist two full column rank matrices $\mathbf{E}\in\mathbb{C}^{(2I_1+1)\times R}$ and $\mathbf{F}\in\mathbb{C}^{(2I_2+1)\times R}$ such that
    \begin{align*}
        &\|\widetilde{F}f(x_1,x_2)-f(x_1,x_2)\|_{L^2}\\
        &\quad\leq C\max\{I_1,I_2\}^{-\alpha}|f(x_1,x_2)|_{\alpha}+\epsilon\sqrt{(2I_1+1)(2I_2+1)}\|\mathbf{A}\|_2,
    \end{align*}
    where $\widetilde{F}f(x_1,x_2)$ is in {\rm(\ref{t-lasso:approximation-expression-general-main})}, the constant $C$ is independent of $f(x_1,x_2)$ and $\{I_1,I_2\}$, and $R$ is given by $R=\lceil 72\ln(2(I_1+I_2)+3)/\epsilon^2\rceil$.
    \label{t-lasso:general-approximation-theorem-main}
\end{theorem}
\begin{proof}
    Note that we have $$\widetilde{F}f(x_1,x_2)-f(x_1,x_2)=\widetilde{F}f(x_1,x_2)-Ff(x_1,x_2)+Ff(x_1,x_2)-f(x_1,x_2),$$
    which implies that
    \begin{align*}
        &\|\widetilde{F}f(x_1,x_2)-f(x_1,x_2)\|_{L^2}\\
        &\quad\leq \|\widetilde{F}f(x_1,x_2)-Ff(x_1,x_2)\|_{L^2}+\|Ff(x_1,x_2)-f(x_1,x_2)\|_{L^2}.
    \end{align*}
Combining (\ref{t-lasso:orginal-expression-main}) and (\ref{t-lasso:approximation-expression-general-main}), one has
\[
\widetilde{F}f(x_1,x_2) - Ff(x_1,x_2) = \sum_{i_1=1}^{2I_1+1}\sum_{i_2=1}^{2I_2+1} (\mathbf{A}(i_1,i_2) - \tilde{\mathbf{A}}(i_1,i_2)) e^{\iota k_1 x_1} e^{\iota k_2 x_2}
\]
with $k_n=i_n-I_n-1$. Due to orthonormality of the Fourier basis $\{e^{\iota k_1 x_1} e^{\iota k_2 x_2}\}$ under the scaled inner product,
\[
\|\widetilde{F}f(x_1,x_2)-Ff(x_1,x_2)\|_{L^2}^2 = \sum_{i_1,i_2} |\mathbf{A}(i_1,i_2) - \tilde{\mathbf{A}}(i_1,i_2)|^2 = \| \mathbf{A} - \tilde{\mathbf{A}} \|_F^2.
\]
Using the inequality $\| \cdot \|_F \leq \sqrt{mn} \| \cdot \|_{\max}$ for an $m \times n$ matrix,
\[
\|\widetilde{F}f(x_1,x_2)-Ff(x_1,x_2)\|_{L^2} \leq \sqrt{(2I_1+1)(2I_2+1)} \cdot \| \mathbf{A} - \tilde{\mathbf{A}} \|_{\max}.
\]
    Therefore, the proof is complete by combining Theorems \ref{t-lasso:general-theorem-main} and \ref{t-lasso:epsilonrank-main}.
\end{proof}

\subsection{The existing algorithms and drawbacks}

More generally, the problem of the low-rank approximation to a matrix $\mathbf{A}\in\mathbb{C}^{(2I_1+1)\times (2I_2+1)}$ can be divided into the following categories:
\begin{enumerate}
   \item[(a)] {\bf The fixed rank problem:} for a given integer $0<R<\min\{2I_1+1,2I_2+1\}$, to find a matrix $\widetilde{\mathbf{A}}\in\mathbb{C}^{(2I_1+1)\times (2I_2+1)}$ with ${\rm rank}(\widetilde{\mathbf{A}})\leq R$ such that the matrix $\widetilde{\mathbf{A}}$ is a minimum point to the following optimization problem
       \begin{align*}
          \min\|\mathbf{B}-\mathbf{A}\|_F,\quad \text{ s.t. } \mathbf{B}\in\mathbb{C}^{(2I_1+1)\times (2I_2+1)},\ {\rm rank}(\mathbf{B})\leq R.
       \end{align*}
   \item[(b)] {\bf The fixed-precision problem:} for a given tolerance $0<\epsilon<1$, to find a minimum positive integer $R$ and a matrix $\widetilde{\mathbf{A}}\in\mathbb{C}^{(2I_1+1)\times (2I_2+1)}$ such that
       \begin{align*}
       \|\widetilde{\mathbf{A}}-\mathbf{A}\|_2\leq \epsilon,\quad {\rm rank}(\widetilde{\mathbf{A}})\leq R.
       \end{align*}
\end{enumerate}

Note that SVD is an efficient algorithm the above two problems. In detail, let $\mathbf{A}=\mathbf{U}\mathbf{\Sigma}\mathbf{V}^*$, where $\mathbf{U}\in\mathbb{C}^{(2I_1+1)\times (2I_1+1)}$ and $\mathbf{V}\in\mathbb{C}^{(2I_2+1)\times (2I_2+1)}$ are unitary, and $\mathbf{\Sigma}\in\mathbb{C}^{(2I_1+1)\times (2I_2+1)}$ is diagonal whose diagonal entries $\{\sigma_k\}_{k=1}^{\min\{2I_1+1,2I_2+1\}}$ are the singular values of $\mathbf{A}$, ordered so that $\sigma_1\geq \sigma_2\geq \dots\geq\sigma_{\min\{2I_1+1,2I_2+1\}}\geq 0$. Hence, for the fixed rank approximation, the matrix $\widetilde{\mathbf{A}}$ is given as $\widetilde{\mathbf{A}}=\mathbf{A}_R:=\mathbf{U}_R\mathbf{\Sigma}_R\mathbf{V}_R^*$ with $\mathbf{U}_R=\mathbf{U}(:,1:R)$, $\mathbf{V}_R=\mathbf{V}(:,1:R)$ and $\mathbf{\Sigma}_R=\mathbf{\Sigma}(1:R,1:R)$; and for fixed-precision problem, we have $R=\min\{k:\sigma_k\leq \epsilon,k=1,2,\dots,\min\{2I_1+1,2I_2+1\}\}$ and $\widetilde{\mathbf{A}}=\mathbf{A}_R$. Computing the SVD of $\mathbf{A}$ needs $O((2I_1+1)(2I_2+1)\min\{2I_1+1,2I_2+1\})$ operations, which could be expensive when the matrix $\mathbf{A}$ has large dimension.

There exist several efficient randomized algorithms for solving the fixed-rank problem, including the randomized SVD method (see \cite{halko2011finding}) and the generalized Nystr\"{o}m method (see \cite{clarkson2009numerical}). From these two algorithms, we obtain an approximation $\widetilde{\mathbf{A}}\in\mathbb{C}^{(2I_1+1)\times (2I_1+1)}$ with ${\rm rank}(\widetilde{\mathbf{A}})\leq R$ such that $\|\widetilde{\mathbf{A}}-\mathbf{A}\|_F\approx\|\mathbf{A}_R-\mathbf{A}\|_F$ while needing $o((2I_1+1)(2I_2+1)\min\{2I_1+1,2I_2+1\})$ operations. As shown in \cite{halko2011finding}, the existing randomized algorithms for the fixed-rank problem can be divided into two stages.
In the second stage, an approximation $\widetilde{\mathbf{A}}$ with ${\rm rank}(\mathbf{A})\leq R$ is obtained according to the structure of any approximation $\mathbf{B}$ to $\mathbf{A}$.

We now discuss the way to obtain $\mathbf{B}$. The first stage of the randomized SVD algorithm contains three steps: a) generating a random matrix $\mathbf{\Omega}\in\mathbb{C}^{(2I_2+1)\times (R+L)}$; b) computing the product $\mathbf{Y}=\mathbf{A}\mathbf{\Omega}$; and c) orthonormalizing the columns of $\mathbf{Y}$ by a thin QR factorization, which is denoted by $\mathbf{Q}={\rm orth}(\mathbf{Y})$, where $L>0$ is any oversampling parameter. Then a low-rank approximation of $\mathbf{B}$ is obtained by setting as $\mathbf{B}=\mathbf{Q}(\mathbf{Q}^*\mathbf{A})$ with ${\rm rank}(\mathbf{B})\leq R+L$.
\begin{remark}
   The randomized SVD method is suitable for the case that the singular values of the matrix $\mathbf{A}$ decay fast. When the singular values of the matrix $\mathbf{A}$ decay slowly, we introduce two related versions: the randomized subspace iteration (RSI \cite{bjarkason2019pass,rokhlin2010randomized}) and the randomized block Krylov iteration (RBKI \cite{halko2011algorithm,musco2015randomized,rokhlin2010randomized}). For RSI, the matrix $\mathbf{Y}$ is given by $\mathbf{Y}=\mathbf{A}(\mathbf{A}^*\mathbf{A})^{q}\mathbf{\Omega}$, and for RBKI, the matrix $\mathbf{Y}$ is set as $\mathbf{Y}=[\mathbf{A}\mathbf{\Omega},\dots,\mathbf{A}(\mathbf{A}^*\mathbf{A})^{q}\mathbf{\Omega}]$, where $q\geq 1$ is any power parameter.
\end{remark}

Similarly, in the first stage of the generalized Nystr\"{o}m method, we generate two random matrices $\mathbf{\Omega}\in\mathbb{C}^{(2I_2+1)\times (R+L)}$ and $\mathbf{\Phi}\in\mathbb{C}^{(2I_1+1)\times (R+L+L')}$, compute three products $\mathbf{X}=\mathbf{A}^*\mathbf{\Phi}$, $\mathbf{Y}=\mathbf{A}\mathbf{\Omega}$ and $\mathbf{Z}=\mathbf{\Phi}^*\mathbf{A}\mathbf{\Omega}$, and obtain a low-rank approximation $\mathbf{B}=\mathbf{Y}\mathbf{Z}^\dag\mathbf{X}^*=(\mathbf{A}\mathbf{\Omega})(\mathbf{\Phi}^*\mathbf{A}\mathbf{\Omega})^\dag(\mathbf{A}^*\mathbf{\Phi})^*$ with ${\rm rank}(\mathbf{B})\leq R+L$. The subspace and block Krylov iteration versions (cf. \cite{halko2011finding,tropp2023randomized}) for the generalized Nystr\"{o}m method are also proposed for the case that the singular values of the matrix $\mathbf{A}$ decay slowly.

\begin{remark}
   Common types of random matrices used in the randomized SVD and generalized Nystr\"{o}m methods include standard Gaussian matrices (cf. \cite{tropp2017practical}), the subsampled randomized Fourier/Hadamard transform (cf. \cite{ailon2009fast,boutsidis2013improved,woolfe2008fast}), the sparse subspace embedding (cf. \cite{clarkson2017low}), the fast Kronecker Johnson-Lindenstrauss transform (cf. \cite{bamberger2022johnson,jin2021faster}), TensorSketch (cf. \cite{diao2018sketching,pagh2013compressed}), Multisketching (cf. \cite{clarkson2017low,sobczyk2022pylspack,woodruff2014sketching}), the Khatri--Rao product of standard Gaussian matrices (cf. \cite{chen2021tensor,chen2020structured}), and random sampling matrices generated by different sampling strategies. Several types for sampling strategies are shown in Appendix \ref{t-lasso:app1-main}.
\end{remark}

In general, the fixed-precision problem for low-rank approximations can be also reformulated as follows. Given a matrix $\mathbf{A}\in\mathbb{C}^{(2I_1+1)\times (2I_1+1)}$ and a tolerance $0<\epsilon<1$, the goal is to find a matrix $\mathbf{Q}$ with orthonormal columns that satisfies
$\|\mathbf{A}-\mathbf{Q}\mathbf{Q}^\top\mathbf{A}\|_F\leq \epsilon\|\mathbf{A}\|_F$. Several adaptive randomized range finders for the fixed precision problem have been proposed in the literature \cite{che2019randomized,che2025efficient-siam,halko2011finding,hallman2022block,martinsson2016randomized,meier2024fast,yu2018efficient}. In detail, Halko {\it et al.} \cite{halko2011finding} used the incremental sampling approach with a probabilistic error estimator to determine the size of $\mathbf{Q}$. Martinsson {\it et al.} \cite{martinsson2016randomized} described a randomized blocked algorithm (called the blocked randQB algorithm) to return an approximate low rank factorization that is accurate to within $\epsilon$ in the Frobenius norm and the estimated $\epsilon$-rank that is very close to the theoretically optimal $\epsilon$-rank. Yu {\it et al.} \cite{yu2018efficient} proposed a mechanism for calculating the approximation error in Frobenius norm, which  enables efficient adaptive rank determination for large and/or sparse matrix, and then resulted the randQB\_EI algorithm by combining it with the blocked randQB algorithm. Hallman \cite{hallman2022block} presented a randomized algorithm (denoted by the randUBV algorithm) for matrix sketching based on the block Lanzcos bidiagonalization process. Following the work of \cite{halko2011finding}, Che and Wei \cite{che2019randomized} obtained another adaptive randomized algorithm for the fixed-precision problem of low rank approximations when replacing the standard Gaussian vectors by the Khatri-Rao product of several standard Gaussian vectors. Che {\it et al.} \cite{che2025efficient-siam} modified the blocked randQB algorithm by replacing the standard Gaussian matrices by the Khatri-Rao product of standard Gaussian matrices and/or uniform random matrices. Meier and Nakatsukasa \cite{meier2024fast} developed a randomized algorithm for estimating the $\epsilon$ rank of a matrix, which is based on sketching the matrix with random matrices from both left and right.
\begin{remark}
    Note that when obtaining the approximation $\widetilde{F}f(x_1,x_2)$ in (\ref{t-lasso:approximation-expression-general-main}), we need explicitly to form all the entries of the coefficient matrix $\mathbf{A}$ by using the existing algorithms to find the approximation $\widetilde{\mathbf{A}}$ and the estimated $\epsilon$-rank $R$. As we know, $(2I_1+1)(2I_2+1)$ multiple integrals are required to obtain all coefficients $\mathbf{A}(i_1,i_2)$, which is very impossible for large $I_n$'s. Hence, the main work in this paper is to design an effective method for calculating approximate matrix $\widetilde{\mathbf{A}}$ while reducing the number of double integrals.
\end{remark}

\section{The efficient method for approximating all coefficients}
\label{t-lasso:sec4-main}
More general, to ensure numerical stability, the matrix $\widetilde{\mathbf{A}}$ that satisfies the conditions and conclusions of Theorem \ref{t-lasso:epsilonrank-main} or Corollary \ref{t-lasso:epsilonrank-corollary-main} with $0<\epsilon<1$ can be also represented as $\widetilde{\mathbf{A}}=\mathbf{C}\cdot\mathbf{G}^\dag\cdot\mathbf{R}$, where $\mathbf{C}\in\mathbb{C}^{(2I_1+1)\times S_2}$, $\mathbf{R}\in\mathbb{C}^{S_1\times (2I_2+1)}$, the rank of $\mathbf{G}\in\mathbb{C}^{S_2\times S_1}$ is less than or equal to $R$ with $\min\{S_1,S_2\}> R$, and $\mathbf{G}^\dag$ is the Moore-Penrose inverse of $\mathbf{G}$. Under this case, the approximation $\widetilde{F}f(x_1,x_2)$ in (\ref{t-lasso:approximation-expression-general-main}) can be also represented as
\begin{equation}
    \widetilde{F}f(x_1,x_2)=\sum_{j_1=1}^{S_1}\sum_{j_2=1}^{S_2}(\mathbf{G}^\dag)(j_1,j_2)\phi_{j_1}^{(1)}(x_1)\phi_{j_2}^{(2)}(x_2),
    \label{t-lasso:approximation-expression-general-v2-main}
\end{equation}
where the functions $\phi_{j_1}^{(1)}(x_1)$ and $\phi_{j_2}^{(2)}(x_2)$ are given by
\begin{align*}
    \phi_{j_1}^{(1)}(x_1)&=\sum_{i_1=1}^{2I_1+1}\mathbf{C}(i_1,j_1)e^{\iota\cdot (i_1-I_1-1)x_1},\\
    \phi_{j_2}^{(2)}(x_2)&=\sum_{i_2=1}^{2I_2+1}\mathbf{R}(j_2,i_2)e^{\iota\cdot (i_2-I_2-1)x_2}.
\end{align*}

Hence, finding the tuple $\{\mathbf{C},\mathbf{G},\mathbf{R}\}$ and the pair $\{S_1,S_2\}$ from the coefficient matrix $\mathbf{A}$ is a key point to obtain the approximation $\widetilde{F}f(x_1,x_2)$ in (\ref{t-lasso:approximation-expression-general-v2-main}). To solve this problem, we will consider the following two cases: the case of $S_1$ and $S_2$ being given and the case of $S_1$ and $S_2$ being unknown.
\subsection{The case of $S_1$ and $S_2$ being given}
\label{T-Lasso-sec4:subsec1-main}
For each $n$, let $S_n$ be any positive integer with $S_n<2I_n+1$ and $\mathbb{T}_n=\{k_{n,1},k_{n,2},\dots,k_{n,S_n}\}$ be a subset of $\{-I_n,-I_{n}+1,\dots,I_{n}-1,I_n\}$ such that $-I_n\leq k_{n,1}<k_{n,2}<\dots<k_{n,S_n}\leq I_n$. For example, the set $\mathbb{T}_n$ can be sampled uniformly from $\{-I_n,-I_n+1,\dots,I_n,I_n\}$ with $S_n=O(R_n\log(2I_n+1))$ and $n=1,2$, where $R$ can be estimated according to Theorem \ref{t-lasso:epsilonrank-main} with $0<\epsilon<1$.

Define $\mathbf{C}\in\mathbb{C}^{(2I_1+1)\times S_2}$, $\mathbf{R}\in\mathbb{C}^{S_1\times (2I_2+1)}$, and $\mathbf{G}\in\mathbb{C}^{S_1\times S_2}$, respectively, as
\begin{align}
    \mathbf{C}(i_1,s_2)&=\left(\frac{1}{2\pi}\right)^2\int_{\mathbb{H}_2^0}f(u_1,u_2)e^{-\iota\cdot (k_1u_1+k_{2,s_2}u_2)}du_1 du_2,\label{t-lasso:method-one:sub1-main}\\
    \mathbf{R}(s_1,i_2)&=\left(\frac{1}{2\pi}\right)^2\int_{\mathbb{H}_2^0}f(u_1,u_2)e^{-\iota\cdot (k_{1,s_1}u_1+k_2u_2)}du_1 du_2,\label{t-lasso:method-one:sub2-main}\\
    \mathbf{G}(s_1,s_2)&=\left(\frac{1}{2\pi}\right)^2\int_{\mathbb{H}_2^0}f(u_1,u_2)e^{-\iota\cdot (k_{1,s_1}u_1+k_{2,s_2}u_2)}du_1 du_2,\label{t-lasso:method-one:sub3-main}
\end{align}
where $s_n=1,2,\dots,S_n$ and $i_n=k_n+I_n+1$ with $n=1,2$.

Note that the core matrix $\mathbf{G}\in\mathbb{C}^{S_1\times S_2}$ can be easily obtained from one of the pair $\{\mathbf{C},\mathbf{R}\}$. For example, we have $\mathbf{G}=\mathbf{C}(\mathbb{I}_1,:)=\mathbf{R}(:,\mathbb{I}_2)$ where
\begin{align*}
    \mathbb{I}_n&=\{k_{n,1}+I_n+1,k_{n,2}+I_n+1,\dots,k_{n,S_1}+I_n+1\}:=\mathbb{T}_n+I_n+1
\end{align*}
with $n=1,2$. Let $\mathbf{U}$ be the Moore-Penrose inverse of $\mathbf{G}$. Hence, we obtain an approximation to the matrix $\mathbf{A}$ as follows
\begin{equation}
\label{t-lasso:approximate-one-main}
    \widetilde{\mathbf{A}}
    =\mathbf{C}\cdot\mathbf{U}\cdot\mathbf{R},
\end{equation}
where the rank of $\widetilde{\mathbf{A}}$ is less than or equal to $\min\{S_1,S_2\}$.

We now consider the number of multiple integrals and the computational complexity to obtain the approximation $\widetilde{\mathbf{A}}$. In detail, one has
\begin{enumerate}
    \item[(a)] according to (\ref{t-lasso:method-one:sub1-main}) and (\ref{t-lasso:method-one:sub2-main}), it requires $(2I_1+1)S_2+(2I_2+1)S_1-S_1S_2$ multiple integrals to form the pair $\{\mathbf{C},\mathbf{R}\}$; and
    \item[(b)] to obtain the matrix $\mathbf{U}$ needs $C_{\rm {svd}}\cdot S_1S_2\min\{S_1,S_2\}$ operations.
\end{enumerate}

Similar to (\ref{t-lasso:approximate-one-main}), another approximation $\widehat{\mathbf{A}}$ to the matrix $\mathbf{A}$ can be obtained as follows:
\begin{equation}
\label{t-lasso:approximate-two-main}
    \widehat{\mathbf{A}}
    =\mathbf{C}\cdot\mathbf{C}(\mathbb{I}_1,:)^\dag\cdot\mathbf{G}\cdot\mathbf{R}(:,\mathbb{I}_2)^\dag\cdot\mathbf{R},
\end{equation}
with $\mathbb{I}_n=\mathbb{T}_n+I_n+1$ and $n=1,2$.

\begin{remark}
    Several comments for (\ref{t-lasso:approximate-one-main}) and (\ref{t-lasso:approximate-two-main}) are listed as follows:
    \begin{enumerate}
        \item[1)] the approximation given by (\ref{t-lasso:approximate-one-main}) and (\ref{t-lasso:approximate-two-main}) is efficient as we only require the matrices $\mathbf{C}$ and $\mathbf{R}$, and their overlapping entries, and need not even read the whole matrix $\mathbf{A}$;
        \item[2)] when $S_1\geq S_2$ and $\mathbf{C}$ is of full column rank, or $S_1<S_2$ and $\mathbf{R}$ is of full row rank, (\ref{t-lasso:approximate-one-main}) is the same as (\ref{t-lasso:approximate-two-main}); and
        \item[3)] the approximation in (\ref{t-lasso:approximate-one-main}) is called the CUR decomposition with cross approximation (see \cite{park2025accuracy}); we call the approximation in (\ref{t-lasso:approximate-two-main}) as the two-sided interpolative decomposition (see \cite{voronin2017efficient,xia2024making}).
    \end{enumerate}
\end{remark}
\begin{remark}
    When the entries of $\mathbf{A}$ are known in advance, the matrix $\mathbf{U}$ in (\ref{t-lasso:approximate-one-main}) can be chosen as $\mathbf{U}=\mathbf{C}^\dag\cdot\mathbf{A}\cdot\mathbf{R}^\dag$, which minimizes the Frobenius norm error $\|\mathbf{A}-\mathbf{C}\cdot\mathbf{U}\cdot\mathbf{R}\|_F$ with given the choice of $\mathbf{C}$ and $\mathbf{R}$. In this case, the approximation $\mathbf{C}\cdot\mathbf{U}\cdot\mathbf{R}$ is called the CUR decomposition with best approximation.
\end{remark}

According to different strategies for selecting the index sets $\mathbb{T}_1$ and $\mathbb{T}_2$, many scholars focus on the difference between $\mathbf{A}$ and $\widetilde{\mathbf{A}}:=\mathbf{C}\cdot\mathbf{U}\cdot\mathbf{R}$, which can be measured by $\|\widetilde{\mathbf{A}}-\mathbf{A}\|_F$, $\|\widetilde{\mathbf{A}}-\mathbf{A}\|_2$ or $\|\widetilde{\mathbf{A}}-\mathbf{A}\|_{\max}$. The common strategies are discussed in Appendix \ref{t-lasso:app1-main}.

As shown in the proof of Theorem \ref{t-lasso:general-approximation-theorem-main}, we now consider the upper bound for $\|\widetilde{\mathbf{A}}-\mathbf{A}\|_{\max}$ with several special choices for $\mathbb{T}_1$ and $\mathbb{T}_2$.
\begin{theorem}{{\bf (see \cite[Theorem 6]{osinsky2018pseudo})}}
   Let $\mathbf{A}\in\mathbb{C}^{(2I_1+1)\times (2I_2+1)}$ be obtained from (\ref{t-lasso:allcoefficients-main}). For two given index sets $\mathbb{T}_1$ and $\mathbb{T}_2$ with $|\mathbb{T}_1|=S_1$, $|\mathbb{T}_2|=S_2$ and $S_1\geq S_2$, let $\mathbf{C}=\mathbf{A}(:,\mathbb{I}_2)$ and $\mathbf{R}=\mathbf{A}(\mathbb{I}_1,:)$, and assume that $\mathbf{G}=\mathbf{A}(\mathbb{I}_1,\mathbb{I}_2)$ has the maximal volume among all $S_1\times S_2$ submatrices of $\mathbf{A}$ with $\mathbf{G}$ being full column rank with $\mathbb{I}_n=\mathbb{T}_n+I_n+1$ and $n=1,2$. Then,
   \begin{align*}
      \|\mathbf{C}\cdot\mathbf{U}\cdot\mathbf{R}-\mathbf{A}\|_{\max}\leq \sqrt{1+S_2}\sqrt{1+\frac{S_2}{S_1-S_2+1}}\sigma_{S_2+1}(\mathbf{A})
   \end{align*}
   with $\mathbf{U}=\mathbf{G}^\dag$.
   \label{t-lasso:cur-bound-main}
\end{theorem}

The following corollary is similar to Theorem \ref{t-lasso:general-approximation-theorem-main}, which illustrates the upper bound for $\|\widetilde{F}f(x_1,x_2)-f(x_1,x_2)\|_{L^2}$, where the index sets $\mathbb{T}_1$ and $\mathbb{T}_2$ satisfy that $\mathbf{G}=\mathbf{A}(\mathbb{T}_1+I_1+1,\mathbb{T}_2+I_2+1)$ has the maximal volume among all $S_1\times S_2$ submatrices of $\mathbf{A}$.
\begin{corollary}
   Adopt conditions in Theorem \ref{t-lasso:general-theorem-main} and Theorem \ref{t-lasso:cur-bound-main}, let $\widetilde{F}f(x_1,x_2)$ be defined via the CUR approximation
$\widetilde{\mathbf{A}} = \mathbf{C}\mathbf{G}^\dagger\mathbf{R}$ with $\mathbf{G}$ having maximal volume. We have
   \begin{align*}
        &\|\widetilde{F}f(x_1,x_2)-f(x_1,x_2)\|_{L^2}\leq C\max\{I_1,I_2\}^{-\alpha}|f(x_1,x_2)|_{\alpha}\\
        &\quad+
        \sqrt{(2I_1+1)(2I_2+1)} \sqrt{1+S_2}\sqrt{1+\frac{S_2}{S_1-S_2+1}}\sigma_{S_2+1}(\mathbf{A}),
    \end{align*}
    with $S_1\geq S_2$.
\end{corollary}

\subsection{The case of $S_1$ and $S_2$ being unknown}
For a given rank $R$, when the coefficient matrix $\mathbf{A}$ is easily obtained, there exist many strategies (see Appendix \ref{t-lasso:app1-main}) to select a pair $\{\mathbb{T}_1,\mathbb{T}_2\}$ such that $\widetilde{\mathbf{A}}$ in (\ref{t-lasso:approximate-one-main}) is an approximation of $\mathbf{A}$. However, in many practical applications, obtaining the matrix $\mathbf{A}$ through (\ref{t-lasso:allcoefficients-main}) can be prohibitive for large $I_n$'s and the desired rank $R$ is unknown in advance. Hence, without forming $\mathbf{A}$ explicitly, we now propose an adaptive strategy to obtain the pair $\{\mathbb{T}_1,\mathbb{T}_2\}$ from $f(x_1,x_2)$ and $\{e^{-\iota \cdot k_nx_n}:k_n=0,\pm1,\pm2,\dots,\pm I_n,n=1,2\}$ such that $\widetilde{F}f(x_1,x_2)$ in (\ref{t-lasso:approximation-expression-general-v2-main}) is the desired approximation to $f(x_1,x_2)$, where the pair $\{I_1,I_2\}$ is determined according to Theorem \ref{t-lasso:general-approximation-theorem-main}.

Note that there are three important parts $\mathbf{C}$, $\mathbf{U}$ and $\mathbf{R}$ in (\ref{t-lasso:approximate-one-main}), which are determined by the coefficient matrix $\mathbf{A}$ and the pair of two index sets $\{\mathbb{T}_1,\mathbb{T}_2\}$. The following two algorithms are derived by combining this adaptive strategy and the terminated condition determined by the pair $\{\mathbf{C},\mathbf{R}\}$, or the matrix $\mathbf{U}$.

\subsubsection{The first algorithm}
\label{t-lasso:sec4:sub2:sub1-main}
Suppose that two numbers of block sizes $b_1$ and $b_2$ satisfy $b_1\ll I_1$ and $b_2\ll I_2$. Let $\mathbb{T}_1=\mathbb{T}_2=[0]$. Then
\begin{align*}
   \mathbf{C}(i_1,1)&=\left(\frac{1}{2\pi}\right)^2\int_{\mathbb{H}_2^0}f(u_1,u_2)e^{-\iota\cdot (i_1-I_1-1)u_1}du_1 du_2,\\
   \mathbf{R}(1,i_2)&=\left(\frac{1}{2\pi}\right)^2\int_{\mathbb{H}_2^0}f(u_1,u_2)e^{-\iota\cdot (i_2-I_2-1)u_2}du_1 du_2,
\end{align*}
with $i_1=1,2,\dots,2I_1+1$ and $i_2=1,2,\dots,2I_2+1$. We initialize ${\rm nF}_{\mathbf{C}}=\|\mathbf{C}\|_F^2$ and ${\rm nF}_{\mathbf{R}}=\|\mathbf{R}\|_F^2$. For each iteration $k$, we complete the following operations:
\begin{enumerate}
   \item[(a)] we choose two index sets $\mathbb{T}_n'=[-kb_n:-(k-1)b_n-1]\cup[(k-1)b_n+1:kb_n]$\footnote{For given two integers $a$ and $b$ with $a<b$, we use $[a:b]$ to denote the set $\{a,a+1,\dots,b-1,b\}$.} with $n=1,2$;
   \item[(b)] we compute the matrix $\mathbf{C}_k$ according to (\ref{t-lasso:method-one:sub1-main}) with $\mathbb{T}_2'$, and $\mathbf{R}_k$ according to (\ref{t-lasso:method-one:sub2-main}) with $\mathbb{T}_1'$;
   \item[(c)] we update ${\rm nF}_{\mathbf{C}}={\rm nF}_{\mathbf{C}}+\|\mathbf{C}_k\|_F^2$ and ${\rm nF}_{\mathbf{R}}={\rm nF}_{\mathbf{R}}+\|\mathbf{R}_k\|_F^2$.
   \item[(d)] we update $\mathbf{C}=[\mathbf{C},\mathbf{C}_k]$, $\mathbf{R}=[\mathbf{R};\mathbf{R}_k]$ and $\mathbb{T}_n=\mathbb{T}_n\cup\mathbb{T}_n'$ with $n=1,2$;
   \item[(e)] if $\max\{\sigma_{\min}(\mathbf{C})^2/{\rm nF}_{\mathbf{C}},\sigma_{\min}(\mathbf{R})^2/{\rm nF}_{\mathbf{R}}\}< \tau^2$ is not satisfied, then we continue Steps (a)-(d), where $0<\tau<1$ is a given parameter.
\end{enumerate}

When $\max\{\sigma_{\min}(\mathbf{C})^2/{\rm nF}_{\mathbf{C}},\sigma_{\min}(\mathbf{R})^2/{\rm nF}_{\mathbf{R}}\}< \tau^2$ is satisfied, the core matrix $\mathbf{U}$ is obtained as $\mathbf{U}=\mathbf{G}^\dag$ with $\mathbf{G}=\mathbf{C}(\mathbb{T}_1+I_1+1,:)=\mathbf{R}(:,\mathbb{T}_2+I_2+1)$. The detailed process discussed above is summarized in Algorithm \ref{t-lasso:alg1-main}.

\begin{algorithm}[htb]
    \caption{An efficient algorithm for obtaining $\widetilde{F}f(x_1,x_2)$ from $f(x_1,x_2)$}
    \begin{algorithmic}[1]
        \STATEx {\bf Input}: A function $f(x_1,x_2)\in H^{\alpha}(\mathbb{H}_2^0)$, a given $0<\epsilon<1$, the pair of block sizes $(b_1,b_2)$, a tolerance $0<\tau<1$, and the maximum number of iterations $K$.
        \STATEx {\bf Output}: The 3-tuple $\{\mathbf{C},\mathbf{U},\mathbf{R}\}$, which is used to form the function $\widetilde{F}f(x_1,x_2)$.
        \STATE Initialize $k=0$ and ${\rm tol}=+\infty$.
        \STATE According to Theorem \ref{t-lasso:general-theorem-main} to estimate the pair $\{I_1,I_2\}$ such that $I_n=O((1/\epsilon)^{1/\alpha})$.
        \STATE Select $\mathbb{T}_n=\{0\}$ with $n=1,2$.
        \STATE Compute $\mathbf{C}$ according to (\ref{t-lasso:method-one:sub1-main}) with $\mathbb{T}_2$, and $\mathbf{R}$ according to (\ref{t-lasso:method-one:sub2-main}) with $\mathbb{T}_1$.
        \STATE Compute ${\rm nF}_{\mathbf{C}}=\|\mathbf{C}\|_F^2$, ${\rm nF}_{\mathbf{R}}=\|\mathbf{R}\|_F^2$.
        \WHILE{${\rm tol}>\tau$ or $k\leq K$}
           \STATE Select $\mathbb{T}_n'=[-kb_n:-(k-1)b_n-1]\cup[(k-1)b_n+1:kb_n]$ with $n=1,2$.
           \STATE Compute $\mathbf{C}_k$ according to (\ref{t-lasso:method-one:sub1-main}) with $\mathbb{T}_2'$, and $\mathbf{R}_k$ according to (\ref{t-lasso:method-one:sub2-main}) with $\mathbb{T}_1'$.
           \STATE Update $\mathbf{C}=[\mathbf{C},\mathbf{C}_k]$ and $\mathbf{R}=[\mathbf{R};\mathbf{R}_k]$.
           \STATE Compute ${\rm nF}_{\mathbf{C}}={\rm nF}_{\mathbf{C}}+\|\mathbf{C}_k\|_F^2$ and ${\rm nF}_{\mathbf{R}}={\rm nF}_{\mathbf{R}}+\|\mathbf{R}_k\|_F^2$.
           \STATE Obtain ${\rm tol}=\min\{\sigma_{\min}(\mathbf{C})/\sqrt{{\rm nF}_{\mathbf{C}}},\sigma_{\min}(\mathbf{R})/\sqrt{{\rm nF}_{\mathbf{R}}}\}$.
           \STATE Update $\mathbb{T}_n=\mathbb{T}_n\cup\mathbb{T}_n'$ and $k=k+1$.
        \ENDWHILE
        \STATE Compute $\mathbf{U}=\mathbf{G}^\dag$ $\mathbf{G}=\mathbf{C}(\mathbb{T}_1+I_1+1,:)=\mathbf{R}(:,\mathbb{T}_2+I_2+1)$.
        \STATE Return the 3-tuple $\{\mathbf{C},\mathbf{U},\mathbf{R}\}$.
    \end{algorithmic}
    \label{t-lasso:alg1-main}
\end{algorithm}

We now count the number of two-dimensional integrals and the computational complexity of Algorithm \ref{t-lasso:alg1-main}. For clarity, we assume that the maximum number of iterations for Algorithm \ref{t-lasso:alg1-main} is $K$. Then the approximation $\widetilde{F}f(x_1,x_2)$ is obtained according to (\ref{t-lasso:approximation-expression-general-v2-main}) with $S_n=2Kb_n+1$ with $n=1,2$.

First of all, to obtain $\mathbf{C}(:,1)$ and $\mathbf{R}(1,:)$ needs $(2I_1+1)+(2I_2+1)$ double integrals. Then, for each $k$, it requires $2b_2(2I_1+1)+2b_1(2I_2+1)$ double integrals to obtain all the matrices $\mathbf{C}_k$ and $\mathbf{R}_k$. Hence, the total number of double integrals required to obtain two matrices $\mathbf{C}$ and $\mathbf{R}$ is $(2I_1+1)S_2+(2I_2+1)S_1$. Based on the relationship between $\mathbb{T}_n$ and $\mathbb{T}_n'$, the number of double integrals to obtain the matrices $\mathbf{C}$ and $\mathbf{R}$ can be reduced to $(2I_1+1)S_2+(2I_2+1)S_1-S_1S_2$.

We now count the complexity of Algorithm \ref{t-lasso:alg1-main}: a) when $k=0$, to compute ${\rm nF}_{\mathbf{C}}$ and ${\rm nF}_{\mathbf{R}}$ needs $4(I_1+I_2)+2$ operations; b) for each $k\geq 1$, updating ${\rm nF}_{\mathbf{C}}$ and ${\rm nF}_{\mathbf{R}}$ requires $4b_2(2I_1+1)+4b_2(2I_2+1)$ operations; c) it costs $O((2I_1+1)(2kb_2+1)^2+(2I_2+1)(2kb_1+1)^2)$ operations to form $\sigma_{\min}(\mathbf{C})$ and $\sigma_{\min}(\mathbf{R})$; and d) to form the matrix $\mathbf{U}$ needs $O(S_1S_2\min\{S_1,S_2\})$ operations.

\begin{remark}
     We cannot ensure that $\mathbf{C}\in\mathbb{C}^{(2I_1+1)\times S_2}$ and $\mathbf{R}\in\mathbb{C}^{(2I_2+1)\times S_1}$ are of full column rank. The reason is that there exists an integer $k_*$ such that $\sigma_{\min}(\mathbf{C}_{k_*})$ and $\sigma_{\min}(\mathbf{R}_{k_*})$ may be zero.
\end{remark}
\subsubsection{The second algorithm}
Let $b_1$ and $b_2$ be the same as in Section \ref{t-lasso:sec4:sub2:sub1-main}. Let $\mathbb{T}_1=\mathbb{T}_2=[0]$, then one has
\begin{align*}
   \mathbf{G}(1,1)&=\left(\frac{1}{2\pi}\right)^2\int_{\mathbb{H}_2^0}f(u_1,u_2)du_1 du_2.
\end{align*}
We let ${\rm nF}_{\mathbf{G}}=\|\mathbf{G}\|_F^2$. For each iteration $k$, we complete the following operations:
\begin{enumerate}
   \item[(a)] we choose two index sets $\mathbb{T}_n'=[-kb_n:-(k-1)b_n-1]\cup[(k-1)b_n+1:kb_n]$ with $n=1,2$;
   \item[(b)] we compute the matrices $\mathbf{G}_{1k}$, $\mathbf{G}_{2k}$, and $\mathbf{G}_{3k}$ according to (\ref{t-lasso:method-one:sub3-main}) with $\{\mathbb{T}_1,\mathbb{T}_2'\}$, $\{\mathbb{T}_1',\mathbb{T}_2\}$, and $\{\mathbb{T}_1',\mathbb{T}_2'\}$, respectively;
   \item[(c)] we compute ${\rm nF}_{\mathbf{G}}={\rm nF}_{\mathbf{G}}+\|\mathbf{G}_{1k}\|_F^2+\|\mathbf{G}_{2k}\|_F^2+\|\mathbf{G}_{3k}\|_F^2$;
   \item[(d)] we update $\mathbf{G}$ as
      \begin{align*}
         \mathbf{G}=
         \begin{bmatrix}
            \mathbf{G} & \mathbf{G}_{1k}\\
            \mathbf{G}_{2k}& \mathbf{G}_{3k}
         \end{bmatrix},
      \end{align*}
      and $\mathbb{T}_n=\mathbb{T}_n\cup\mathbb{T}_n'$ with $n=1,2$;
   \item[(e)] we continue Steps (a)-(d) until the condition $\sigma_{\min}(\mathbf{G})^2<{\rm nF}_{\mathbf{G}}\cdot\tau^2$ is satisfied.
\end{enumerate}

When $\sigma_{\min}(\mathbf{G})^2< {\rm nF}_{\mathbf{G}}\cdot\tau^2$ is satisfied, we compute the matrix $\mathbf{C}$ according to (\ref{t-lasso:method-one:sub1-main}) with $\mathbb{T}_2$ and the matrix $\mathbf{R}$ according to (\ref{t-lasso:method-one:sub2-main}) with $\mathbb{T}_1$. Hence, the detailed process discussed above is summarized in Algorithm \ref{t-lasso:alg2-main}. The maximum number of iterations for Algorithm \ref{t-lasso:alg2-main} is also denoted by $K$. We also assume that $S_n=2Kb_n+1$ with $n=1,2$.
\begin{algorithm}[htb]
    \caption{Another efficient algorithm for obtaining $\widetilde{F}f(x_1,x_2)$ from $f(x_1,x_2)$}
    \begin{algorithmic}[1]
        \STATEx {\bf Input}: A function $f(x_1,x_2)\in H^{\alpha}(\mathbb{H}_2^0)$, a given $0<\epsilon<1$, the pair of block sizes $(b_1,b_2)$, a tolerance $0<\tau<1$ and the maximum number of iterations $K$.
        \STATEx {\bf Output}: The 3-tuple $\{\mathbf{C},\mathbf{U},\mathbf{R}\}$, which is used to form the function $\widetilde{F}f(x_1,x_2)$.
        \STATE Initialize $k=0$ and ${\rm tol}=+\infty$.
        \STATE According to Theorem \ref{t-lasso:general-theorem-main} to estimate the pair $\{I_1,I_2\}$ such that $I_n=O((1/\epsilon)^{1/\alpha})$.
        \STATE Select $\mathbb{T}_n=\{0\}$ and let $\mathbb{T}_n':=\mathbb{T}_n$ with $n=1,2$.
        \STATE Compute $\mathbf{G}$ according to (\ref{t-lasso:method-one:sub3-main}) with $\{\mathbb{T}_1',\mathbb{T}_2'\}$.
        \STATE Compute ${\rm nF}_{\mathbf{G}}=\|\mathbf{G}\|_F^2$.
        \WHILE{${\rm tol}>\tau$ or $k\leq K$}
           \STATE Select $\mathbb{T}_n'=[-kb_n:-(k-1)b_n-1]\cup[(k-1)b_n+1:kb_n]$ with $n=1,2$.
           \STATE Compute $\mathbf{G}_{1k}$, $\mathbf{G}_{2k}$, and $\mathbf{G}_{3k}$ according to (\ref{t-lasso:method-one:sub3-main}) with $\{\mathbb{T}_1,\mathbb{T}_2'\}$, $\{\mathbb{T}_1',\mathbb{T}_2\}$, and $\{\mathbb{T}_1',\mathbb{T}_2'\}$, respectively.
           \STATE Update ${\rm nF}_{\mathbf{G}}={\rm nF}_{\mathbf{G}}+\|\mathbf{G}_{1k}\|_F^2+\|\mathbf{G}_{2k}\|_F^2+\|\mathbf{G}_{3k}\|_F^2$ and
           \begin{align*}
              \mathbf{G}=
              \begin{bmatrix}
                 \mathbf{G} & \mathbf{G}_{1k}\\
                 \mathbf{G}_{2k}& \mathbf{G}_{3k}
              \end{bmatrix}.
           \end{align*}
           \STATE Obtain ${\rm tol}=\sigma_{\min}(\mathbf{G})/\sqrt{{\rm nF}_{\mathbf{G}}}$.
           \STATE Update $k=k+1$.
        \ENDWHILE
        \STATE Compute $\mathbf{U}=\mathbf{G}^\dag$.
        \STATE Compute $\mathbf{C}$ according to (\ref{t-lasso:method-one:sub1-main}) with $\mathbb{T}_2$, and $\mathbf{R}$ according to (\ref{t-lasso:method-one:sub2-main}) with $\mathbb{T}_1$.
        \STATE Return the 3-tuple $\{\mathbf{C},\mathbf{U},\mathbf{R}\}$.
    \end{algorithmic}
    \label{t-lasso:alg2-main}
\end{algorithm}

We now count the number of double integrals used in Algorithm \ref{t-lasso:alg2-main} as follows: a) to obtain $\mathbf{G}(1,1)$ needs one double integral; b) for each $k$, to generate $\mathbf{G}_{1k}$, $\mathbf{G}_{2k}$, and $\mathbf{G}_{3k}$ requires $4(k^2-1)b_1b_2+2(b_1+b_2)$ double integrals; and c) it needs $(2I_1+1-S_1)S_2+(2I_2+1-S_2)S_1$ double integrals to compute $\mathbf{C}$ and $\mathbf{R}$. Therefore, Algorithm \ref{t-lasso:alg2-main} needs $(2I_1+1)S_2+(2I_2+1)S_1-S_1S_2$ double integrals to obtain the 3-tuple $\{\mathbf{C},\mathbf{U},\mathbf{R}\}$, that is, the number of double integrals in Algorithms \ref{t-lasso:alg1-main} and \ref{t-lasso:alg2-main} are the same. Finally, we count the complexity of Algorithm \ref{t-lasso:alg2-main}: a) when $k=0$, it costs one operation to obtain ${\rm nF}_{\mathbf{G}}$; b) for each $k\geq 1$, to update ${\rm nF}_{\mathbf{G}}$ requires $8b_1b_2(2k-1)+4(b_1+b_2)$ operations; c) to compute $\sigma_{\min}(\mathbf{G})$ needs $O((2kb_1+1)(2kb_2+1)\min\{2kb_1+1,2kb_2+1\})$ operations; and d) to form the matrix $\mathbf{U}$ amends $O(S_1S_2\min\{S_1,S_2\})$ operations.

\section{Numerical examples}
\label{t-lasso:sec5-main}
In this section, we use the numerical computation software MATLAB R2024b to develop computer programs and implement the calculations on a desktop computer with an Intel Core i5-11300H (3.11 GHz) and 16.0GB (15.8GB usable) RAM. We set MATLAB maxNumCompThreads to 1 and use ``tic" and ``toc" to measure running time. The CPU time is measured in seconds. We run each algorithm 10 times and take the average result. We implement each method 10 times for each precision and take the average result. We have shared the MATLAB codes of the proposed algorithms on https://github.com/chncml/approximate-truncated-fourier-series for reproducibility.

We use the tensor-product quadrature method (cf. \cite{shin2017randomized,wu2017randomized}) to approximate the double integrals in (\ref{t-lasso:method-one:sub1-main}), (\ref{t-lasso:method-one:sub2-main}), and (\ref{t-lasso:method-one:sub3-main}). This method extends one-dimensional quadrature rules to two dimensions by combining nodes and weights from each dimension. In detail, as shown in \cite{trefethen2019approximation}, for any $(k_1,k_2)\in\mathbb{Z}_2$, any tensor-product quadrature method is given as
\begin{equation*}
   \begin{split}
      &\int_{\mathbb{H}_2^0}f(x_1,x_2)e^{-\iota\cdot (k_1x_1+k_2x_2)}dx_1dx_2\\
      &\quad\quad\approx \sum_{i_1=1}^{M_1}\sum_{i_2=1}^{M_2}w_1^{(i_1)}w_2^{(i_2)}
      f(x_1^{(i_1)},x_2^{(i_2)})e^{-\iota\cdot (k_1x_1^{(i_1)}+k_2x_2^{(i_2)})},
   \end{split}
\end{equation*}
where for $n=1,2$, $\{(x_n^{(i_n)},w_n^{(i_n)}):i_n=1,2,\dots,M_n\}$ is a set of one-dimensional quadrature points and weights such that
\begin{equation*}
   \int_{-\pi}^{\pi}h(x_n)dx_n\approx\sum_{i_n=1}^{M_n}w_n^{(i_n)}h(x_n^{(i_n)}):=I(h;x_n^{(i_n)},w_n^{(i_n)},M_n)
\end{equation*}
for any integrable function $h:[-\pi,\pi]\rightarrow \mathbb{R}$. Three common kinds for $\{(x_n^{(i_n)},w_n^{(i_n)}):i_n=1,2,\dots,M_n\}$ are chosen based on the Clenshaw-Curtis (CC) quadrature, the Gauss-Legendre (GL) quadrature, and the Newton-Cotes (NC) quadrature, which can be implemented by the functions ``chebpts'', ``legpts'' and ``trigpts'' in Chebfun\footnote{More details for Chebfun can be seen in https://www.chebfun.org/.}, respectively.

As shown in the tensor-product quadrature method, we need to compute or estimate $M_1M_2$ function values $\{f(x_1^{(i_1)},x_2^{(i_2)}):i_n=1,2,\dots,M_n,n=1,2\}$, which lead to a curse of dimensionality. There exist three methods to overcome this issue: (quasi-)Monte Carlo methods, sparse grids and the Bayesian quadrature. Hence, one of the future work is to implement Algorithms \ref{t-lasso:alg1-main} and \ref{t-lasso:alg2-main} by using these three algorithms to approximate (\ref{t-lasso:method-one:sub1-main}), (\ref{t-lasso:method-one:sub2-main}) and (\ref{t-lasso:method-one:sub3-main}). The choices of $M_1$ and $M_2$ are important factors in determining the quality of approximation $I(h;x_n^{(i_n)},w_n^{(i_n)},M_n)$. For clarity, we set $M_1=M_2=5001$. Another of the future work is to consider the choices of $M_1$ and $M_2$ by the adaptive quadrature strategy, which is a numerical integration method such that the integral of a function $h:[-\pi,\pi]\rightarrow \mathbb{R}$ is approximated using static quadrature rules on adaptively refined subintervals of the region of integration. Generally, adaptive algorithms are just as efficient and effective as traditional algorithms for ``well behaved'' integrands, but are also effective for ``badly behaved'' integrands for which traditional algorithms may fail.

For clarity, we set that $\alpha=2$, $C=1$, $\epsilon=1e-7$ and $|f(x_1,x_2)|_{\alpha}=1$. It follows from Theorem \ref{t-lasso:general-theorem-main} that $(I_1,I_2)=(3163,3163)$. For convenience, we use $\widetilde{F}_1f(x_1,x_2)$ and $\widetilde{F}_2f(x_1,x_2)$ to denote the approximation obtained from Algorithms \ref{t-lasso:alg1-main} and \ref{t-lasso:alg2-main}, respectively. Then the error function $e(x_1,x_2)$ is defined as
\begin{equation*}
   e(x_1,x_2)=|f(x_1,x_2)-g(x_1,x_2)|
\end{equation*}
with $g(x_1,x_2)\in\{Ff(x_1,x_2),\widetilde{F}_1f(x_1,x_2),\widetilde{F}_2f(x_1,x_2)\}$ and $-\pi\leq x_1,x_2\leq \pi$. To present the error function $e(x_1,x_2)$, we selected 3600 points $\mathbb{M}(x_1,x_2):=\{(x_1^{(i)},x_2^{(i)}):i=1,2,\dots,3600\}$ such that
\begin{align*}
    (x_1^{(1)},\dots,x_1^{(60)})&={\rm linspace}(-\pi,\pi,60),\ (x_2^{(1)},\dots,x_2^{(60)})={\rm linspace}(-\pi,\pi,60),
\end{align*}
where in MATLAB, ${\rm linspace}(-\pi,\pi,60)$ generates 60 points from $-\pi$ to $\pi$.
\subsection{Several test functions}
Three test functions are listed as follows:
\begin{equation*}
   \begin{split}
      f_1(x_1,x_2)&=\left(\frac{5^{3/4}15}{4\sqrt{3}}\right)^2\prod_{i=1}^2\max\{0,1/5-(x_i-1/2)^2\},\\
      f_2(x_1,x_2)&=(1 - x_1^2 - x_2^2) \exp(x_1 \cos(x_2)),\\
      f_3(x_1,x_2)&=\frac{1}{0.1+x_1^2+x_2^2}+\frac{1}{0.01+(x_1-0.5)^2+(x_2-0.5)^2},
   \end{split}
\end{equation*}
with $-\pi\leq x_1,x_2\leq \pi$. Note that the function $f_1(x_1,x_2)$ is the kink function (see \cite[(7.1)]{bartel2025minimal}), the function $f_2(x_1,x_2)$ is appeared in \cite{an2021lasso}, and the function $f_3(x_1,x_2)$ is in \cite{trefethen2017cubature}.

\begin{table}[htb]
   \scriptsize
   \centering
   \begin{tabular}{|c|ccc|ccc|ccc|}
      \hline
      \multirow{2}{*}{$(b_1,b_2)$} & \multicolumn{3}{|c}{$f_{1}(x_1,x_2)$} &  \multicolumn{3}{|c|}{$f_2(x_1,x_2)$} & \multicolumn{3}{|c|}{$f_3(x_1,x_2)$}  \\
      \cline{2-10}
      & CC & GL & NC & CC & GL & NC & CC & GL & NC     \\
      \hline
      (2,2)
      & 4.9316 & 4.3503 & 5.2247
      & 4.8111 & 4.6554 & 4.1436
      & 4.5723 & 4.2798 & 4.5591   \\
      \hline
      (4,4)
      & 4.2174 & 8.5491 & 6.0496
      & 8.3552 & 8.9892 & 4.5796
      & 5.6009 & 6.0802 & 4.2841   \\
      \hline
      (6,6)
      & 5.2091 & 4.1807 & 4.0461
      & 4.5531 & 9.7766 & 6.3879
      & 4.3129 & 6.6881 & 6.1283   \\
      \hline
      (8,8)
      & 5.1487 & 4.9226 & 8.2671
      & 4.6828 & 8.7716 & 6.3595
      & 4.5540 & 6.6212 & 6.4264   \\
      \hline
      (10,10)
      & 5.1267 & 5.2175 & 6.9408
      & 9.5202 & 8.7144 & 8.5821
      & 5.8749 & 6.7720 & 6.2940   \\
      \hline
      (12,12)
      & 5.0788 & 8.3341 & 9.7320
      & 7.5839 & 8.9114 & 9.3525
      & 4.5580 & 6.8577 & 6.1942   \\
      \hline
      (14,14)
      & 5.1961 & 5.0205 & 4.5079
      & 9.1719 & 4.8375 & 5.9143
      & 4.6249 & 6.6168 & 6.3820   \\
      \hline
      (16,16)
      & 5.1961 & 4.7132 & 5.0248
      & 8.7091 & 6.4546 & 5.9617
      & 6.1113 & 6.6876 & 4.7732   \\
      \hline
      (18,18)
      & 5.0524 & 5.2271 & 5.2940
      & 8.7059 & 6.3717 & 6.0448
      & 6.0593 & 6.7256 & 6.6795   \\
      \hline
      (20,20)
      & 5.2324 & 5.1505 & 5.3865
      & 9.0566 & 6.1544 & 9.1755
      & 6.3145 & 7.0426 & 6.3063   \\
      \hline
   \end{tabular}
   \caption{When setting \texorpdfstring{$\tau=1e-5$}{tau=1e-5}, with different pairs \texorpdfstring{$(b_1,b_2)$}{(b1,b2)}, the running time (seconds) obtained by applying Algorithm \ref{t-lasso:alg2-main} with three tensor-product quadratures to the test functions \texorpdfstring{$f_1(x_1,x_2)$}{f1(x1,x2)}, \texorpdfstring{$f_2(x_1,x_2)$}{f2(x1,x2)} and \texorpdfstring{$f_3(x_1,x_2)$}{f3(x1,x2)}.}
   \label{t-lasso:tab1-main}
\end{table}
\subsection{Choosing the parameters $(b_1,b_2)$ and $\tau$}
\label{t-lasso:sec5:sub2-main}
Note that the values of the pair $(b_1,b_2)$ and $\tau$ are also important factors in determining the quality of the approximation $\widetilde{F}_1f(x_1,x_2)$ or $\widetilde{F}_2f(x_1,x_2)$ to denote the approximation obtained from Algorithm \ref{t-lasso:alg1-main} or \ref{t-lasso:alg2-main}, respectively.

First of all, by setting $\tau=1e-5$, we now illustrate the comparison of the approximation error and running time of Algorithm \ref{t-lasso:alg2-main} with different values of $(b_1,b_2)$ via three test functions $f_1(x_1,x_2)$, $f_2(x_1,x_2)$ and $f_3(x_1,x_2)$. The related results are shown in Table \ref{t-lasso:tab1-main}, and Figs. \ref{t-lasso:fig1-main}, \ref{t-lasso:app:fig1-main} and \ref{t-lasso:app:fig2-main}. Hence, for Algorithm \ref{t-lasso:alg2-main} with CC, the values of $(b_1,b_2)$ associated to $f_1(x_1,x_2)$, $f_2(x_1,x_2)$ and $f_3(x_1,x_2)$ are (4,4), (6,6), and (4,4), respectively; for Algorithm \ref{t-lasso:alg2-main} with GL, the values of $(b_1,b_2)$ associated to $f_1(x_1,x_2)$, $f_2(x_1,x_2)$ and $f_3(x_1,x_2)$ are (6,6), (14,14), and (4,4), respectively; and for Algorithm \ref{t-lasso:alg2-main} with NC, the values of $(b_1,b_2)$ associated to $f_1(x_1,x_2)$, $f_2(x_1,x_2)$ and $f_3(x_1,x_2)$ are (6,6), (14,14), and (10,10), respectively.

\begin{figure}
    \setlength{\tabcolsep}{4pt}
    \renewcommand\arraystretch{1}
    \centering
    \subfloat[$f_1(x_1,x_2)$]{\includegraphics[width=0.8\linewidth]{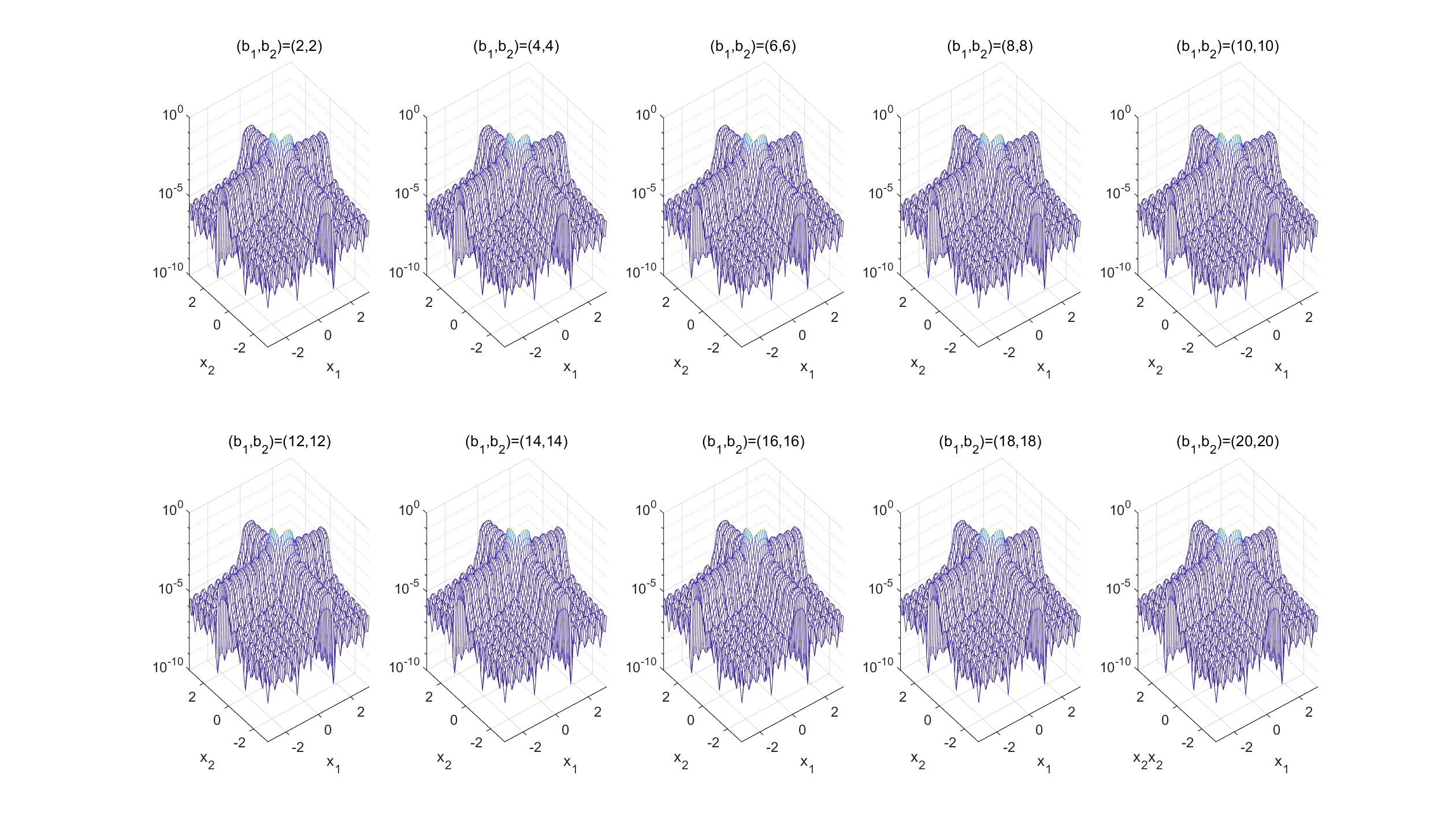}}\\
    \subfloat[$f_2(x_1,x_2)$]{\includegraphics[width=0.8\linewidth]{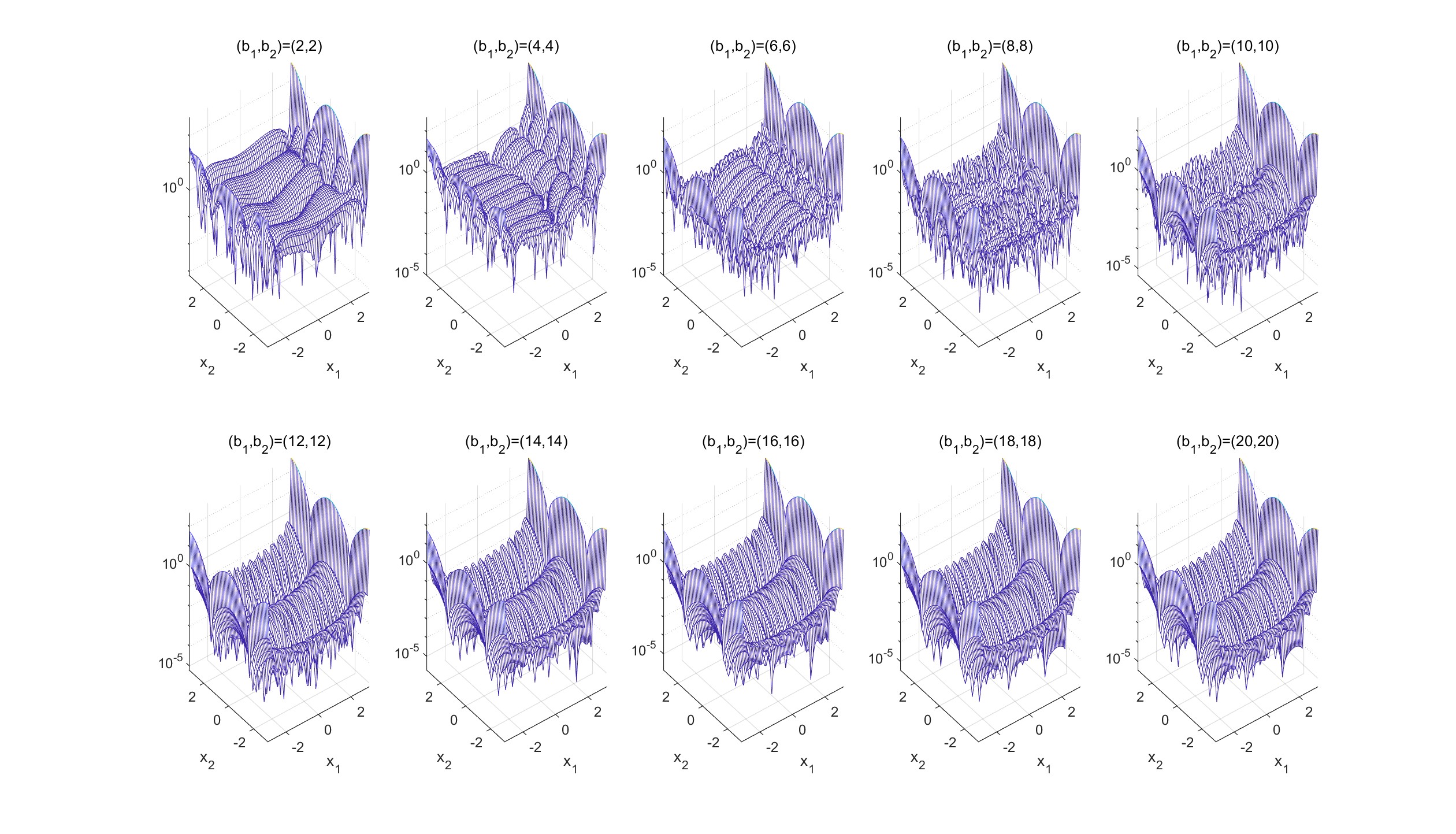}}\\
    \subfloat[$f_3(x_1,x_2)$]{\includegraphics[width=0.8\linewidth]{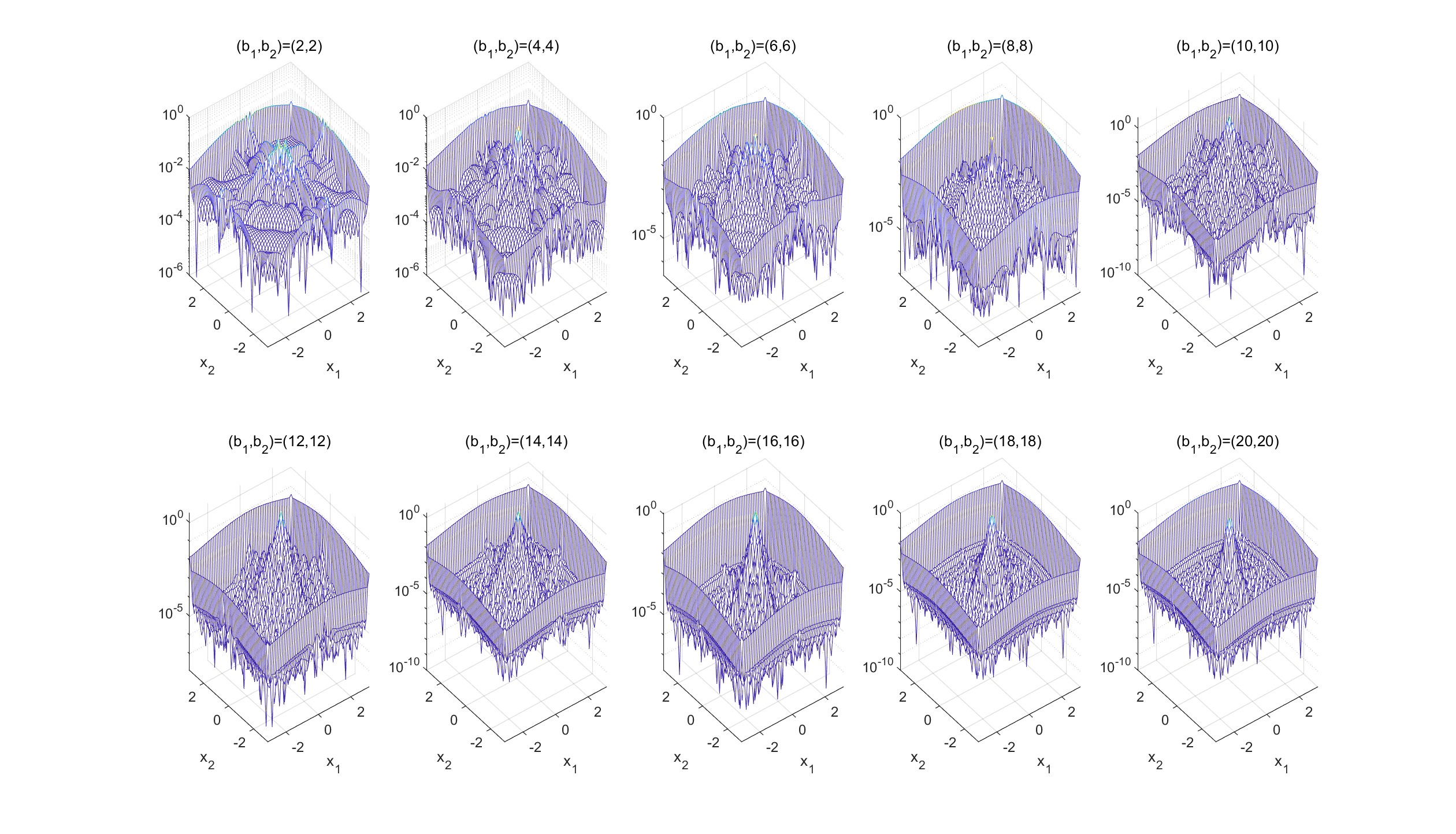}}\\
    \caption{When setting $\tau=1e-5$, with different pairs $(b_1,b_2)$, the values of $e(x_1,x_2)$ obtained by applying Algorithm \ref{t-lasso:alg2-main} with the tensor-product Newton-Cotes quadrature to the test functions $f_1(x_1,x_2)$, $f_2(x_1,x_2)$ and $f_3(x_1,x_2)$.}\label{t-lasso:fig1-main}
\end{figure}
\begin{figure}
    \setlength{\tabcolsep}{4pt}
    \renewcommand\arraystretch{1}
    \centering
    \subfloat[$f_1(x_1,x_2)$]{\includegraphics[width=0.8\linewidth]{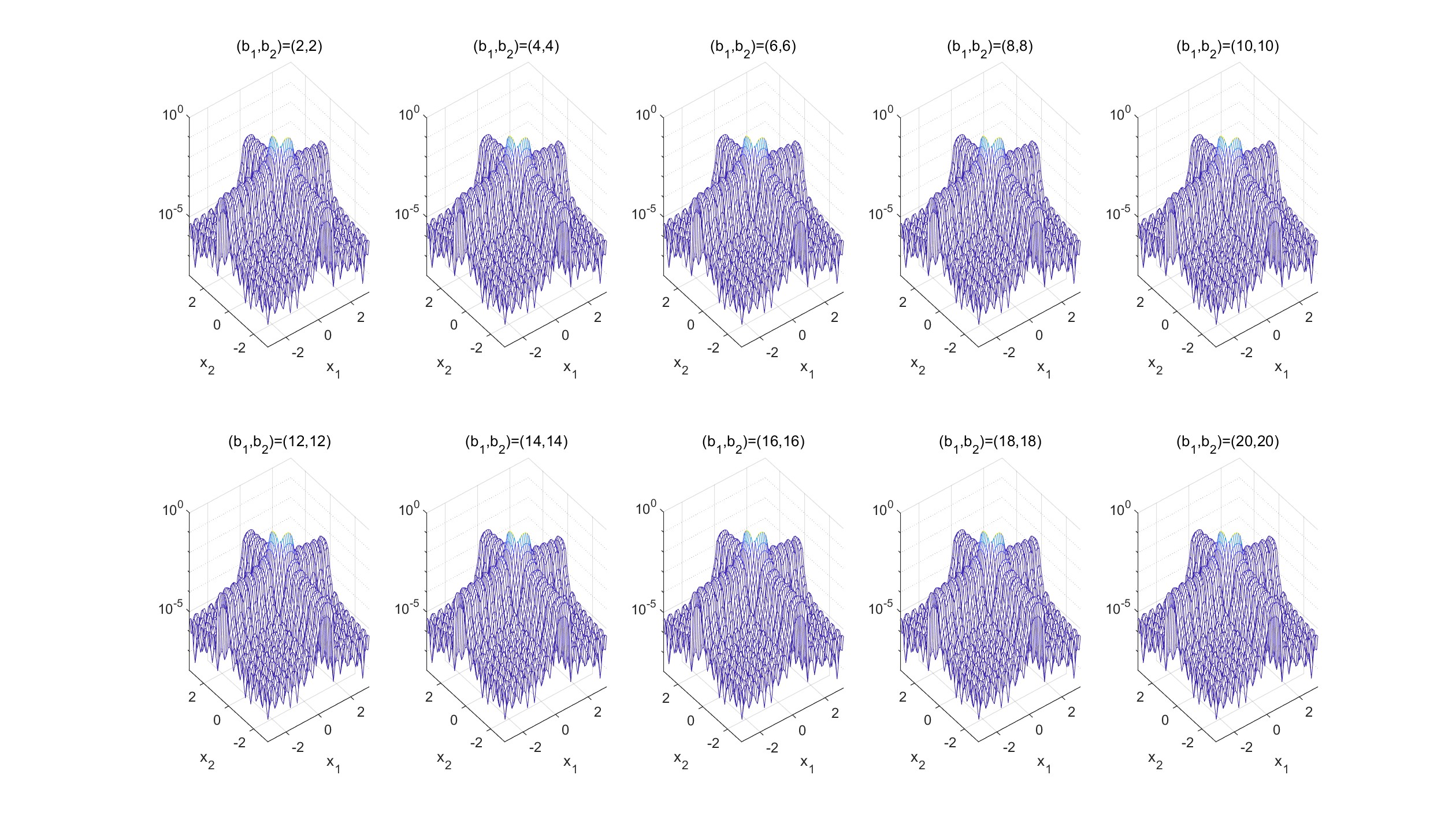}}\\
    \subfloat[$f_2(x_1,x_2)$]{\includegraphics[width=0.8\linewidth]{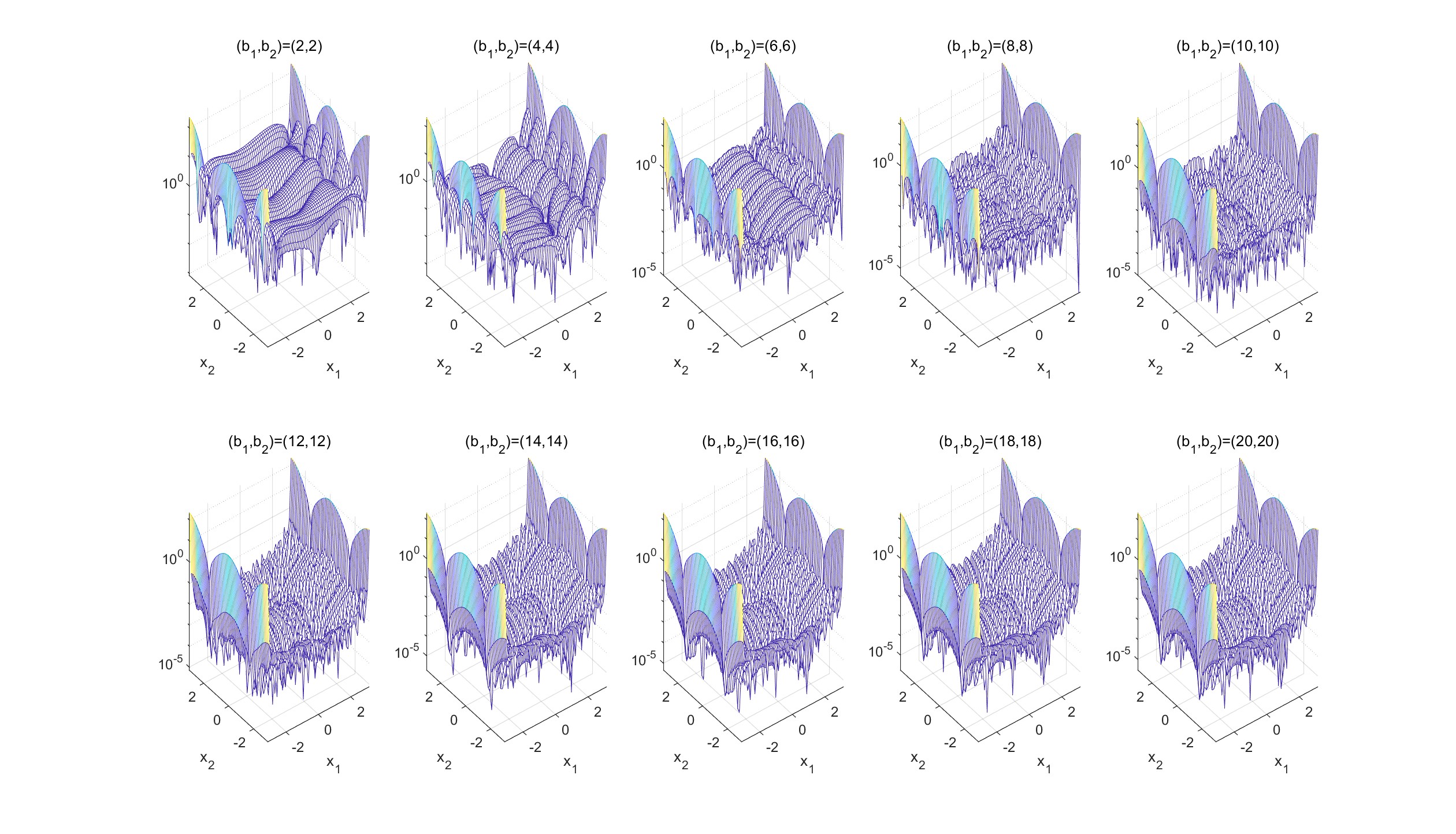}}\\
    \subfloat[$f_3(x_1,x_2)$]{\includegraphics[width=0.8\linewidth]{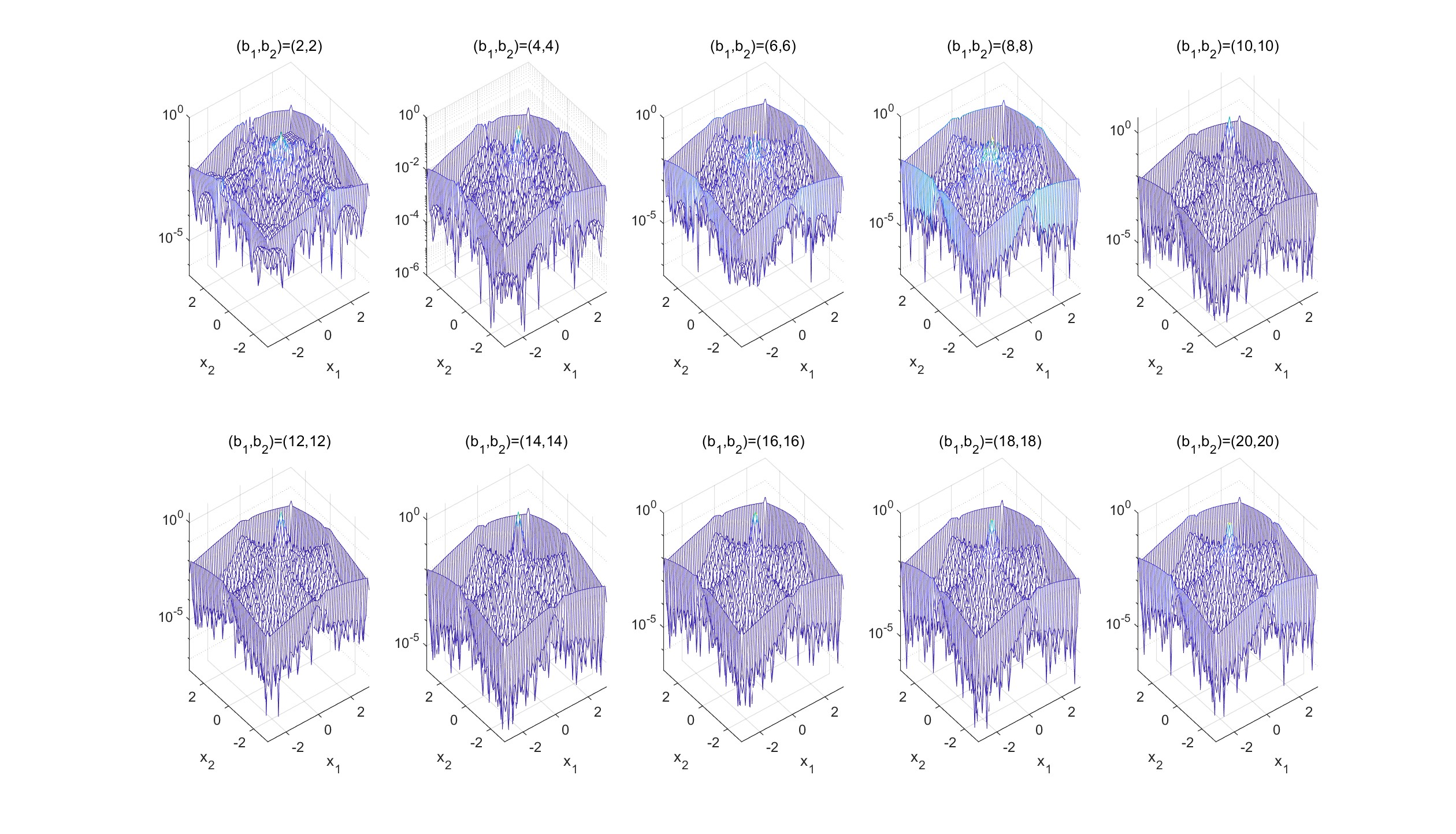}}\\
    \caption{When setting $\tau=1e-5$, with different pairs $(b_1,b_2)$, the values of $e(x_1,x_2)$ obtained by applying Algorithm \ref{t-lasso:alg2-main} with the tensor-product Clenshaw-Curtis quadrature to the test functions $f_1(x_1,x_2)$, $f_2(x_1,x_2)$ and $f_3(x_1,x_2)$.}\label{t-lasso:app:fig1-main}
\end{figure}
\begin{figure}
    \setlength{\tabcolsep}{4pt}
    \renewcommand\arraystretch{1}
    \centering
    \subfloat[$f_1(x_1,x_2)$]{\includegraphics[width=0.8\linewidth]{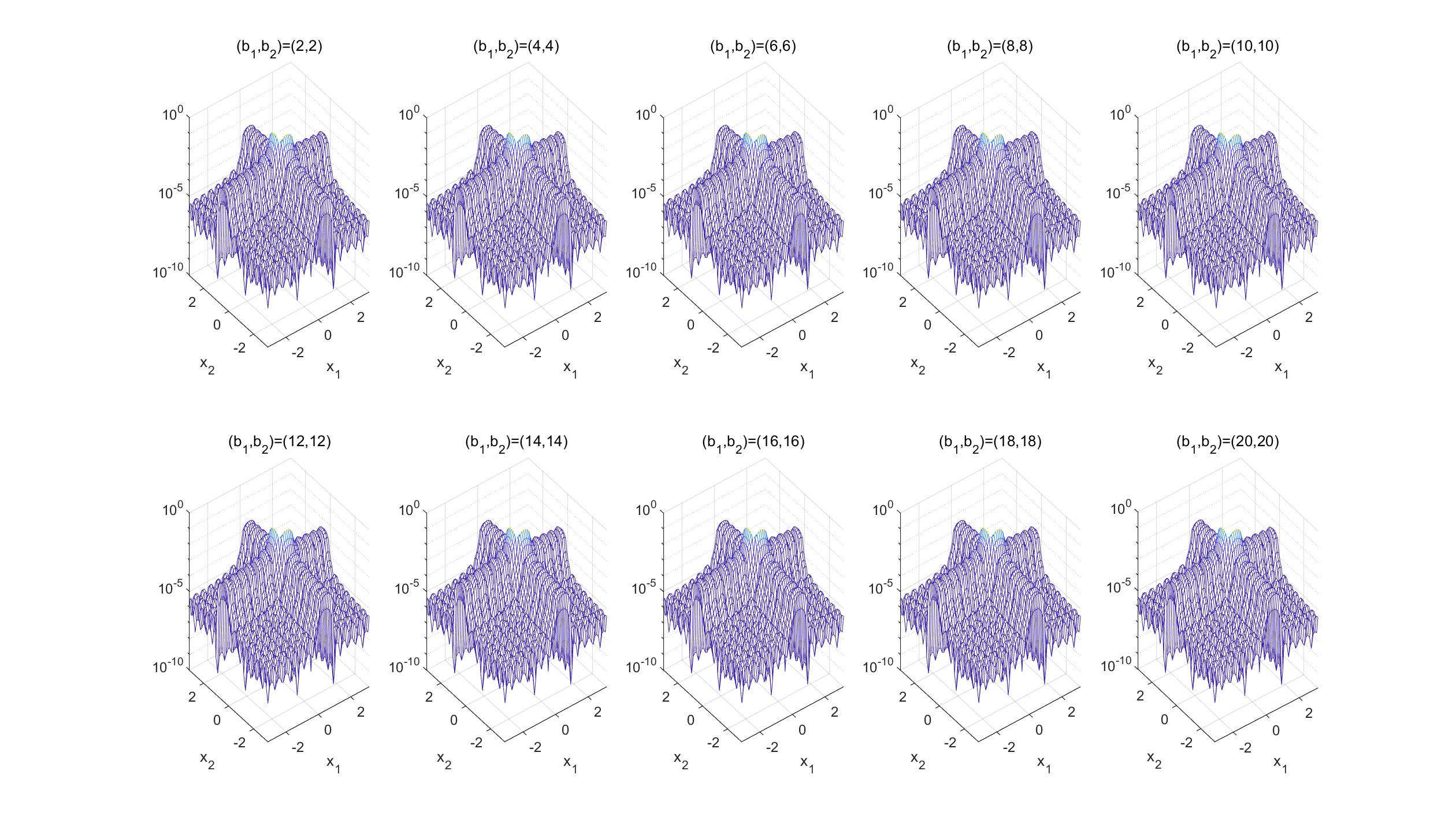}}\\
    \subfloat[$f_2(x_1,x_2)$]{\includegraphics[width=0.8\linewidth]{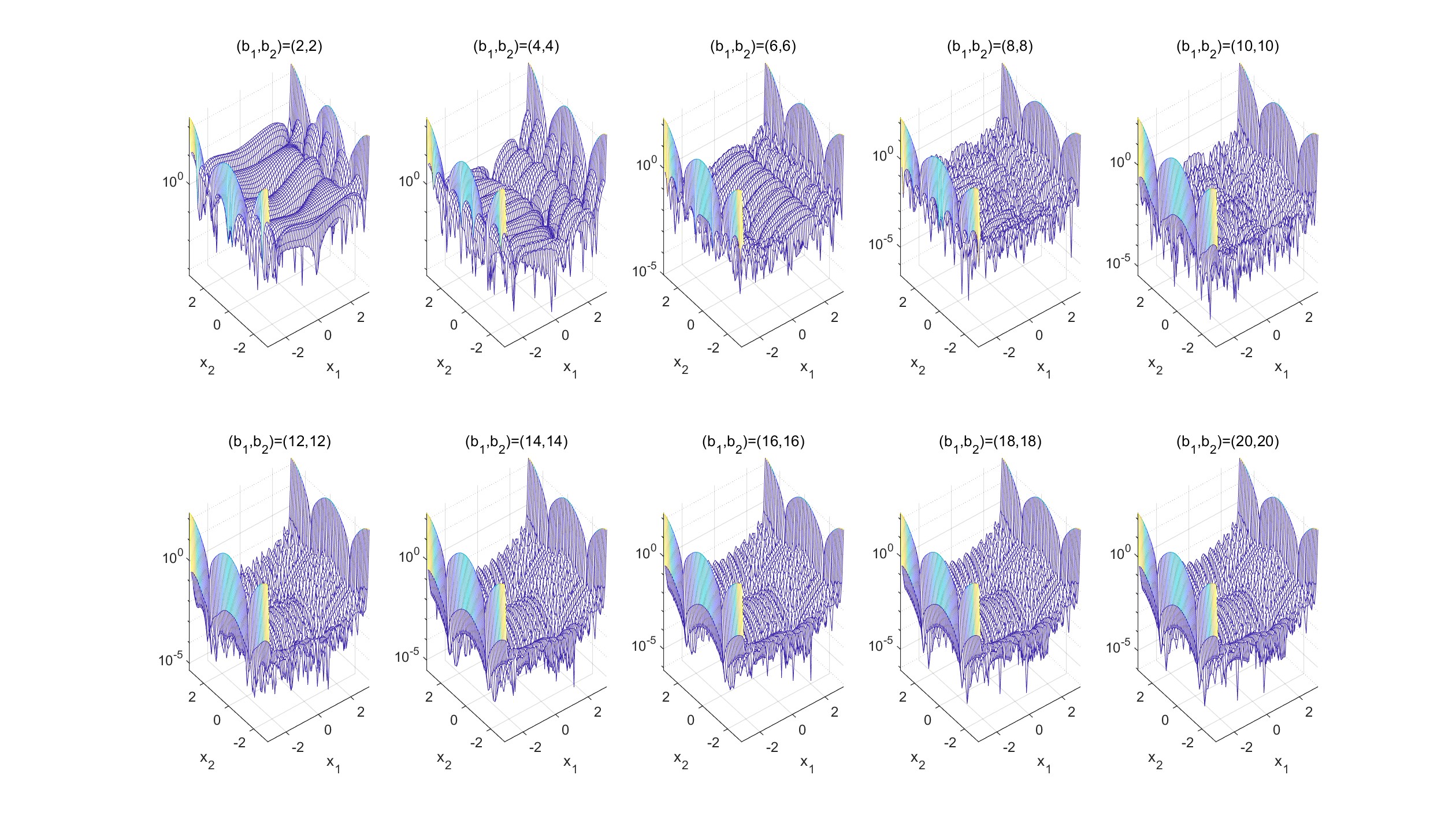}}\\
    \subfloat[$f_3(x_1,x_2)$]{\includegraphics[width=0.8\linewidth]{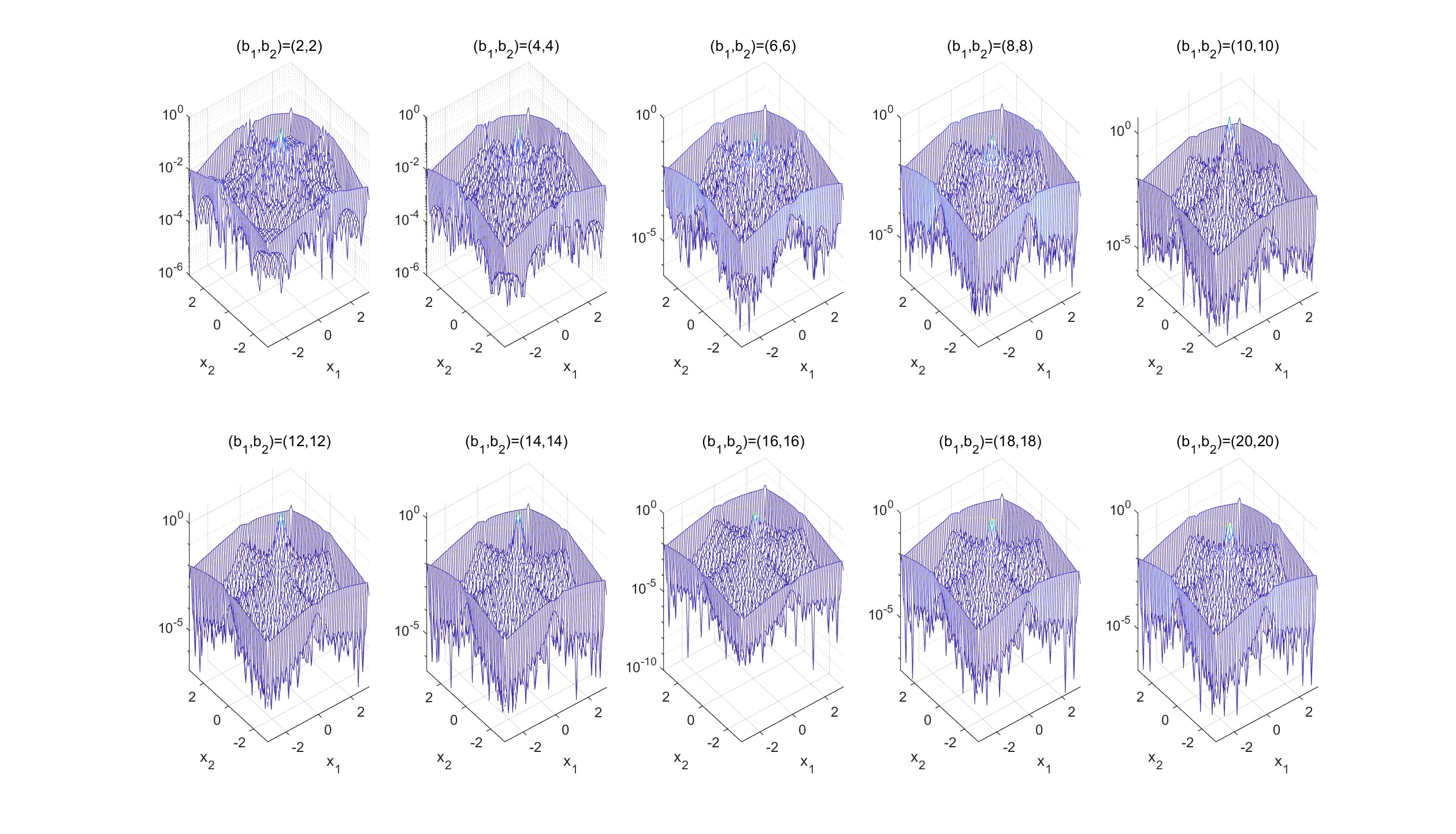}}\\
    \caption{When setting $\tau=1e-5$, with different pairs $(b_1,b_2)$, the values of $e(x_1,x_2)$ obtained by applying Algorithm \ref{t-lasso:alg2-main} with the tensor-product Gauss-Legendre quadrature to the test functions $f_1(x_1,x_2)$, $f_2(x_1,x_2)$ and $f_3(x_1,x_2)$.}\label{t-lasso:app:fig2-main}
\end{figure}

Finally, with the special choice of $(b_1,b_2)$ associated to each test function, we now consider the comparison of the approximation error and running time of Algorithm \ref{t-lasso:alg2-main} with different values of $\tau$ via these test functions. The related results are shown in Table \ref{t-lasso:tab2-main}, and Figs. \ref{t-lasso:fig2-main}, \ref{t-lasso:app:fig3-main} and \ref{t-lasso:app:fig4-main}, which illustrate that without loss of generality, we set $\tau=1e-5$ for each test function.

\begin{table}
   \scriptsize
   \centering
   \begin{tabular}{|c|ccc|ccc|ccc|}
      \hline
      \multirow{2}{*}{$\tau$} & \multicolumn{3}{|c}{$f_{1}(x_1,x_2)$} &  \multicolumn{3}{|c|}{$f_2(x_1,x_2)$} & \multicolumn{3}{|c|}{$f_3(x_1,x_2)$}  \\
      \cline{2-10}
      & CC & GL & NC & CC & GL & NC & CC & GL & NC     \\
      \hline
      $1e-1$
      & 4.6844 & 4.6604 & 5.1196
      & 6.8472 & 5.3382 & 4.3024
      & 4.6370 & 4.1311 & 4.2752   \\
      \hline
      $1e-2$
      & 4.0700 & 8.0897 & 6.5951
      & 4.5164 & 8.6188 & 4.3299
      & 4.3799 & 8.7755 & 4.6445   \\
      \hline
      $1e-3$
      & 6.4375 & 8.2904 & 5.7892
      & 6.2525 & 8.6726 & 6.1518
      & 6.0871 & 4.4096 & 7.1306   \\
      \hline
      $1e-4$
      & 6.0024 & 8.0956 & 6.1412
      & 8.2492 & 8.7472 & 6.4480
      & 6.0259 & 6.0995 & 7.2213   \\
      \hline
      $1e-5$
      & 6.2740 & 6.8364 & 6.2958
      & 9.1357 & 8.7507 & 6.4035
      & 6.5655 & 6.4456 & 7.3628   \\
      \hline
      $1e-6$
      & 6.2224 & 8.3803 & 6.2352
      & 4.6091 & 9.5111 & 6.3780
      & 6.5687 & 4.4670 & 7.0145   \\
      \hline
      $1e-7$
      & 6.1706 & 8.0074 & 6.3261
      & 4.6412 & 9.1666 & 6.5384
      & 6.5796 & 6.6366 & 6.7920   \\
      \hline
      $1e-8$
      & 6.1438 & 6.6473 & 6.3582
      & 6.1697 & 7.2456 & 5.9240
      & 6.1694 & 6.4278 & 7.4061   \\
      \hline
      $1e-9$
      & 6.2679 & 6.6865 & 6.1950
      & 8.1273 & 5.7725 & 6.4334
      & 6.4268 & 4.5961 & 7.0779   \\
      \hline
      $1e-10$
      & 6.4897 & 6.3935 & 6.2120
      & 8.1203 & 7.5289 & 6.4491
      & 7.5694 & 6.5324 & 7.6118   \\
      \hline
   \end{tabular}
   \caption{For the special choice of $(b_1,b_2)$, with different $\tau$, the running time (seconds) obtained by applying Algorithm \ref{t-lasso:alg2-main} with three tensor-product quadratures to the test functions $f_1(x_1,x_2)$, $f_2(x_1,x_2)$ and $f_3(x_1,x_2)$.}
\label{t-lasso:tab2-main}
\end{table}

\begin{figure}
    \setlength{\tabcolsep}{4pt}
    \renewcommand\arraystretch{1}
    \centering
    \subfloat[$f_1(x_1,x_2)$]{\includegraphics[width=0.8\linewidth]{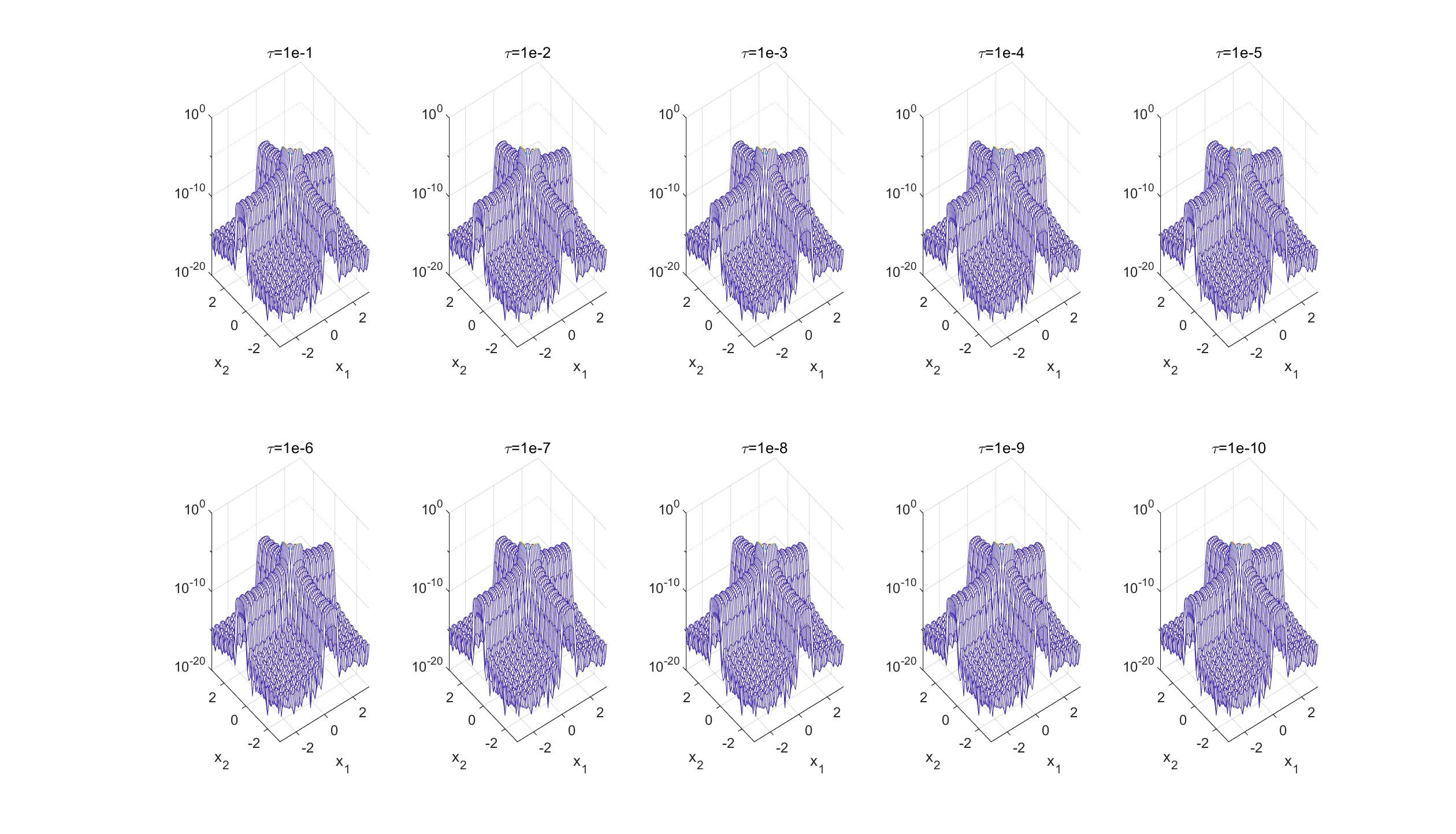}}\\
    \subfloat[$f_2(x_1,x_2)$]{\includegraphics[width=0.8\linewidth]{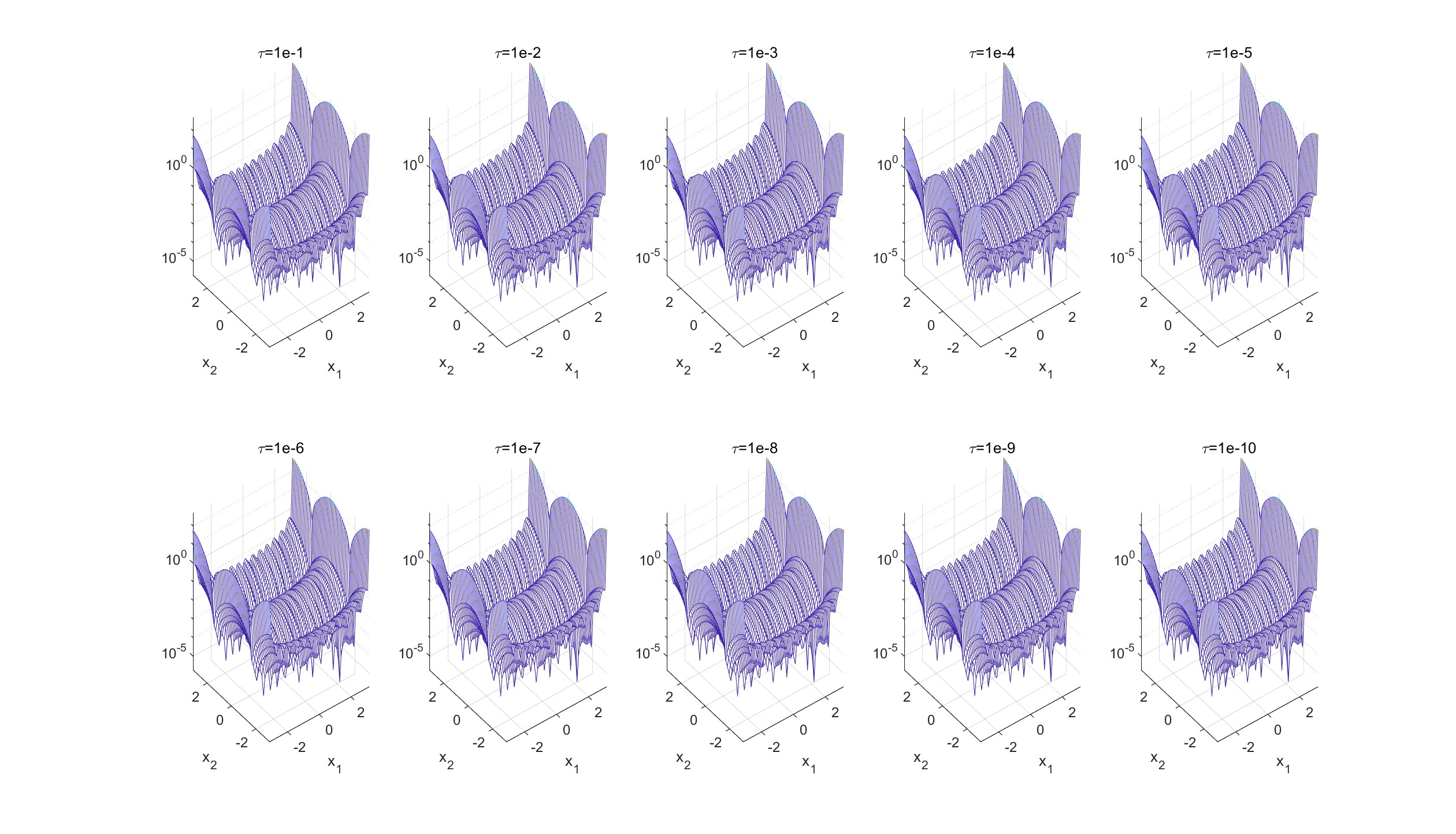}}\\
    \subfloat[$f_3(x_1,x_2)$]{\includegraphics[width=0.8\linewidth]{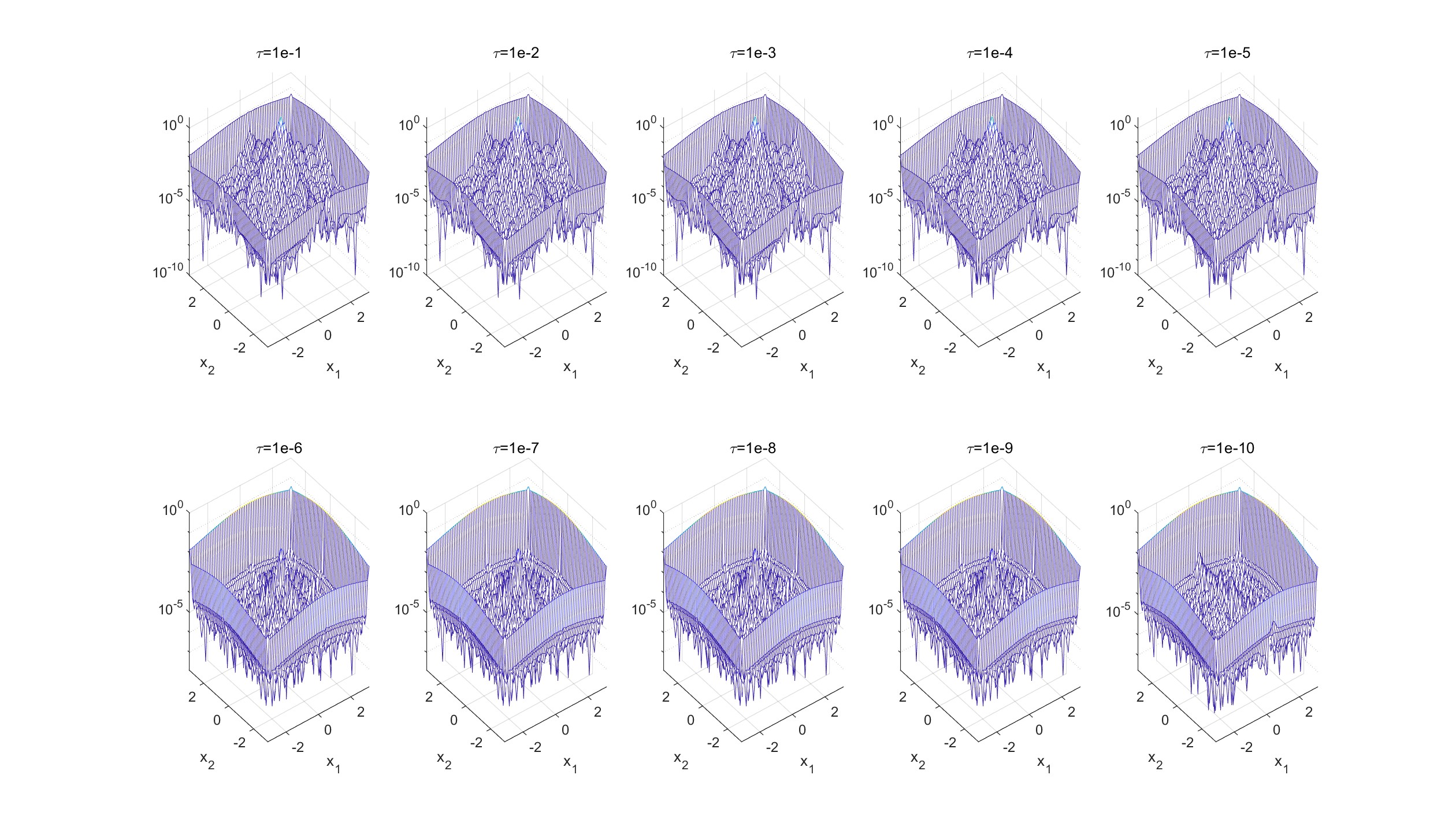}}\\
    \caption{For the special choice of $(b_1,b_2)$, with different $\tau$, the values of $e(x_1,x_2)$ obtained by applying Algorithm \ref{t-lasso:alg2-main} with the tensor-product Newton-Cotes quadrature to the test functions $f_1(x_1,x_2)$, $f_2(x_1,x_2)$ and $f_3(x_1,x_2)$.}\label{t-lasso:fig2-main}
\end{figure}

\begin{figure}
    \setlength{\tabcolsep}{4pt}
    \renewcommand\arraystretch{1}
    \centering
    \subfloat[$f_1(x_1,x_2)$]{\includegraphics[width=0.8\linewidth]{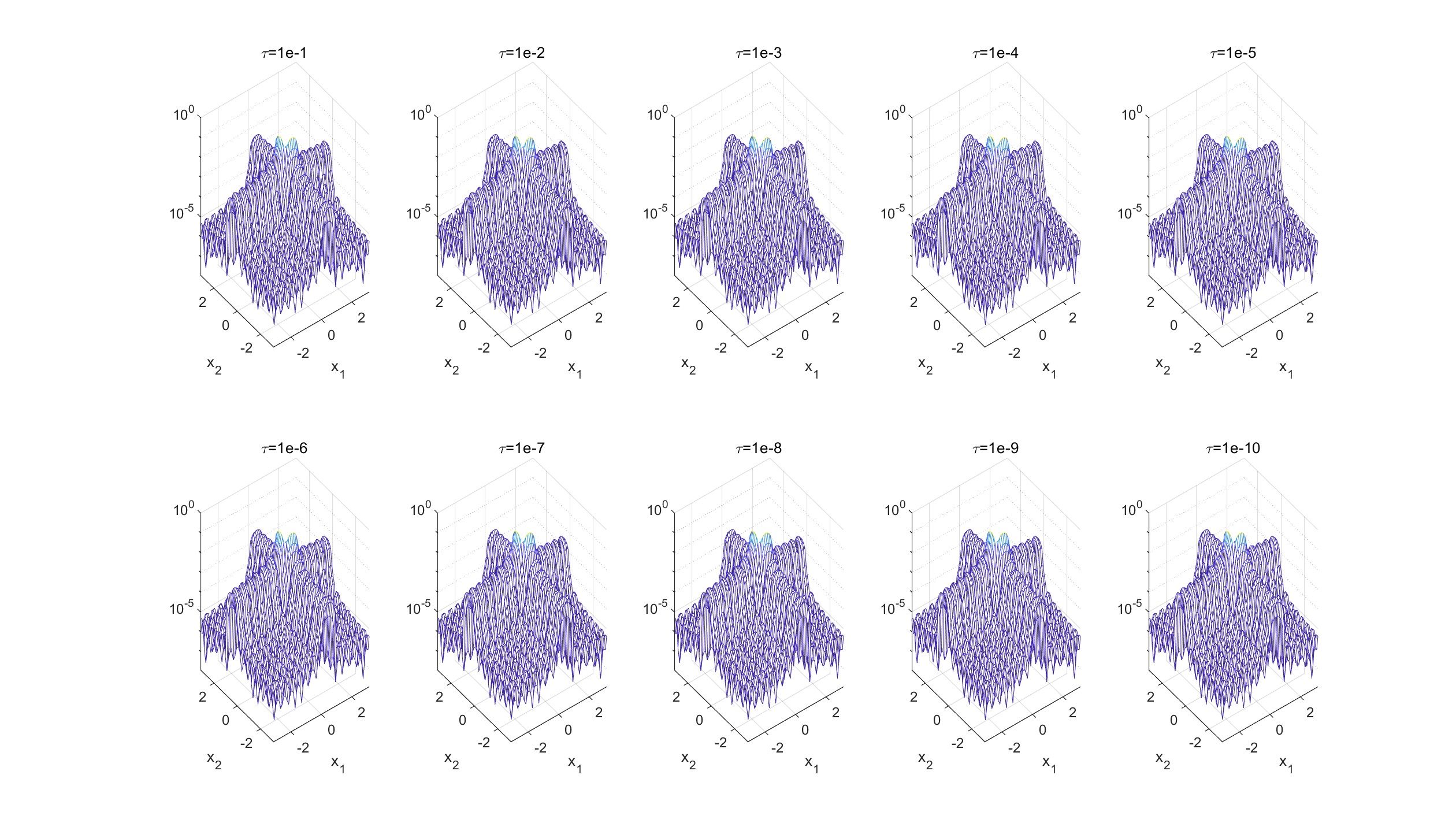}}\\
    \subfloat[$f_2(x_1,x_2)$]{\includegraphics[width=0.8\linewidth]{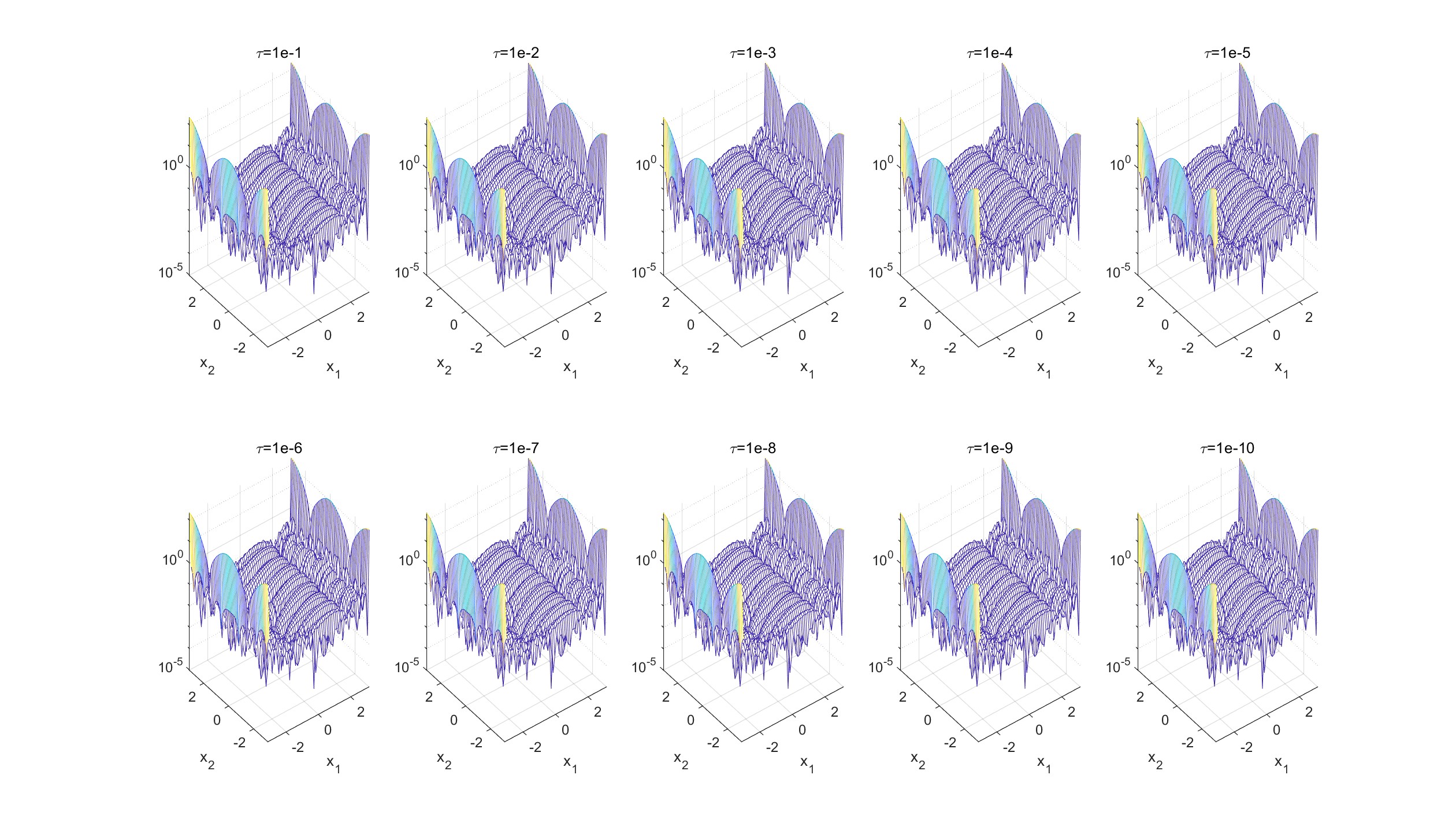}}\\
    \subfloat[$f_3(x_1,x_2)$ ]{\includegraphics[width=0.8\linewidth]{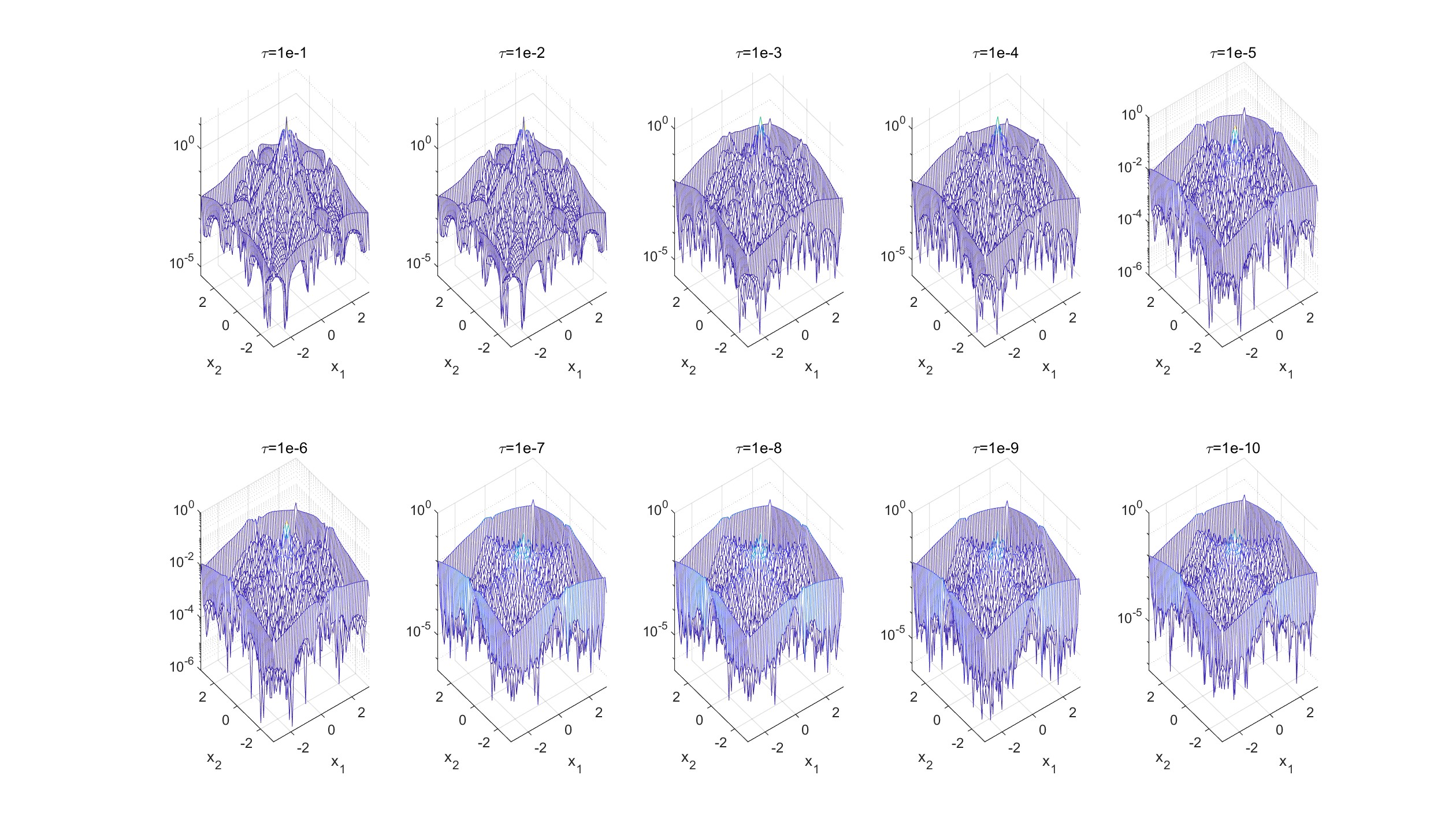}}\\
    \caption{For the special choice of $(b_1,b_2)$, with different $\tau$, the values of $e(x_1,x_2)$ obtained by applying Algorithm \ref{t-lasso:alg2-main} with the tensor-product Clenshaw-Curtis quadrature to the test functions $f_1(x_1,x_2)$, $f_2(x_1,x_2)$ and $f_3(x_1,x_2)$.}\label{t-lasso:app:fig3-main}
\end{figure}

\begin{figure}
    \setlength{\tabcolsep}{4pt}
    \renewcommand\arraystretch{1}
    \centering
    \subfloat[$f_1(x_1,x_2)$]{\includegraphics[width=0.8\linewidth]{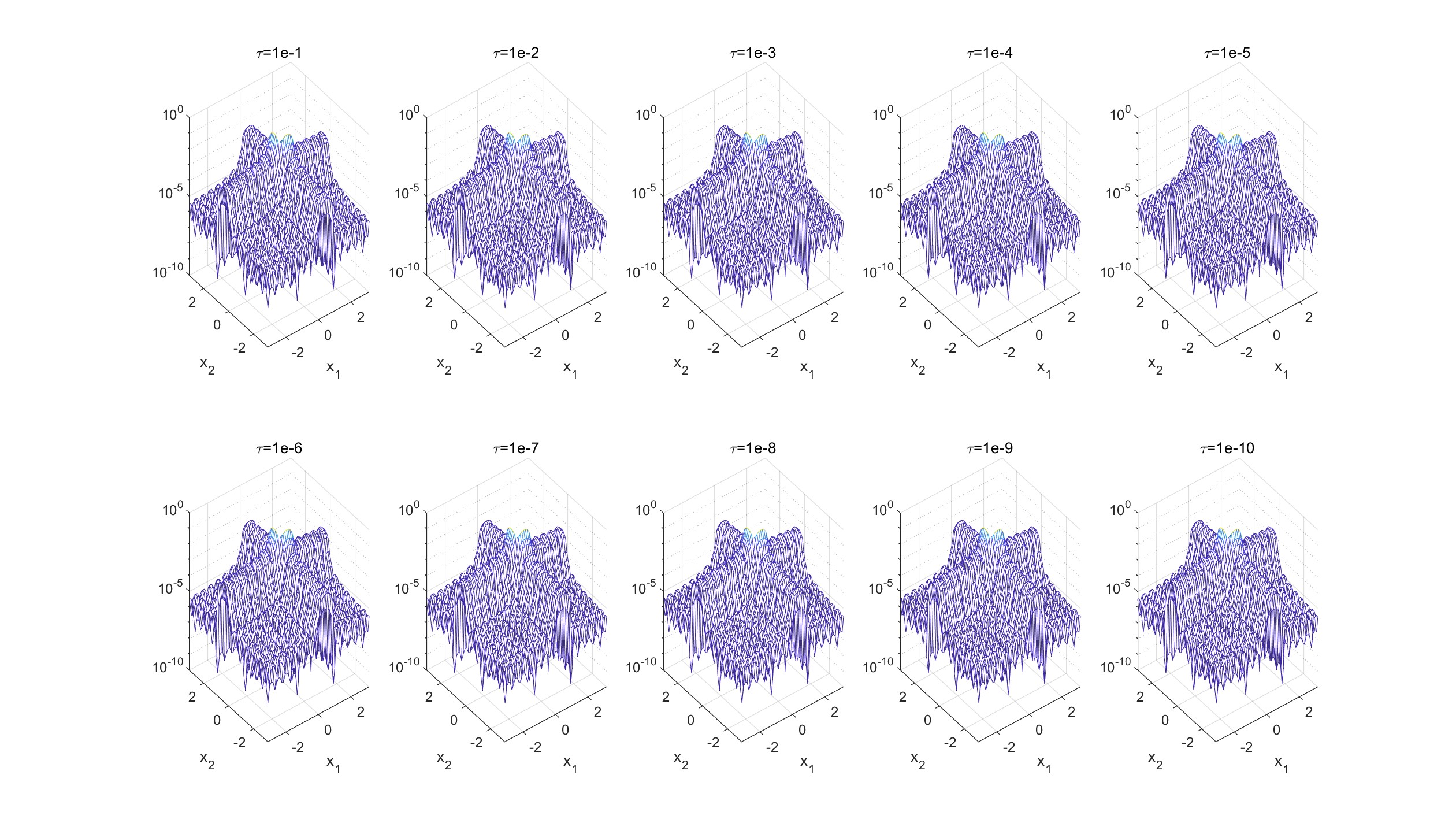}}\\
    \subfloat[$f_2(x_1,x_2)$]{\includegraphics[width=0.8\linewidth]{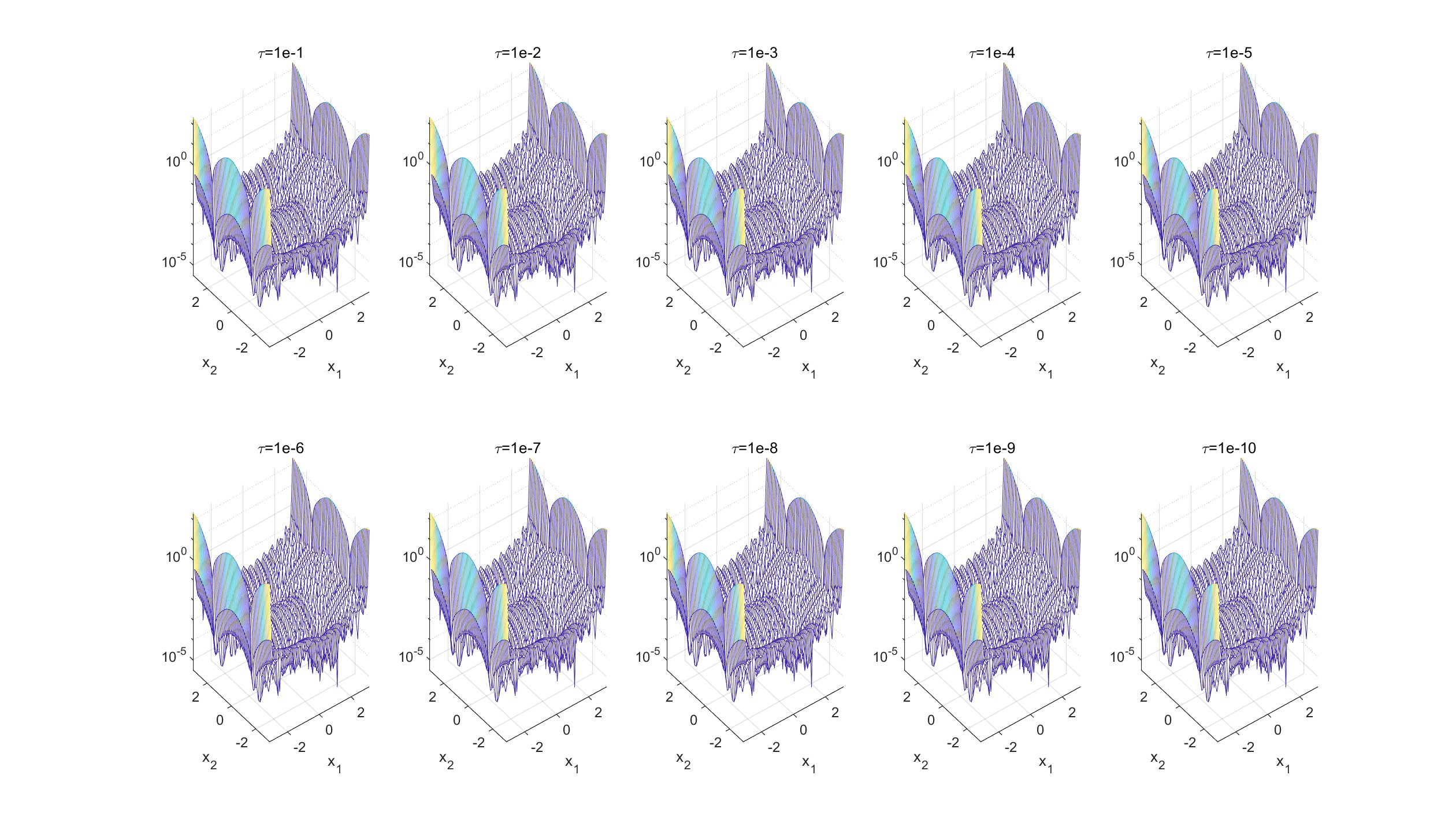}}\\
    \subfloat[$f_3(x_1,x_2)$]{\includegraphics[width=0.8\linewidth]{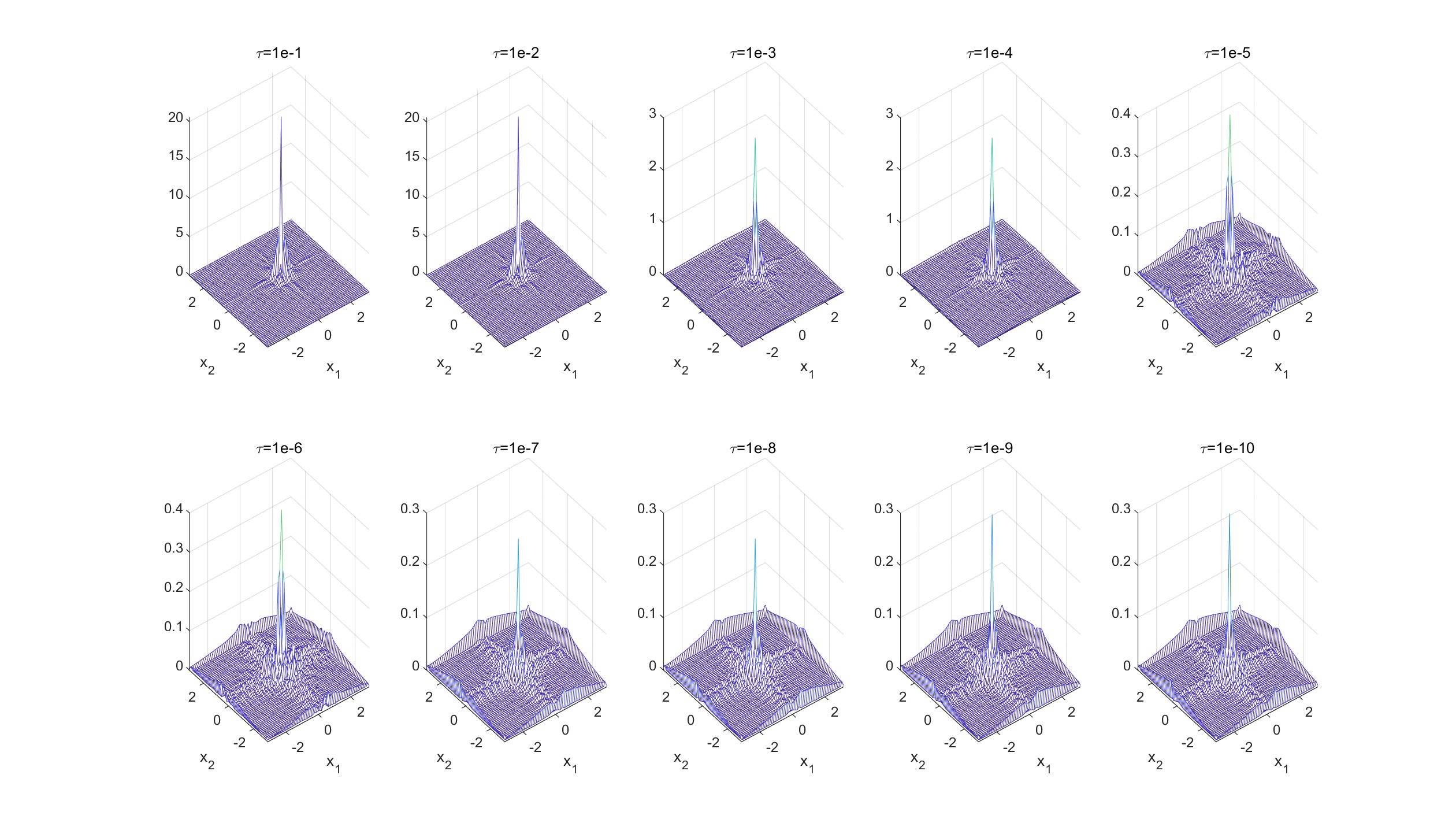}}\\
    \caption{For the special choice of $(b_1,b_2)$, with different $\tau$, the values of $e(x_1,x_2)$ obtained by applying Algorithm \ref{t-lasso:alg2-main} with the tensor-product Gauss-Legendre quadrature to the test functions $f_1(x_1,x_2)$, $f_2(x_1,x_2)$ and $f_3(x_1,x_2)$.}\label{t-lasso:app:fig4-main}
\end{figure}

\begin{remark}
    In this section, we only consider the choices of $(b_1,b_2)$ and $\tau$ in Algorithm \ref{t-lasso:alg2-main}. Similarly, we can also compare the efficiency of Algorithm \ref{t-lasso:alg1-main} with different values of $(b_1,b_2)$ and $\tau$. For clarity, we assume that the values of $(b_1,b_2)$ and $\tau$ in Algorithm \ref{t-lasso:alg1-main} are the same as that in Algorithm \ref{t-lasso:alg2-main}, respectively.
\end{remark}
\subsection{Comparison results}

As discussion in Section \ref{t-lasso:sec5:sub2-main}, we discussed the reasonable choices for $(b_1,b_2)$ and $\tau$ in Algorithms \ref{t-lasso:alg1-main} and \ref{t-lasso:alg2-main} when we apply these two algorithms to the test functions $f_1(x_1,x_2)$, $f_2(x_1,x_2)$ and $f_3(x_1,x_2)$. For a given pair $(I_1,I_2)$ with positive integers, $Ff(x_1,x_2)$ in (\ref{t-lasso:approximation-expression-general-v2-main}) is the truncated Fourier series expression associated with $f(x_1,x_2)$, which is denoted by Truncated Fourier.

We now compare the efficiency of Truncated Fourier, and Algorithms \ref{t-lasso:alg1-main} and \ref{t-lasso:alg2-main} with three tensor-product quadratures via three functions $f_1(x_1,x_2)$, $f_2(x_1,x_2)$ and $f_3(x_1,x_2)$. The values of each original function $f(x_1,x_2)$, and three approximations $Ff(x_1,x_2)$,  $\widetilde{F}_1f(x_1,x_2)$ and $\widetilde{F}_2f(x_1,x_2)$ are shown in Figs. \ref{t-lasso:fig3-main} (for NC), \ref{t-lasso:app:fig5-main} (for CC) and \ref{t-lasso:app:fig6-main} (for GL). Meanwhile, the corresponding running times are listed in Table \ref{t-lasso:tab3-main}. Hence, we conclude that (a) for each tensor-product quadrature, the efficiency of Algorithms \ref{t-lasso:alg1-main} and \ref{t-lasso:alg2-main} is comparable to the truncated Fourier series expression; (b) for $f_1(x_1,x_2)$ and $f_2(x_1,x_2)$, Algorithms \ref{t-lasso:alg1-main} and \ref{t-lasso:alg2-main} are faster than Truncated Fourier with the same tensor-product quadrature; and (c) for $f_3(x_1,x_2)$, Algorithm \ref{t-lasso:alg2-main} is faster than Truncated Fourier and Algorithm \ref{t-lasso:alg1-main} is slower than Truncated Fourier with the same tensor-product quadrature. Note that as shown in Table \ref{t-lasso:fig3-main}, Algorithm \ref{t-lasso:alg2-main} with NC is not suitable for the function $f_3(x_1,x_2)$.

\begin{remark}
    For each part in Figs. \ref{t-lasso:fig3-main}, \ref{t-lasso:app:fig5-main} and \ref{t-lasso:app:fig6-main}, the values of $f(x_1,x_2)$, $F(x_1,x_2)$, $\widetilde{F}f(x_1,x_2)$ and $\widetilde{F}f(x_1,x_2)$ evaluated at the points in $\mathbb{M}(x_1,x_2)$ are shown in the first row, and the values of $e(x_1,x_2)$ associated with $F(x_1,x_2)$, $\widetilde{F}f(x_1,x_2)$ and $\widetilde{F}f(x_1,x_2)$ are shown in the second row.
\end{remark}

\begin{table}[htb]
   \scriptsize
   \centering
   \begin{tabular}{|c|c|c|c|c|}
      \hline
      Types  & Algorithms & $f_1(x_1,x_2)$ & $f_2(x_1,x_2)$ & $f_3(x_1,x_2)$ \\
      \hline
      \multirow{3}{*}{CC} & Truncated Fourier          & 21.8316 &   12.7248  &  12.1434  \\
      \cline{2-5}
      \multirow{3}{*}{} & Algorithm \ref{t-lasso:alg1-main} & 6.6532 & 12.0676 & 14.7140  \\
      \cline{2-5}
      \multirow{3}{*}{} & Algorithm \ref{t-lasso:alg2-main} & 4.0010 & 6.2976 & 6.4122  \\
      \hline
      \multirow{3}{*}{GL} & Truncated Fourier          & 11.4861 & 18.1679 & 12.4595  \\
      \cline{2-5}
      \multirow{3}{*}{} & Algorithm \ref{t-lasso:alg1-main} & 7.8366 & 9.9590 & 18.0444  \\
      \cline{2-5}
      \multirow{3}{*}{} & Algorithm \ref{t-lasso:alg2-main} & 7.9238 & 7.5091 & 6.2866  \\
      \hline
      \multirow{3}{*}{NC} & Truncated Fourier          & 20.7400 & 16.4348 & 13.0728  \\
      \cline{2-5}
      \multirow{3}{*}{} & Algorithm \ref{t-lasso:alg1-main} & 5.6615 & 8.4436 & 16.5136  \\
      \cline{2-5}
      \multirow{3}{*}{} & Algorithm \ref{t-lasso:alg2-main} & 3.9860 & 8.5972 & 5.4351  \\
      \hline
   \end{tabular}
   \caption{For the special choice of $(b_1,b_2)$ and $\tau=1e-5$, the running time (seconds) obtained by applying truncated Fourier, and Algorithms \ref{t-lasso:alg1-main} and \ref{t-lasso:alg2-main} with three tensor-product quadratures to the test functions $f_1(x_1,x_2)$, $f_2(x_1,x_2)$ and $f_3(x_1,x_2)$.}
   \label{t-lasso:tab3-main}
\end{table}

\begin{figure}
    \setlength{\tabcolsep}{4pt}
    \renewcommand\arraystretch{1}
    \centering
    \subfloat[$f_1(x_1,x_2)$]{\includegraphics[width=0.8\linewidth]{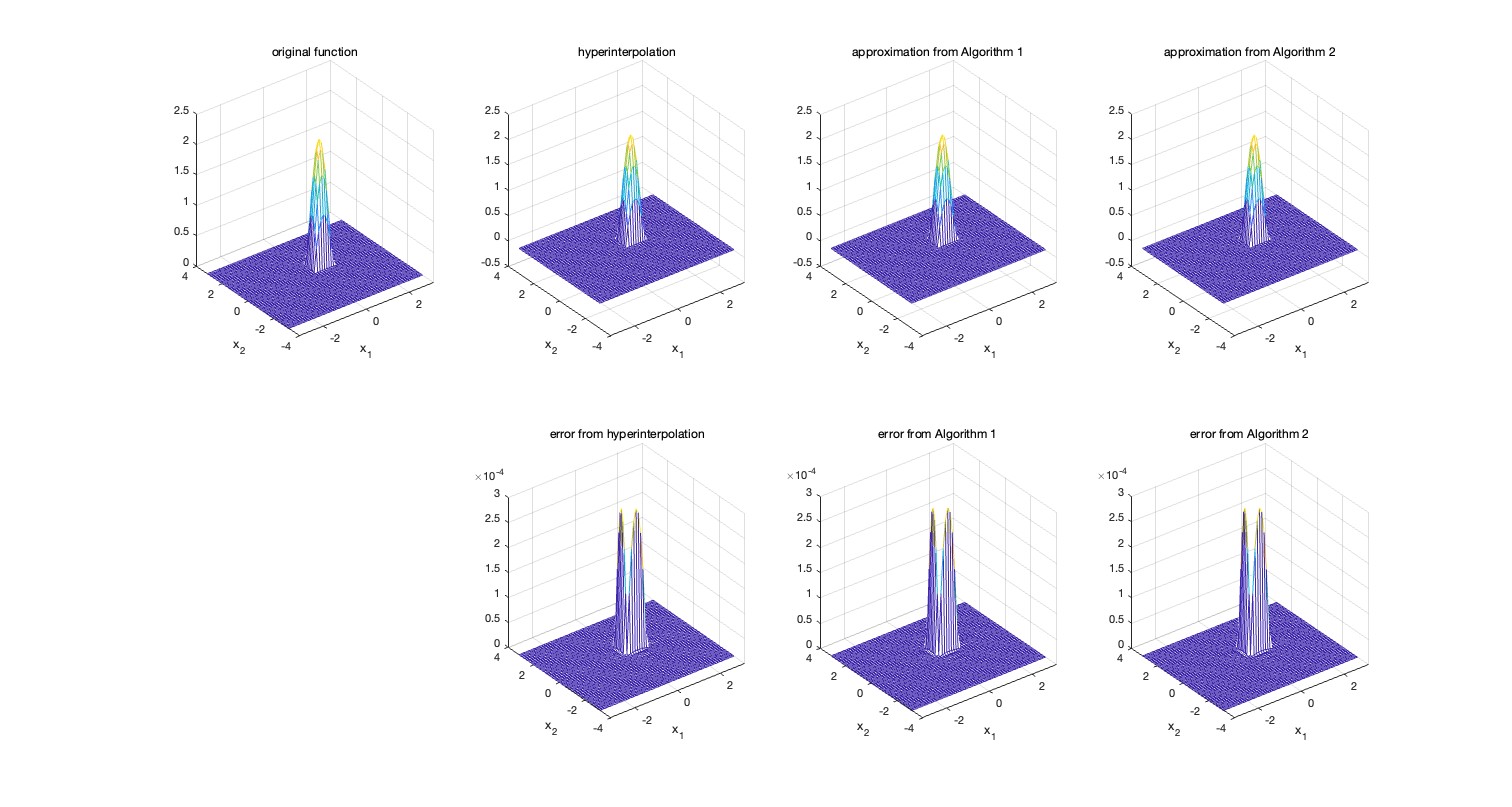}}\\
    \subfloat[$f_2(x_1,x_2)$]{\includegraphics[width=0.8\linewidth]{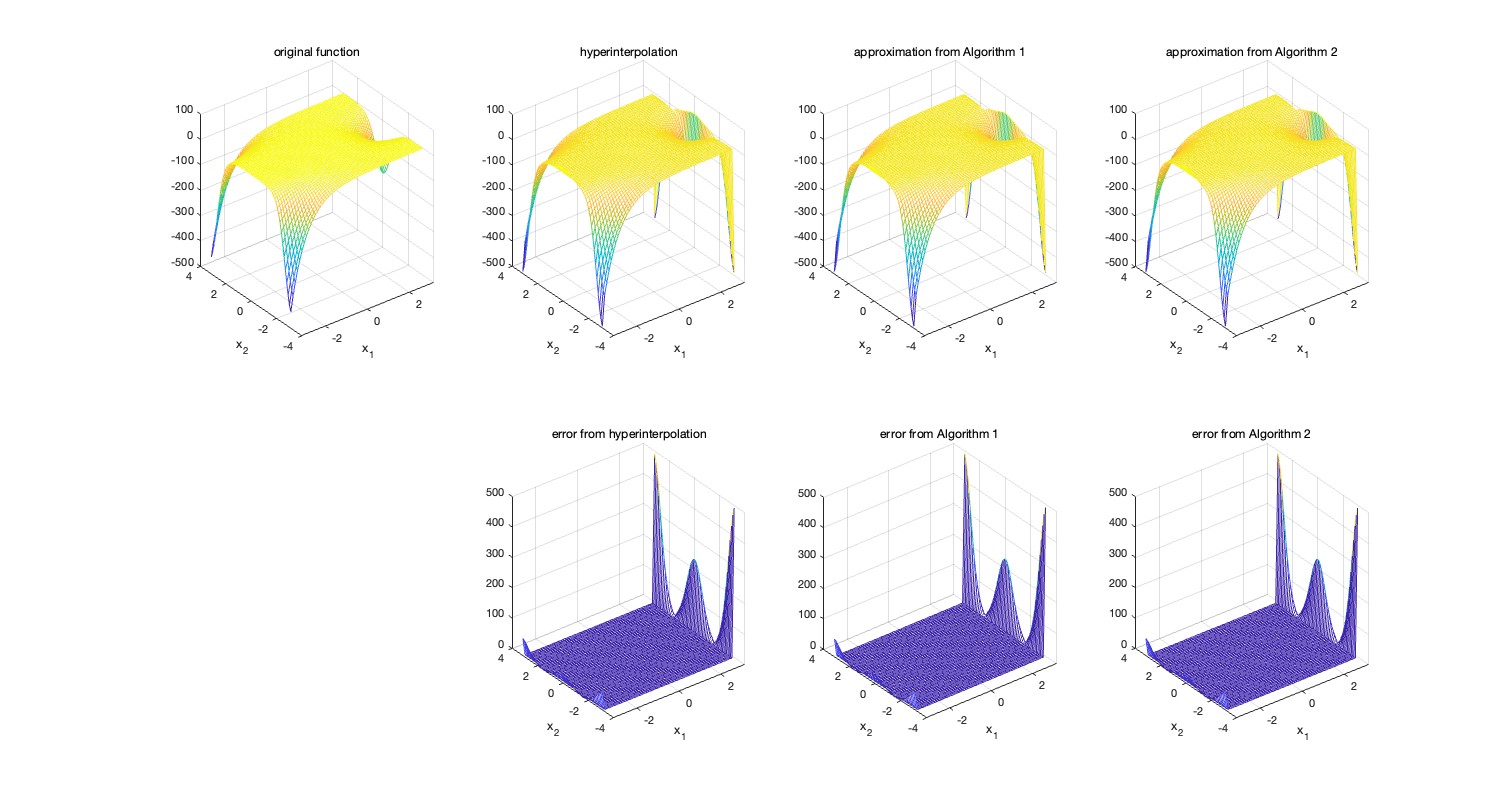}}\\
    \subfloat[$f_3(x_1,x_2)$]{\includegraphics[width=0.8\linewidth]{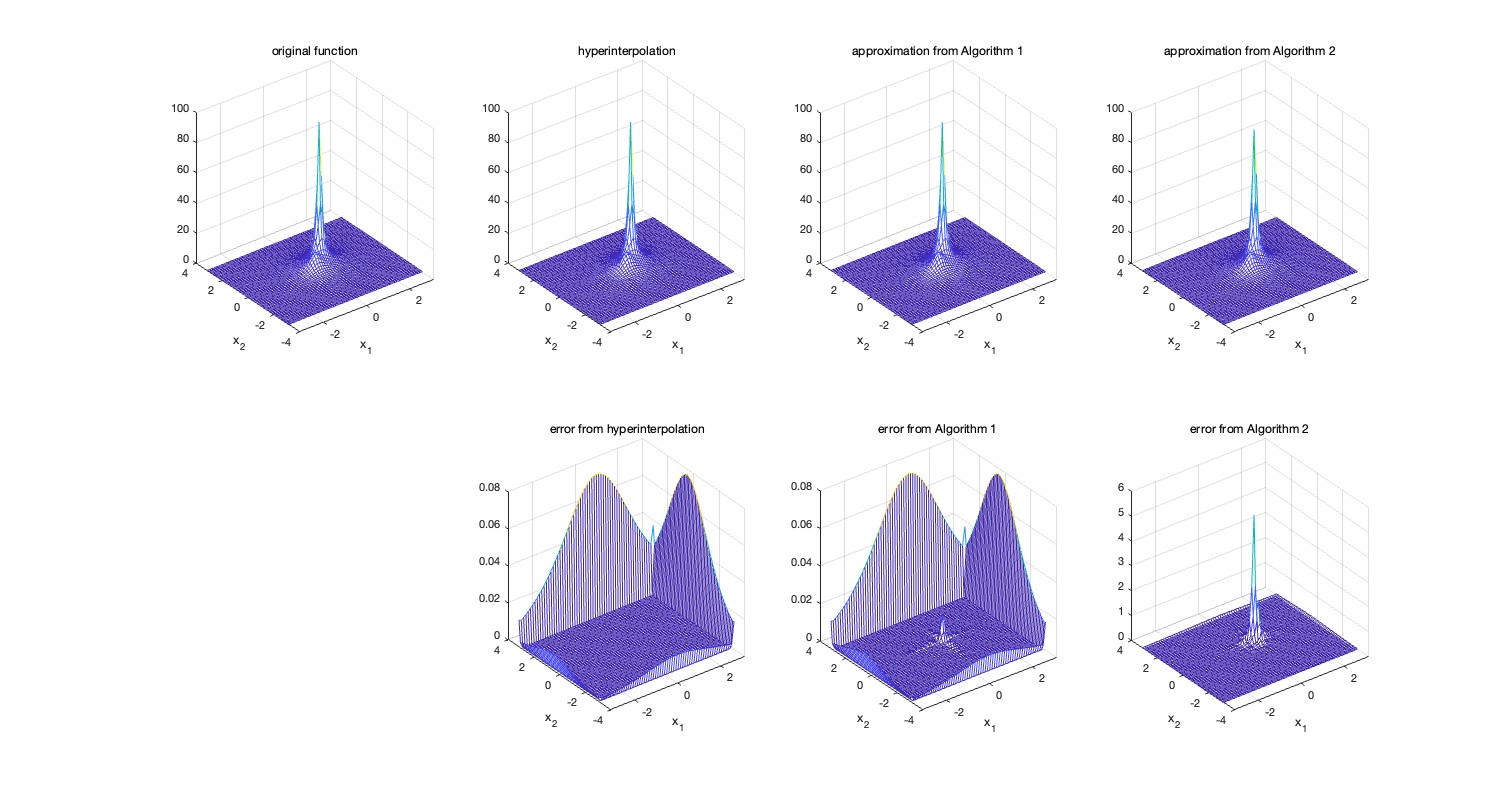}}\\
    \caption{For the special choice of $(b_1,b_2)$ and $\tau=1e-5$, the values of the truncated Fourier series expression and its approximation, and the associated error function obtained by applying Algorithms \ref{t-lasso:alg1-main} and \ref{t-lasso:alg2-main} with the tensor-product Newton-Cotes quadrature to the test functions $f_1(x_1,x_2)$, $f_2(x_1,x_2)$ and $f_3(x_1,x_2)$.}\label{t-lasso:fig3-main}
\end{figure}

\begin{figure}
    \setlength{\tabcolsep}{4pt}
    \renewcommand\arraystretch{1}
    \centering
    \subfloat[$f_1(x_1,x_2)$]{\includegraphics[width=0.8\linewidth]{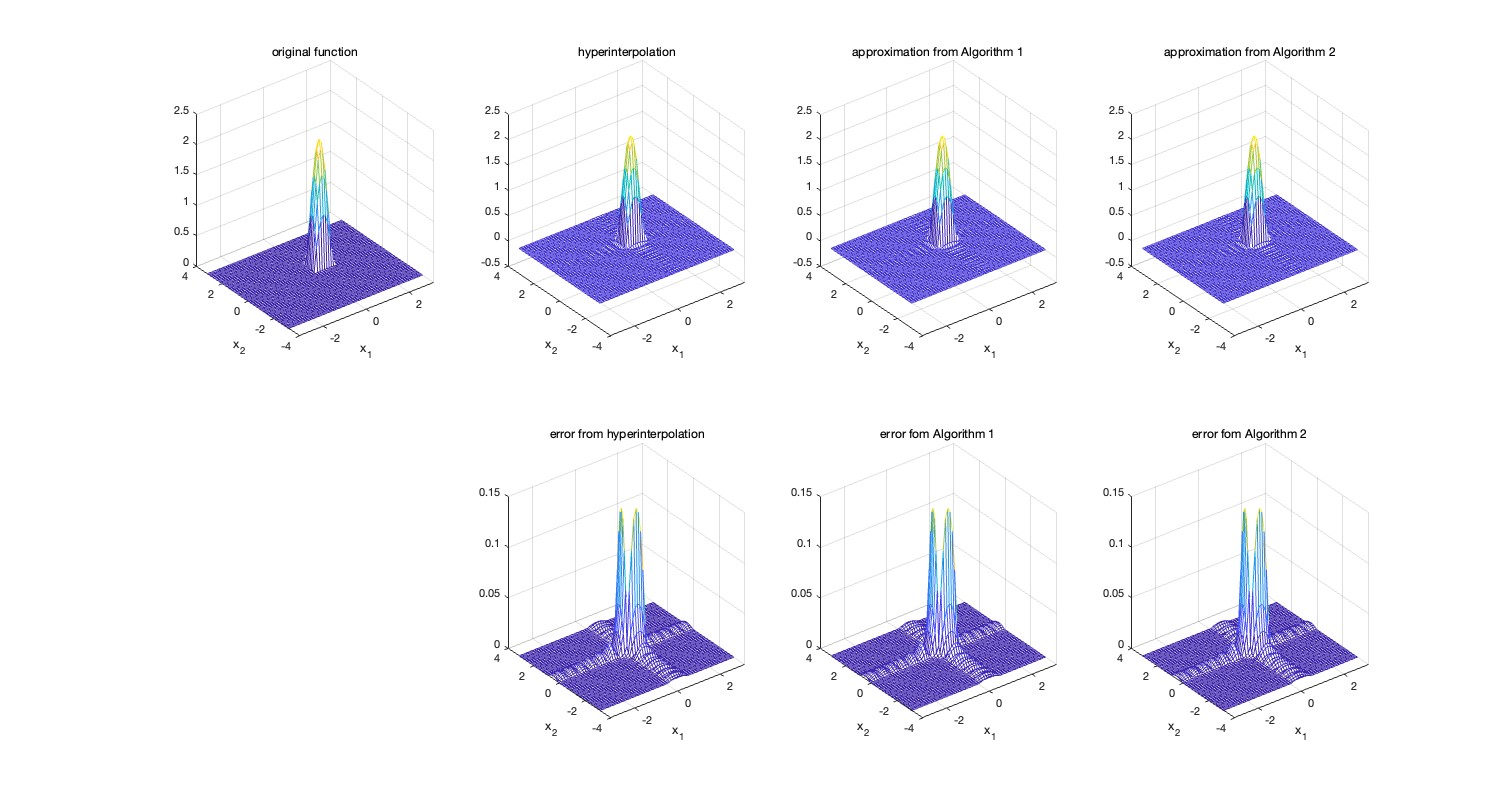}}\\
    \subfloat[$f_2(x_1,x_2)$]{\includegraphics[width=0.8\linewidth]{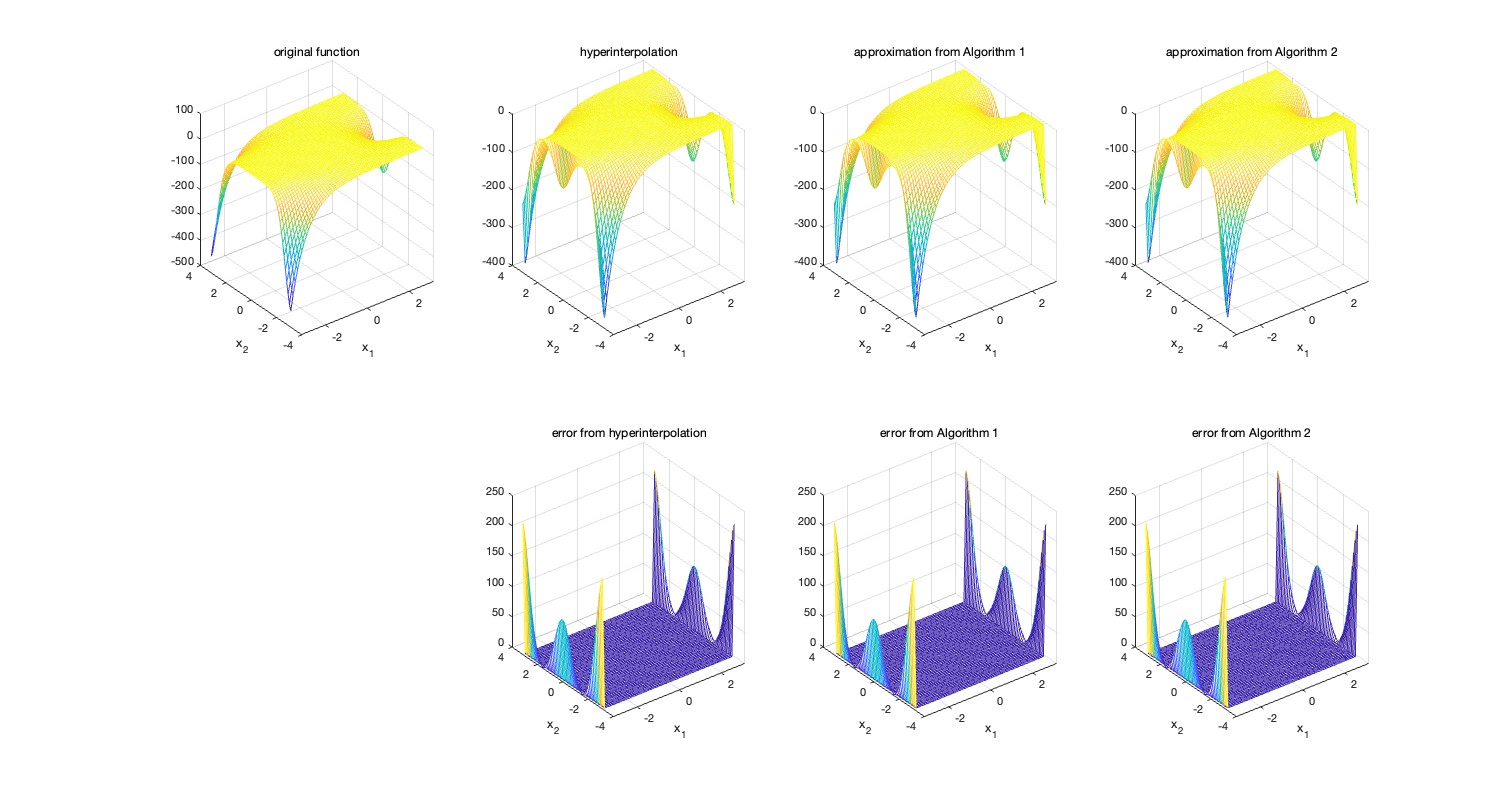}}\\
    \subfloat[$f_3(x_1,x_2)$]{\includegraphics[width=0.8\linewidth]{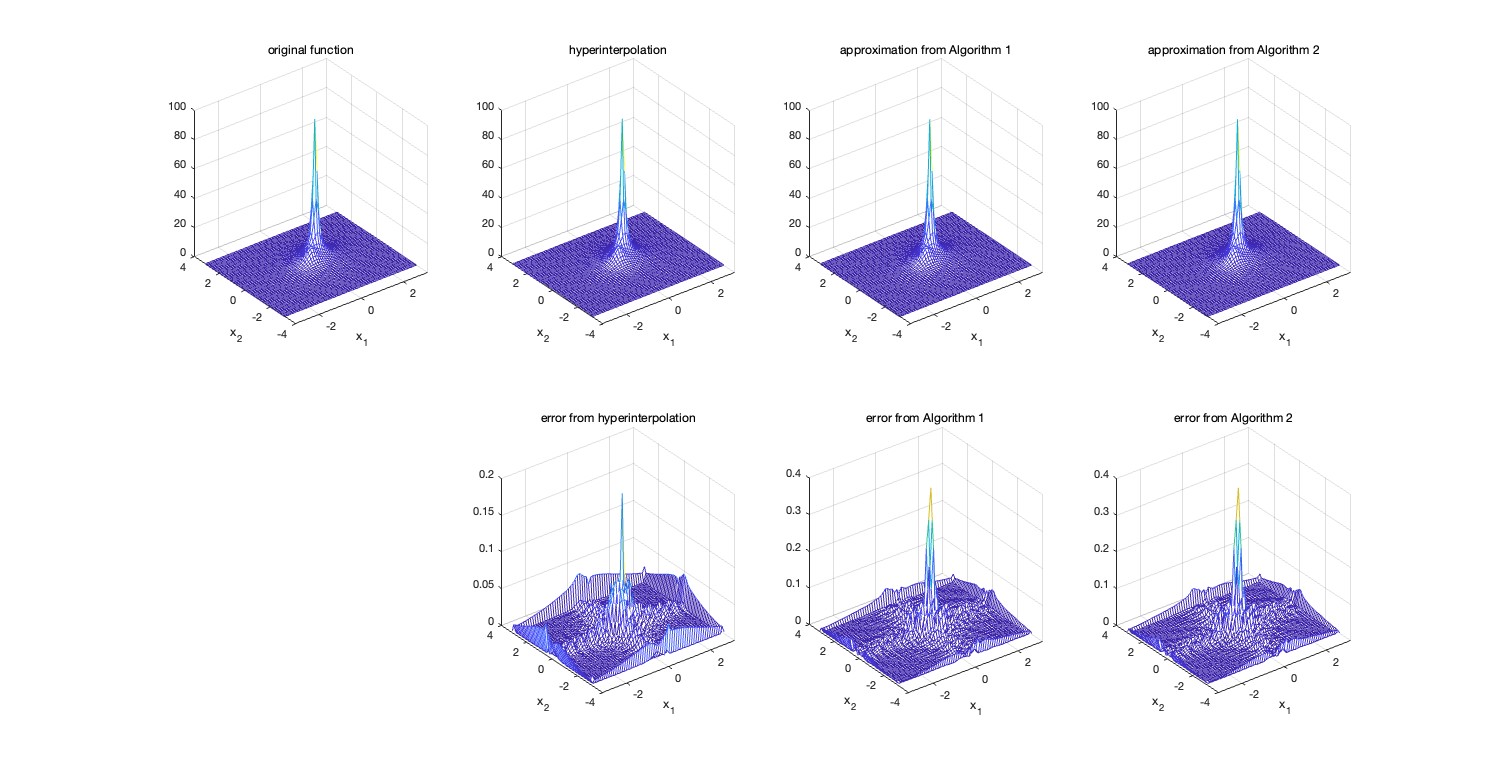}}\\
    \caption{For the special choice of $(b_1,b_2)$ and $\tau=1e-5$, the values of the truncated Fourier series expression and its approximation, and the associated error function  obtained by applying Algorithms \ref{t-lasso:alg1-main} and \ref{t-lasso:alg2-main} with the tensor-product Clenshaw-Curtis quadrature to the test functions $f_1(x_1,x_2)$, $f_2(x_1,x_2)$ and $f_3(x_1,x_2)$.}\label{t-lasso:app:fig5-main}
\end{figure}

\begin{figure}
    \setlength{\tabcolsep}{4pt}
    \renewcommand\arraystretch{1}
    \centering
    \subfloat[$f_1(x_1,x_2)$]{\includegraphics[width=0.8\linewidth]{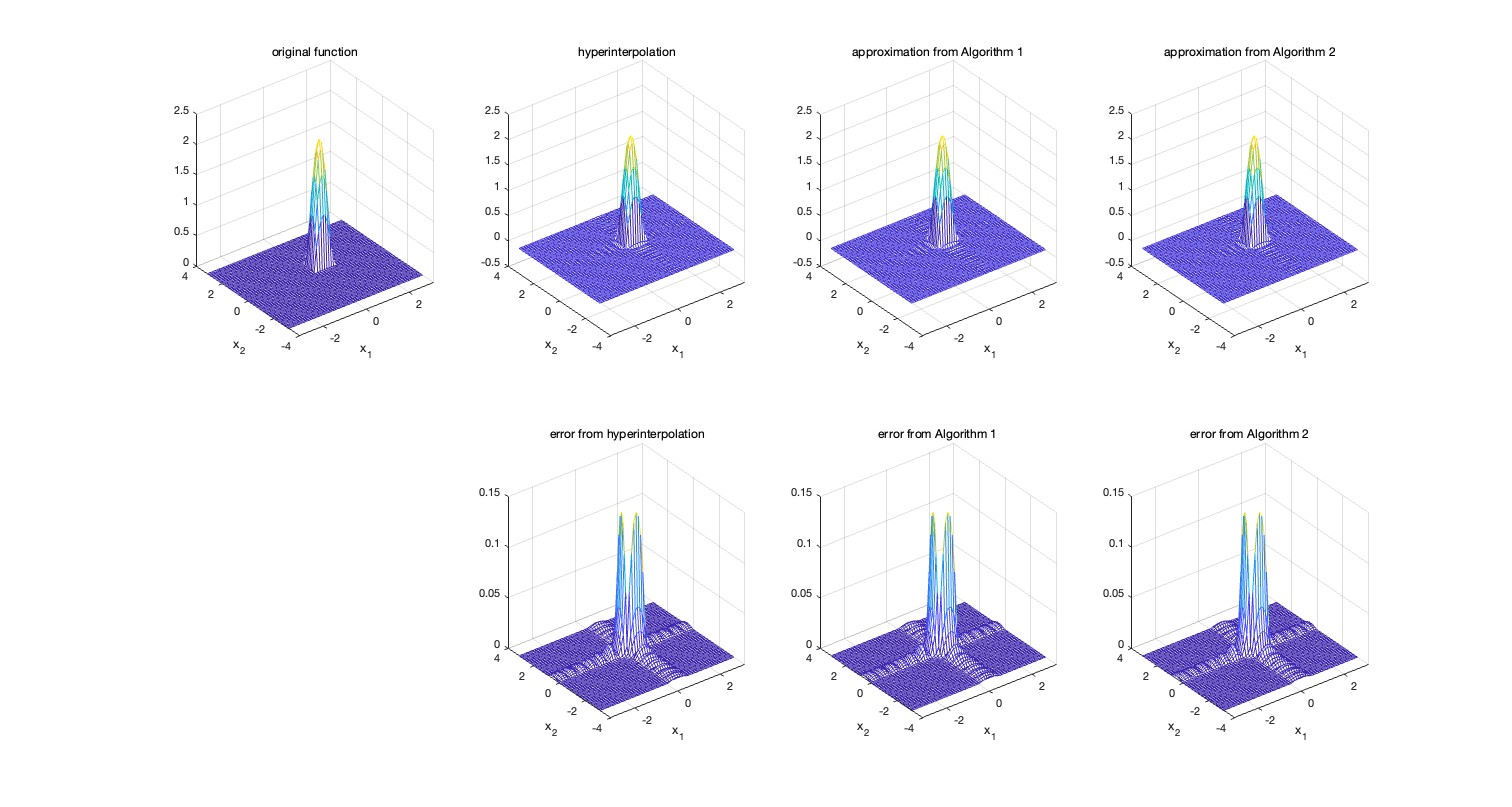}}\\
    \subfloat[$f_2(x_1,x_2)$]{\includegraphics[width=0.8\linewidth]{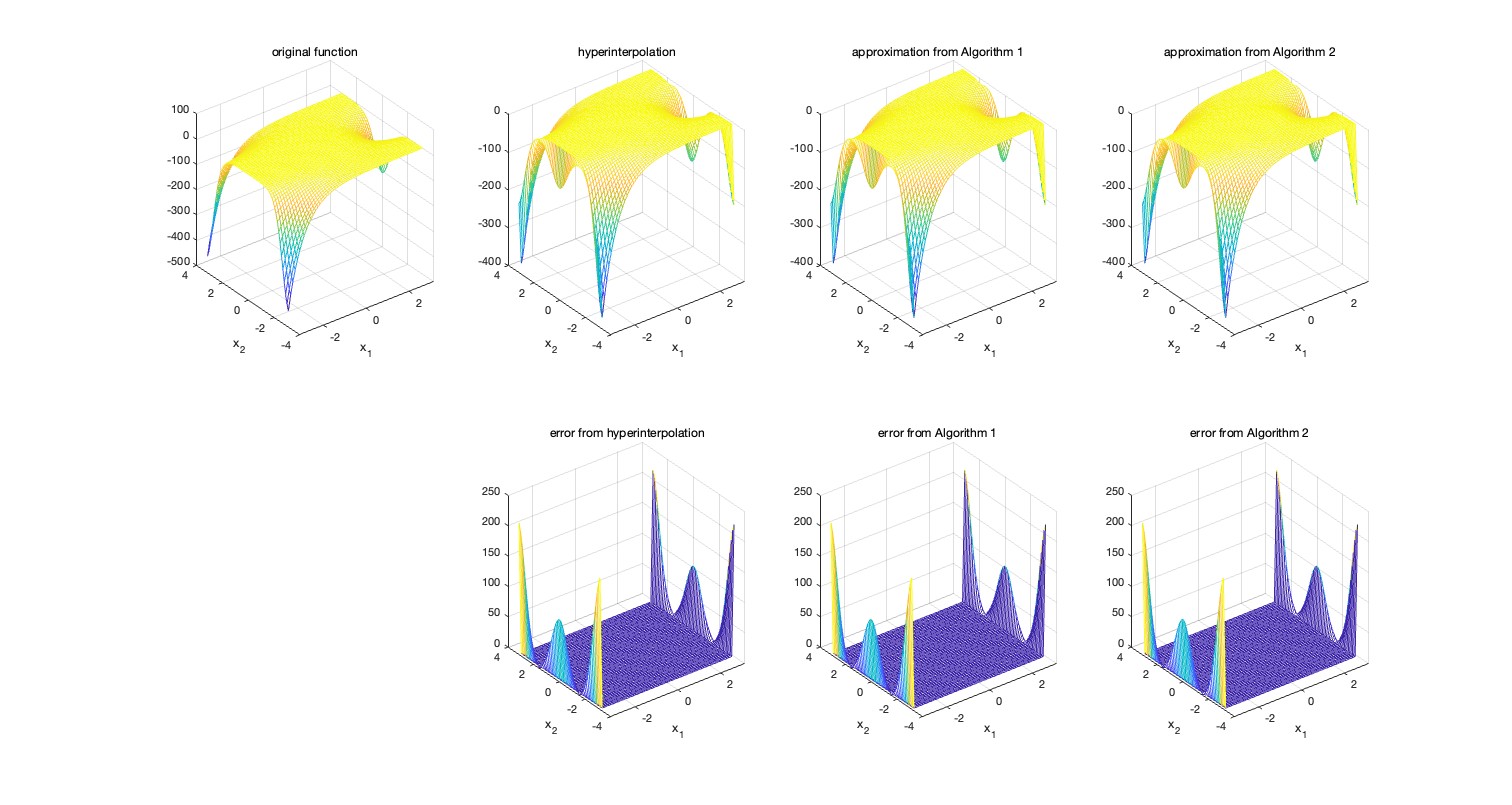}}\\
    \subfloat[$f_3(x_1,x_2)$]{\includegraphics[width=0.8\linewidth]{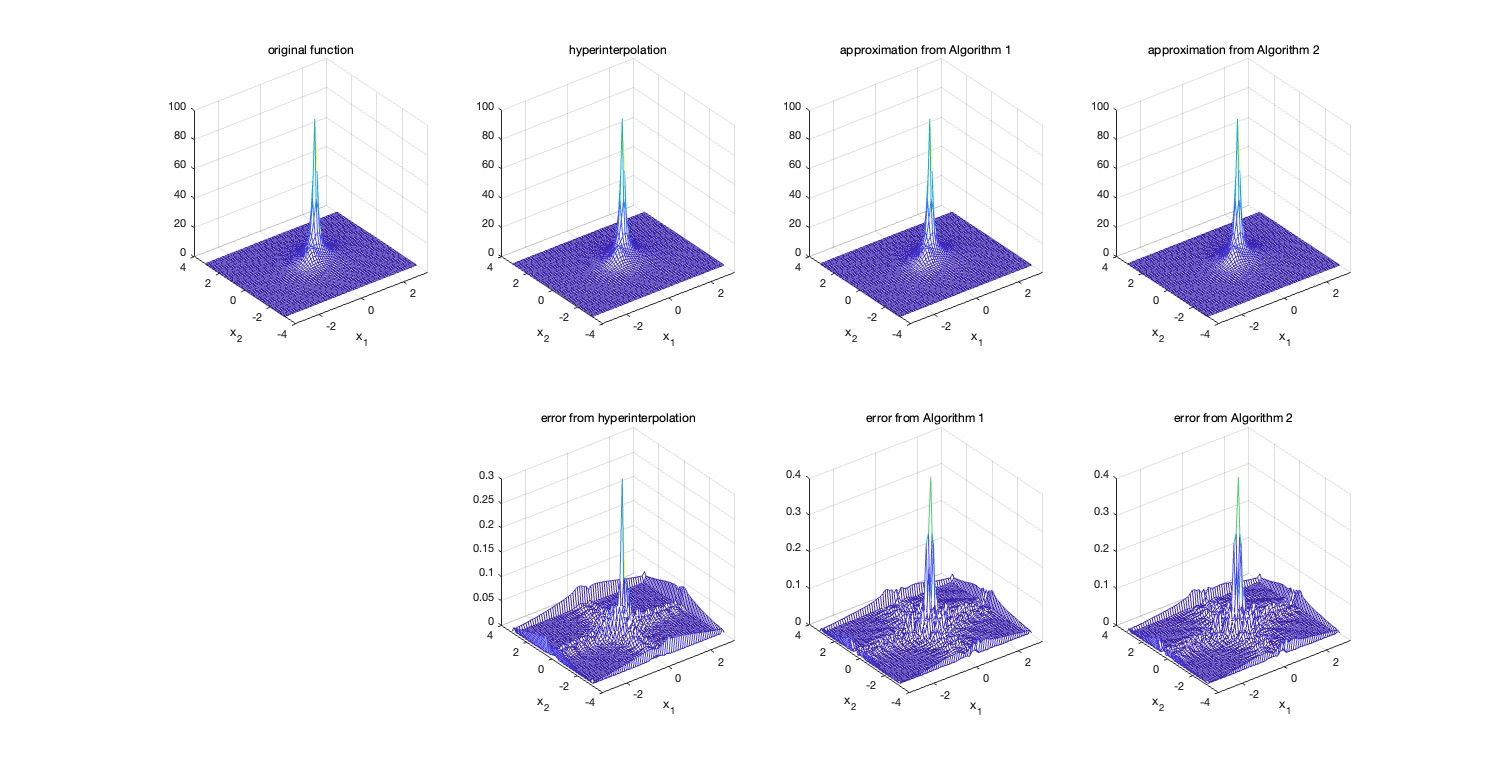}}\\
    \caption{For the special choice of $(b_1,b_2)$ and $\tau=1e-5$, the values of the truncated Fourier series expression and its approximation, and the associated error function  obtained by applying Algorithms \ref{t-lasso:alg1-main} and \ref{t-lasso:alg2-main} with the tensor-product Gauss-Legendre quadrature to the test functions $f_1(x_1,x_2)$, $f_2(x_1,x_2)$ and $f_3(x_1,x_2)$.}\label{t-lasso:app:fig6-main}
\end{figure}
\section{Conclusion}
\label{t-lasso:sec6-main}

This paper tackled the computational bottleneck of traditional hyperinterpolation for bivariate functions—the curse of dimensionality stemming from the need to compute $\mathcal{O}(I_1 I_2)$ Fourier coefficients via double integrals. We introduced a novel, efficient framework that leverages matrix CUR decomposition to construct a low-rank approximation of the coefficient matrix without explicitly calculating all its entries.

Our core contribution is a method that significantly reduces the number of required integrals from $\mathcal{O}(I_1 I_2)$ to $\mathcal{O}(I_1 S_2 + I_2 S_1)$, where $S_n \ll I_n$. This is achieved through two adaptive algorithms that intelligently select a small subset of Fourier modes, compute only the corresponding coefficients, and use them to build a accurate approximation $\widetilde{F}f(x_1,x_2)$ of the full truncated series $Ff(x_1,x_2)$. We provided theoretical error bounds, ensuring the approximation quality is controllable and on par with the standard method when the coefficient matrix is low-rank.

Numerical experiments confirmed the effectiveness of our approach. The algorithms successfully produced approximations with errors comparable to the full truncation while drastically reducing computation time. This work effectively bridges function approximation with matrix decomposition, offering a scalable alternative to spectral methods for high-dimensional problems.

Future work will focus on extending this approach to functions of three or more variables using tensor decompositions (e.g., CP, Tucker, Tensor Train), integrating with sparse grids or (quasi-)Monte Carlo methods to further reduce the number of function evaluations, and enhancing the theory behind the adaptive sampling strategies.

\appendix

\setcounter{equation}{0}
\renewcommand{\theequation}{A.\arabic{equation}}
\setcounter{figure}{0}
\renewcommand{\thefigure}{C.\arabic{figure}}
\setcounter{table}{0}
\renewcommand{\thetable}{C.\arabic{table}}
\setcounter{algorithm}{0}
\renewcommand{\thealgorithm}{C.\arabic{algorithm}}

\section{Hyperinterpolation}
\label{t-lasso:app3-main}
The concept of hyperinterpolation was introduced by Sloan in 1995 \cite{sloan1995hyperinterpolation}. The hyperinterpolation operator $\mathcal{L}_n$ is defined by replacing Fourier integrals in the $L_2$ orthogonal projection onto polynomial spaces with a discrete measure based on a positive weight quadrature rule that has algebraic precision $2n$. The result of \cite{sloan1995hyperinterpolation} has sparked numerous investigations into finding suitable quadrature rules for different regions, thereby expanding the potential applications of hyperinterpolation \cite{an2022exactness,an2024bypassing,An_ran_2025}.

Hyperinterpolation is a powerful tool in high-dimensional approximation \cite{dai2006hyperinterpolation,LeGia2001uniform,lin2021distributed,reimer2012multivariate,Wade2013hyperinterpolation,Wang2017needlet}. It requires only the availability of a positive weight quadrature rule that exactly integrates polynomial of degree $2n$  \cite{sloan1995hyperinterpolation}. This means the function of interest must be sampled on a carefully selected finite set to meet the requirements of the quadrature formula. In most applications, functions are often given by sampled data. However, the modern era of high-throughput data collection creates data with abundant noise. To recover functions from noisy data, An and Wu developed Lasso hyperinterpolation \cite{an2021lasso}, which uses a soft thresholding operator to process all hyperinterpolation coefficients. Although Lasso hyperinterpolation does not retain the projection property or basis invariance of classical hyperinterpolation, it offers an efficient method for basis selection and denoising, leading to a sparse solution. Lasso hyperinterpolation \cite{an2021lasso} is the solution to an $\ell_1$-regularized weighted discrete least squares problem, which is regarded as a convex relaxation of an $\ell_0$-regularized problem \cite{Foucart2013compressing}.

Let $L^2([-1,1])$ denote the Hilbert space of square-integrable functions on $[-1,1]$, equipped with the inner product
\begin{equation}\label{equ:1-1}
\langle f(x),g(x) \rangle = \int_{-1}^{1} f(x)g(x)w(x)\,dx  \qquad \forall f,g \in L^2([-1,1]),
\end{equation}
and the induced norm $\|f\|_2: = \langle f,f\rangle^{1/2}$, where $w(x)$ is a weight function (e.g., $w(x)=1/\sqrt{1-x^2}$ for Chebyshev polynomials of the first kind). For a given integer $I>0$, let $\Pi_I ([-1,1]) \subset L^2([-1,1])$ be the space of polynomials with total degree at most $I$, restricted to $[-1,1]$. We define an orthonormal basis of $\Pi_I([-1,1])$:
\[
\{ \Phi_{\ell}| \ell = 1, \ldots, I+1 \} \subset \Pi_I([-1,1])
\]
satisfying
\[
\langle \Phi_{\ell},\Phi_{\ell'} \rangle = \delta_{\ell \ell'} \qquad \forall 1 \leq \ell, \ell' \leq I+1.
\]

In this section, we model a multivariate function \( f (x^{(1)},\dots,x^{(N)}) \) defined on the canonical domain
$$\Omega =\underbrace{[-1, 1] \times\dots\times [-1,1]}_{N}:=[-1,1]^{\times N}.$$
Let \(\mathcal{F} \in \mathbb{Z}_{+}^{M_1 \times\dots\times M_N}\) represent a discrete array of size $M_1 \times\dots\times M_N$, where each entry $f_{m_1\dots m_N}=f(x_{m_1}^{(1)},\dots,x_{m_N}^{(N)})$ is a real number, sampled at the grid points
 \begin{equation*}
     \mathbb{X}_{\text{source}} = \{(x_{m_1}^{(1)},\dots,x_{m_N}^{(N)}): m_n=1,2,\dots,M_n,n=1,2,\dots,N\}.
 \end{equation*}
For given $N$ positive integers $\{I_1,I_2,\dots,I_N\}$, let $\{\Phi_{i_n}^{(n)}(x^{(n)})\}_{i_n=1}^{I_n+1}$ be an orthonormal basis of $\Pi_{I_n}([-1,1])$. The goal of image scaling (downscaling or upscaling) is to accurately reconstruct the underlying function $f (x^{(1)},\dots,x^{(N)})$ on a new grid of size $S_1\times \dots \times S_N$, denoted by
\begin{equation*}\label{equ:1-3}
\mathbb{X}_{\text{target}} = \{(y_{s_1}^{(1)},\dots,y_{s_N}^{(N)}): s_n=1,2,\dots,S_n,n=1,2,\dots,N\},
\end{equation*}
thereby producing a resized image $\mathcal{R} \in \mathbb{R}^{S_1 \times \dots \times S_N}$ with entries
\[
r_{s_1 \dots s_N} \approx f(y_{s_1}^{(1)},\dots,y_{s_N}^{(N)}).
\]

The core challenge is to approximate $f (x^{(1)},\dots,x^{(N)})$ on $\mathbb{X}_{\text{target}}$ using only its known values on $\mathbb{X}_{\text{source}}$. This is achieved through a polynomial expansion of the form
\begin{equation*}
    p_{I_1,\dots,I_N}(x^{(1)},\dots,x^{(N)}) = \sum_{i_1=1}^{I_1+1} \dots\sum_{i_N=1}^{I_N+1} c_{i_1\dots i_N} \Phi_{1}^{(i_1)}(x^{(1)}) \dots\Phi_{N}^{(i_N)}(x^{(N)}),
\end{equation*}
where the coefficients $c_{i_1\dots i_N}$ are determined by $L_2$-orthogonal projection:
{\small
\begin{equation*}
   \begin{split}
       &c_{i_1\dots i_N} = \langle f (x^{(1)},\dots,x^{(N)}), \Phi_{i_1}^{(1)}(x^{(1)}) \dots\Phi_{i_N}^{(N)}(x^{(N)}) \rangle \\
       &= \int_{-1}^1\dots \int_{-1}^1 f (x^{(1)},\dots,x^{(N)}) \Phi_{i_1}^{(1)}(x^{(1)}) \dots\Phi_{i_N}^{(N)}(x^{(N)}) w^{(1)}(x^{(1)}) \dots w^{(N)}(x^{(N)})  \,dx^{(1)}\dots  dx^{(N)}.
   \end{split}
\end{equation*}
}

However, a major practical obstacle arises: since $f$ is typically provided only as discrete values, these integrals cannot be computed  exactly. To overcome this, we employ numerical integration via positive quadrature rules--a concept central to hyperinterpolation, introduced by Sloan in the seminal paper \cite{sloan1995hyperinterpolation}. Hyperinterpolation replaces the continuous inner product \eqref{equ:1-1} with a discrete (semi) inner product constructed using a cubature rule exact for polynomials up to degree $2I$.

Let $\{x_m\}^{M}_{m=1}$ and  $\{w_m\}_{m=1}^{M}$ be quadrature nodes and positive weights on $[-1,1]$ such that
\begin{equation*}
    \int_{-1}^{1} g(x) w(x)\, dx \approx \sum_{m=1}^{M}w_m g(x_m),
\end{equation*}
is exact for all $g \in \Pi_{2I}([-1,1])$. The \emph{hyperinterpolation operator} $\mathcal{L}_I: \mathcal{C}([-1,1]) \to \Pi_I([-1,1])$ as
\begin{equation*}
\mathcal{L}_I f:= \sum_{\ell=1}^{I+1} \langle {f, \Phi_{\ell}} \rangle_M   \Phi_{\ell},
\end{equation*}
where $\langle \cdot, \cdot \rangle_M$ denotes the discrete (semi) inner product introduced by the quadrature rule:
\begin{equation*}
    \langle {f, \Phi_{\ell}} \rangle_M=\sum_{m=1}^{M}w_m f(x_m)\Phi_{\ell}(x_m).
\end{equation*}

The coefficients $c_{i_1\dots i_N}$ are thus approximated by $c_{i_1\dots i_N}\approx \widetilde{c}_{i_1\dots i_N}$ with
\begin{equation*}
\begin{aligned}
    \widetilde{c}_{i_1\dots i_N}&=\sum_{m_1=1}^{M_1}\dots\sum_{m_N=1}^{M_N}
    f(x_{m_1}^{(1)}, \dots, x_{m_N}^{(N)}) \Phi_{i_1}^{(1)}(x_{m_1}^{(1)}) \dots\Phi_{i_N}^{(N)}(x_{m_N}^{(N)}) w_{m_1}^{(1)} \dots w_{m_N}^{(N)},
\end{aligned}
\end{equation*}
with $i_n=1,2,\dots,I_n+1$ and $n=1,2,\dots,N$, where for each $n$, $\{x_{m_n}^{(n)}\}^{M_n}_{m_n=1}$ and  $\{w_{m_n}^{(n)}\}^{M_n}_{m_n=1}$ are quadrature nodes and positive weights on $[-1,1]$, and the hyperinterpolation polynomial over $\Omega=[-1,1]^{\times N}$ is constructed as:
\begin{equation*}
    \mathcal{L}_{I_1,\dots I_N} f (x^{(1)},\dots,x^{(N)}) = \sum_{i_1=1}^{I_1+1} \dots\sum_{i_N=1}^{I_N+1} \widetilde{c}_{i_1\dots i_N} \Phi_{1}^{(i_1)}(x^{(1)}) \dots\Phi_{N}^{(i_N)}(x^{(N)}).
\end{equation*}


Let $\mathbf{W}_{n} = \operatorname{diag}(w_{1}^{(n)},w_{2}^{(n)}, \cdots,w_{M_n}^{(n)})$, $\mathbf{A}_{n}=[\Phi_{i_n}^{(n)}(x_{m_n}^{(n)})] \in \mathbb{R}^{(I_n+1) \times M_n}$, and $\widetilde{\mathcal{C}}=[\widetilde{c}_{i_1\dots i_N}]\in \mathbb{R}^{(I_1+1)\times\dots\times (I_N+1)}$. On the source grid $\mathbb{X}_{\text{source}}$, the coefficient tensor $\widetilde{\mathcal{C}}$ can be approximated  as:
\[
\widetilde{\mathcal{C}}=\mathcal{F}\times_1(\mathbf{A}_1\mathbf{W}_1)\dots\times_N(\mathbf{A}_N\mathbf{W}_N).
\]
Finally, the reconstructed function values on the target  grid $\mathbb{X}_{\text{target}}$ are given by:
\[
\mathcal{R} = \widetilde{\mathcal{C}}\times_1\widetilde{\mathbf{A}}_1^\top\dots\times_N\widetilde{\mathbf{A}}_N^\top,
\]
where $\widetilde{\mathbf{A}}_n =[\Phi_{i_n}^{(n)}(y_{s_n}^{(n)})] \in \mathbb{R}^{(I_n+1) \times S_n}$ with $n=1,2,\dots,N$. The number of all entries in $\widetilde{\mathcal{C}}$ is $\prod_{n=1}^{N}(I_n+1)$, which leads to the curse of dimensionality. Hence, one key work is to design efficient algorithms for obtaining an approximation $\widetilde{\mathcal{C}}_{{\rm approx}}$ to $\widetilde{\mathcal{C}}$ such that the Frobenius norm of $\mathcal{R}_{{\rm approx}}-\mathcal{R}$ is less than a given tolerance, where $\mathcal{R}_{{\rm approx}} = \widetilde{\mathcal{C}}_{{\rm approx}}\times_1\widetilde{\mathbf{A}}_1^\top\dots\times_N\widetilde{\mathbf{A}}_N^\top$.
\begin{remark}
    The readers can refer to \cite{kolda2009tensor} for the symbol of any tensor, the definition of tensor-matrix multiplication and the Frobenius norm of a tensor. In this paper, we focus on the case of $N=2$ based on the matrix decomposition strategy. Based on the difference between tensor decomposition and matrix decomposition, the case of $N>3$ will be considered in the future work.
\end{remark}
\section{Existing work for finding a good set of indices}
\label{t-lasso:app1-main}
It is worth noting that the low-rank approximation of a matrix is ubiquitous in computational sciences. An efficient method for calculating a low-rank approximation of a matrix is based on a subset of rows and/or columns of this matrix. As we know, the quality of the low-rank approximation is decided by a good subset of row and column indices, which is also called as the Column Subset Selection Problem (CSSP). CSSP is a classical linear algebra problem that connects to a variety of fields, including theoretical computer science and statistical learning. This type of low-rank approximation includes CUR decomposition \cite{boutsidis2017optimal,mahoney2009cur,sorensen2016deim}, interpolative decompositions \cite{voronin2017efficient}, (pseudo-)skeleton approximation \cite{goreinov1997theory}, adaptive cross approximation \cite{bebendorf2000approximation}, pivoted/rank-revealing QR decompositions \cite{gu1996efficient}, pivoted Cholesky decompositions \cite{harbrecht2012low}, and so on. The existing algorithms for finding a good set of row and column indices of a matrix $\mathbf{A}\in\mathbb{R}^{I\times J}$ can be divided into two different categories: pivoting or sampling.

For pivoting-based methods, we can use column pivoted QR (CPQR) decomposition or LU factorization with complete pivoting (see e.g., \cite{golub2013matrix}) on $\mathbf{A}$ or the singular vectors of $\mathbf{A}$ to obtain the pivots which we then use as row or column indices. The discrete empirical interpolation method (DEIM) (see \cite{chaturantabut2010nonlinear}) is another popular method for obtaining pivots from the dominant singular vectors of $\mathbf{A}$. The DEIM method needs estimations of the top singular vectors to proceed. A new variant of DEIM, called L-DEIM (see \cite{gidisu2021hybrid}), is obtained by combining the deterministic leverage scores and DEIM. This method allows for the selection of a number of indices greater than the number of input singular vectors.

Several randomized algorithms based on sketching and/or sampling have been proposed to find a good set of row and column indices of a large-scale matrix (see, e.g., \cite{drineas2008relative,duersch2020randomized,mahoney2009cur,voronin2017efficient}). For the sketch-based algorithms, we use the pivoting schemes to a much smaller matrix, which is obtained by using a random projection matrix to project the matrix $\mathbf{A}$. For the sampling-based methods, the column or row indices are sampled from some probability distributions, which are obtained from certain information about $\mathbf{A}$. There exist several sampling strategies, such as subspace sampling (see \cite{drineas2008relative,mahoney2009cur}), uniform sampling (see \cite{chiu2013sublinear}), volume sampling (see \cite{cortinovis2020low}), DPP sampling (see \cite{derezinski2021determinantal}), and BSS sampling (see \cite{boutsidis2014near}). In particular, the CUR approximation obtained from volume sampling has close-to-optimal error guarantees \cite{cortinovis2020low,voronin2017efficient}.

\section{An modified version from Algorithm \ref{t-lasso:alg2-main}}
\label{t-lasso:app2-main}
For each $k$ in Algorithm \ref{t-lasso:alg2-main}, we need to obtain three matrices $\mathbf{G}_{1k}\in\mathbb{C}^{(2(k-1)b_1+1)\times 2b_2}$, $\mathbf{G}_{2k}\in\mathbb{C}^{2b_1\times (2(k-1)b_2+1)}$ and $\mathbf{G}_{3k}\in\mathbb{C}^{2b_1\times 2b_2}$ according to (\ref{t-lasso:method-one:sub3-main}). When we set $\mathbf{G}_{1k}$ and $\mathbf{G}_{2k}$ as the zero matrix, a simplified version for Algorithm \ref{t-lasso:alg2-main} is summarized in Algorithm \ref{t-lasso:alg1:app-main}.
\begin{algorithm}[htb]
    \caption{A modification of Algorithm \ref{t-lasso:alg2-main}}
    \begin{algorithmic}[1]
        \STATEx {\bf Input}: A function $f(x_1,x_2)\in H^{\alpha}(\mathbb{H}_2^0)$, a given $0<\epsilon<1$, the pair of block sizes $(b_1,b_2)$, a tolerance $0<\tau<1$ and the maximum number of iterations $K$.
        \STATEx {\bf Output}: The 3-tuple $\{\mathbf{C},\mathbf{U},\mathbf{R}\}$, which is used to form the function $\widetilde{F}f(x_1,x_2)$.
        \STATE Initialize $k=0$ and ${\rm tol}=+\infty$.
        \STATE According to Theorem \ref{t-lasso:general-theorem-main} to estimate the pair $\{I_1,I_2\}$ such that $I_n=O((1/\epsilon)^{1/\alpha})$.
        \STATE Select $\mathbb{T}_n=\{0\}$ and let $\mathbb{T}_n':=\mathbb{T}_n$ with $n=1,2$.
        \STATE Compute $\mathbf{G}$ according to (\ref{t-lasso:method-one:sub3-main}) with $\{\mathbb{T}_1',\mathbb{T}_2'\}$.
        \STATE Compute ${\rm nF}_{\mathbf{G}}=\|\mathbf{G}\|_F^2$.
        \WHILE{${\rm tol}>\tau$ or $k\leq K$}
           \STATE Select $\mathbb{T}_n'=[-kb_n:-(k-1)b_n-1]\cup[(k-1)b_n+1:kb_n]$ with $n=1,2$.
           \STATE Compute $\mathbf{G}_{k}$ according to (\ref{t-lasso:method-one:sub3-main}) with $\{\mathbb{T}_1',\mathbb{T}_2'\}$, respectively.
           \STATE Update ${\rm nF}_{\mathbf{G}}={\rm nF}_{\mathbf{G}}+\|\mathbf{G}_{k}\|_F^2$ and
           \begin{align*}
              \mathbf{G}=
              \begin{bmatrix}
                 \mathbf{G} & \mathbf{0}_{2(k-1)b_1+1,2b_2}\\
                 \mathbf{0}_{2b_1,2(k-1)b_2+1}& \mathbf{G}_{k}
              \end{bmatrix}.
           \end{align*}
           \STATE Obtain ${\rm tol}=\sigma_{\min}(\mathbf{G})/\sqrt{{\rm nF}_{\mathbf{G}}}$.
           \STATE Update $k=k+1$.
        \ENDWHILE
        \STATE Compute $\mathbf{U}=\mathbf{G}^\dag$.
        \STATE Compute $\mathbf{C}$ according to (\ref{t-lasso:method-one:sub1-main}) with $\mathbb{T}_2$, and $\mathbf{R}$ according to (\ref{t-lasso:method-one:sub2-main}) with $\mathbb{T}_1$.
        \STATE Return the 3-tuple $\{\mathbf{C},\mathbf{U},\mathbf{R}\}$.
    \end{algorithmic}
    \label{t-lasso:alg1:app-main}
\end{algorithm}

Similarly to Algorithm \ref{t-lasso:alg2-main}, Algorithm \ref{t-lasso:alg1:app-main} needs
\begin{equation*}
   (2I_1+1)S_2+(2I_2+1)S_1-2S_1S_2+Kb_1b_2
\end{equation*}
double integrals to obtain the 3-tuple $\{\mathbf{C},\mathbf{U},\mathbf{R}\}$, that is, the number of double integrals in Algorithm \ref{t-lasso:alg1:app-main} is less than Algorithms \ref{t-lasso:alg1-main} and \ref{t-lasso:alg2-main}.

Finally, we count the complexity of Algorithm \ref{t-lasso:alg1:app-main}: a) when $k=0$, it costs one operation to obtain ${\rm nF}_{\mathbf{G}}$; b) for each $k\geq 1$, to update ${\rm nF}_{\mathbf{G}}$ requires $8b_1b_2$ operations; c) to compute $\sigma_{\min}(\mathbf{G})$ needs $O(4b_1b_2\min\{2b_1,2b_2\})$ operations; and d) to form the matrix $\mathbf{U}$ amends $O(4b_1b_2\min\{2b_1,2b_2\})$ operations.

We now compare the efficiencies of Algorithms \ref{t-lasso:alg2-main} and \ref{t-lasso:alg1:app-main} via three test functions $f_1(x_1,x_2)$, $f_3(x_1,x_2)$ and $f_3(x_1,x_2)$. The choices of parameters $(b_1,b_2)$, $K$, $\epsilon$ and $\tau$ in Algorithm \ref{t-lasso:alg1:app-main} are the same as that in Algorithm \ref{t-lasso:alg2-main} (used in Section \ref{t-lasso:sec5-main}). The values of the running time and the error function, and associated to CC, GL and NC are shown in Table \ref{t-lasso:tab1:app-main} and Figure \ref{t-lasso:app:fig7-main}, respectively.

\begin{table}[htb]
   \scriptsize
   \centering
   \begin{tabular}{|c|c|c|c|c|}
      \hline
      Types  & Algorithms & $f_1(x_1,x_2)$ & $f_2(x_1,x_2)$ & $f_3(x_1,x_2)$ \\
      \hline
      \multirow{2}{*}{CC} & Algorithm \ref{t-lasso:alg2-main}   &  5.3283 & 5.6843  & 7.5571  \\
      \cline{2-5}
      \multirow{2}{*}{} & Algorithm \ref{t-lasso:alg1:app-main} & 4.0714 & 7.3718 & 8.7227  \\
      \hline
      \multirow{2}{*}{GL} & Algorithm \ref{t-lasso:alg2-main}   & 4.4734 & 6.1846 & 6.0974  \\
      \cline{2-5}
      \multirow{2}{*}{} & Algorithm \ref{t-lasso:alg1:app-main} & 4.6489 & 6.1786 & 6.0174  \\
      \hline
      \multirow{2}{*}{NC} & Algorithm \ref{t-lasso:alg2-main}          & 7.2544 & 5.9793 & 8.4729  \\
      \cline{2-5}
      \multirow{2}{*}{} & Algorithm \ref{t-lasso:alg1:app-main} & 4.5033 & 8.7171 & 7.6851  \\
      \hline
   \end{tabular}
   \caption{For the special choice of $(b_1,b_2)$ and $\tau=1e-5$, the running time (seconds) obtained by applying truncated Fourier, and Algorithms \ref{t-lasso:alg2-main} and \ref{t-lasso:alg1:app-main} with three tensor-product quadratures to the test functions $f_1(x_1,x_2)$, $f_2(x_1,x_2)$ and $f_3(x_1,x_2)$.}
   \label{t-lasso:tab1:app-main}
\end{table}

\begin{figure}
    \setlength{\tabcolsep}{4pt}
    \renewcommand\arraystretch{1}
    \centering
    \subfloat[$f_1(x_1,x_2)$]{\includegraphics[width=0.8\linewidth]{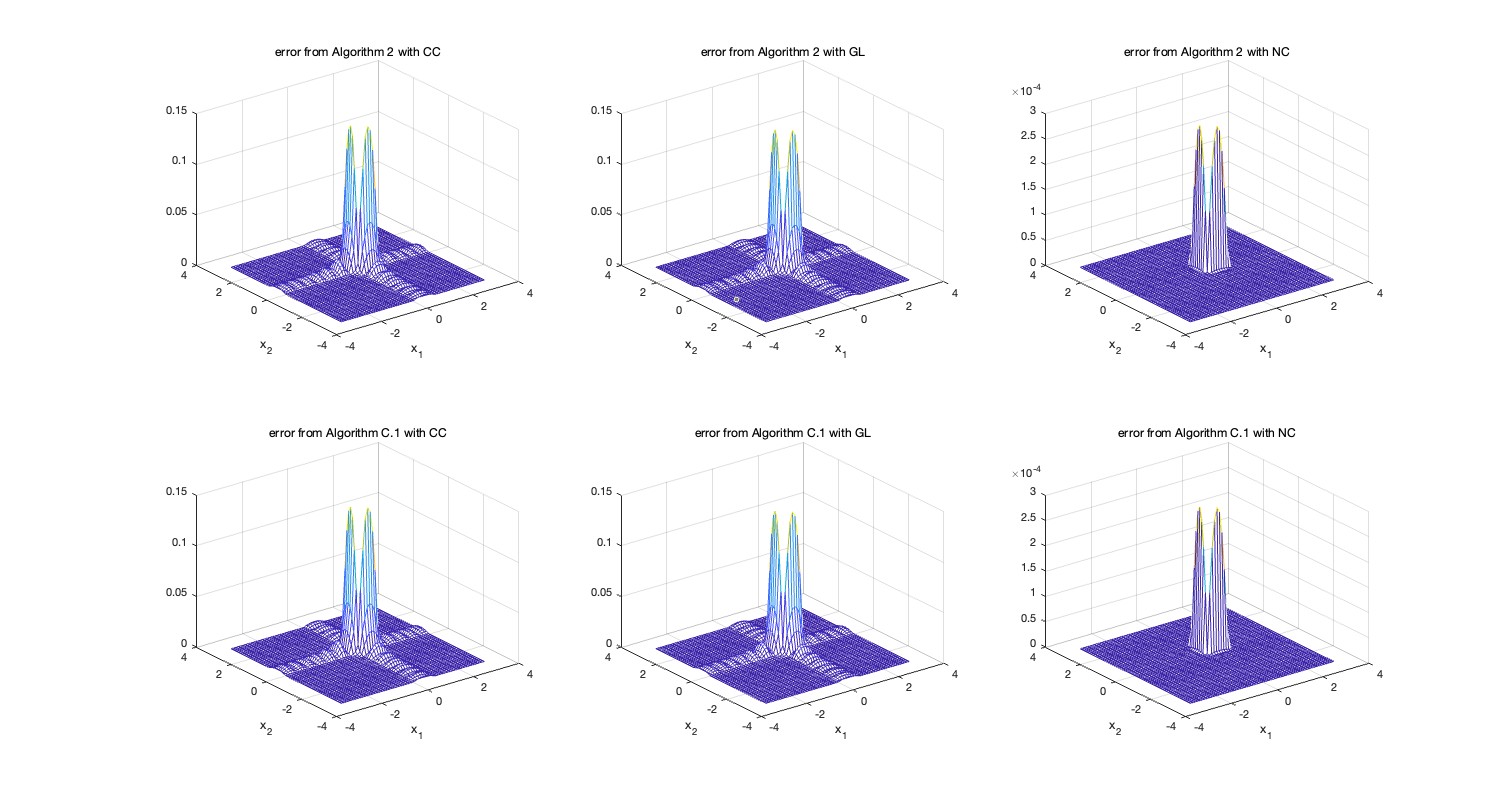}}\\
    \subfloat[$f_2(x_1,x_2)$]{\includegraphics[width=0.8\linewidth]{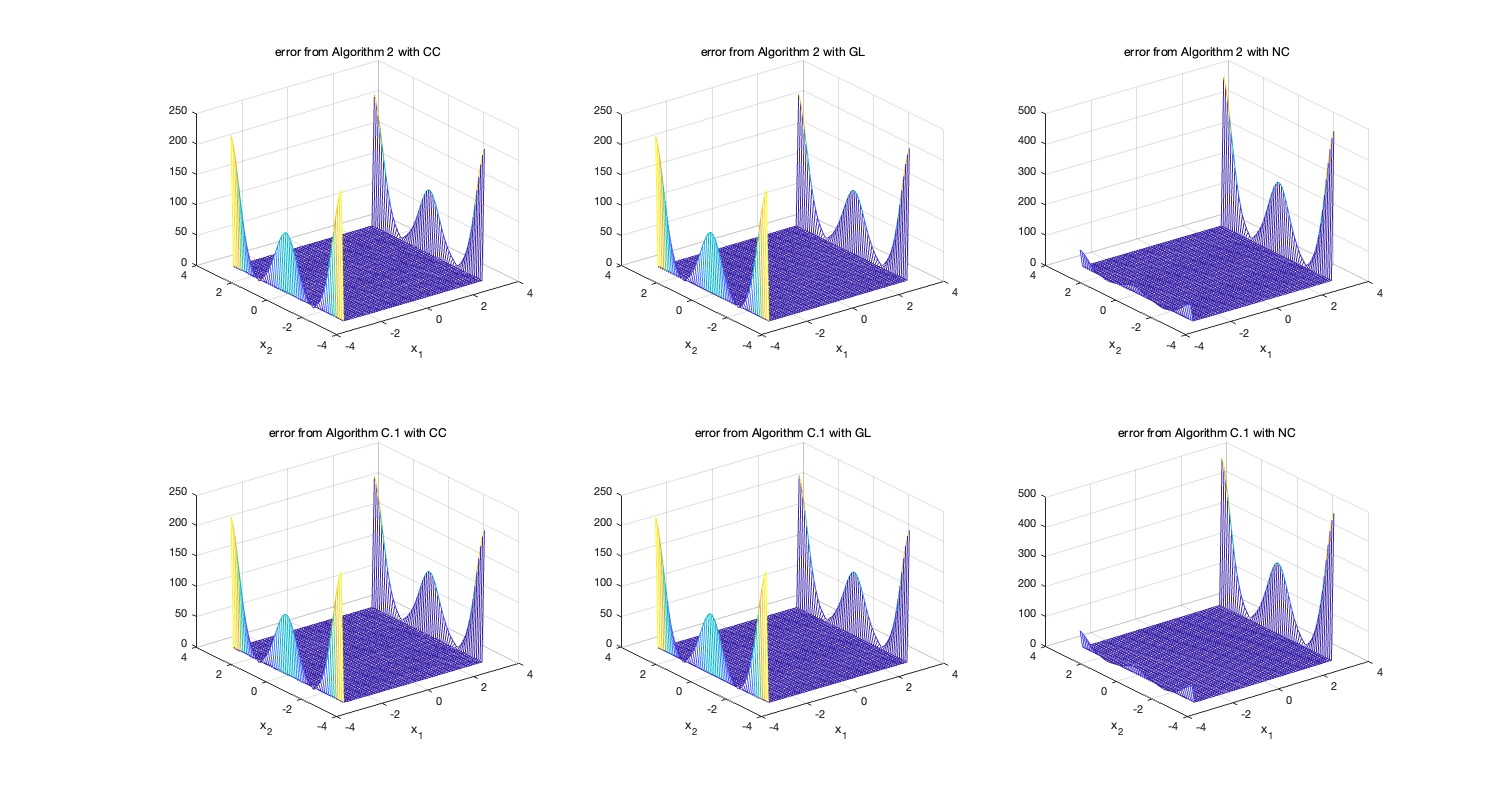}}\\
    \subfloat[$f_3(x_1,x_2)$]{\includegraphics[width=0.8\linewidth]{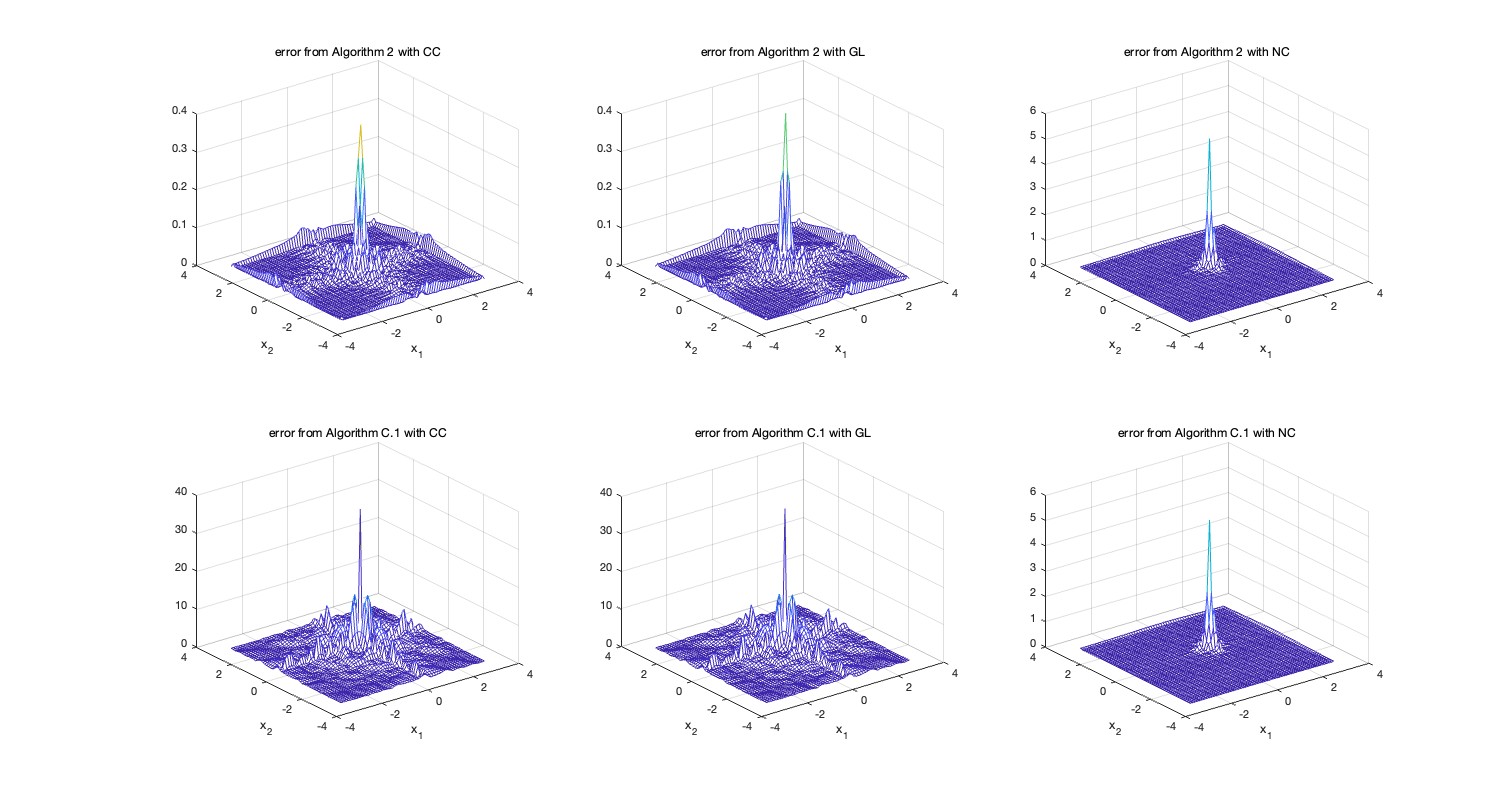}}\\
    \caption{For the special choice of $(b_1,b_2)$ and $\tau=1e-5$, the values of the associated error function obtained by applying Algorithms \ref{t-lasso:alg2-main} and \ref{t-lasso:alg1:app-main} with three tensor-product quadratures to the test functions $f_1(x_1,x_2)$, $f_2(x_1,x_2)$ and $f_3(x_1,x_2)$.}\label{t-lasso:app:fig7-main}
\end{figure}

According to Table \ref{t-lasso:tab1:app-main} and Figure \ref{t-lasso:app:fig7-main}, we can see that (a), for $f_1(x_1,x_2)$, $f_2(x_1,x_2)$ and $f_3(x_1,x_2)$, Algorithms \ref{t-lasso:alg2-main} and \ref{t-lasso:alg1:app-main} with three tensor-product quadratures are comparable in terms of the error function and running time; (b) Algorithms \ref{t-lasso:alg2-main} and \ref{t-lasso:alg1:app-main} with NC are not suitable for $f_3(x_1,x_2)$; and (c) for $f_1(x_1,x_2)$ and $f_2(x_1,x_2)$, Algorithm \ref{t-lasso:alg1:app-main} with CC and NC are faster than Algorithm \ref{t-lasso:alg2-main} with CC and NC, respectively.

\section*{Conflict of interest}
The authors declare that they have no conflict of interest.


\bibliographystyle{spmpsci}      
\bibliography{paper}

@article{kolda2009tensor,
  title={Tensor decompositions and applications},
  author={Kolda, Tamara G and Bader, Brett W},
  journal={SIAM review},
  volume={51},
  number={3},
  pages={455--500},
  year={2009},
  publisher={SIAM}
}

@article{bartel2025minimal,
  title={Minimal Subsampled Rank-1 Lattices for Multivariate Approximation with Optimal Convergence Rate},
  author={Bartel, Felix and Gilbert, Alexander D and Kuo, Frances Y and Sloan, Ian H},
  journal={arXiv preprint arXiv:2506.07729},
  year={2025}
}

@article{shin2017randomized,
  title={A randomized algorithm for multivariate function approximation},
  author={Shin, Yeonjong and Xiu, Dongbin},
  journal={SIAM Journal on Scientific Computing},
  volume={39},
  number={3},
  pages={A983--A1002},
  year={2017},
  publisher={SIAM}
}

@book{trefethen2019approximation,
  author = {Trefethen, Lloyd N.},
  title = {Approximation {T}heory and {A}pproximation {P}ractice},
  publisher = {Society for Industrial and Applied Mathematics},
  year = {2019},
  address = {Philadelphia, PA}
}

@article{wu2017randomized,
  title={A randomized tensor quadrature method for high dimensional polynomial approximation},
  author={Wu, Kailiang and Shin, Yeonjong and Xiu, Dongbin},
  journal={SIAM Journal on Scientific Computing},
  volume={39},
  number={5},
  pages={A1811--A1833},
  year={2017},
  publisher={SIAM}
}

@article{trefethen2017cubature,
  title={Cubature, approximation, and isotropy in the hypercube},
  author={Trefethen, Lloyd N},
  journal={SIAM Review},
  volume={59},
  number={3},
  pages={469--491},
  year={2017},
  publisher={SIAM}
}

@article{sloan1995hyperinterpolation,
  title={Polynomial interpolation and hyperinterpolation over general regions},
  author={Sloan, Ian H},
  journal={Journal of Approximation Theory},
  volume={83},
  number={2},
  pages={238--254},
  year={1995},
  publisher={Elsevier}
}

@article{osinsky2018pseudo,
  title={Pseudo-skeleton approximations with better accuracy estimates},
  author={Osinsky, AI and Zamarashkin, Nikolai L},
  journal={Linear Algebra and its Applications},
  volume={537},
  pages={221--249},
  year={2018},
  publisher={Elsevier}
}

@article{park2025accuracy,
  title={Accuracy and stability of {CUR} decompositions with oversampling},
  author={Park, Taejun and Nakatsukasa, Yuji},
  journal={SIAM Journal on Matrix Analysis and Applications},
  volume={46},
  number={1},
  pages={780--810},
  year={2025},
  publisher={SIAM}
}

@book{reimer2012multivariate,
  title={Multivariate {P}olynomial {A}pproximation},
  author={Reimer, Manfred},
  volume={144},
  year={2012},
  publisher={Birkh{\"a}user}
}

@article{udell2019big,
  title={Why are big data matrices approximately low rank?},
  author={Udell, Madeleine and Townsend, Alex},
  journal={SIAM Journal on Mathematics of Data Science},
  volume={1},
  number={1},
  pages={144--160},
  year={2019},
  publisher={SIAM}
}

@article{beckermann2019bounds,
  title={Bounds on the singular values of matrices with displacement structure},
  author={Beckermann, Bernhard and Townsend, Alex},
  journal={SIAM Review},
  volume={61},
  number={2},
  pages={319--344},
  year={2019},
  publisher={SIAM}
}

@article{trefethen2017multivariate,
  title={Multivariate polynomial approximation in the hypercube},
  author={Trefethen, Lloyd N.},
  journal={Proceedings of the American Mathematical Society},
  volume={145},
  number={11},
  pages={4837--4844},
  year={2017}
}

@article{mason1980near,
  title={Near-best multivariate approximation by {F}ourier series, {C}hebyshev series and {C}hebyshev interpolation},
  author={Mason, John Charles},
  journal={Journal of Approximation Theory},
  volume={28},
  number={4},
  pages={349--358},
  year={1980},
  publisher={Elsevier}
}

@article{griebel2023analysis,
  title={Analysis of tensor approximation schemes for continuous functions},
  author={Griebel, Michael and Harbrecht, Helmut},
  journal={Foundations of Computational Mathematics},
  volume={23},
  pages={219--240},
  year={2023},
  publisher={Springer}
}

@article{saibaba2022efficient,
  title={Efficient randomized tensor-based algorithms for function approximation and low-rank kernel interactions},
  author={Saibaba, Arvind K and Minster, Rachel and Kilmer, Misha E},
  journal={Advances in Computational Mathematics},
  volume={48},
  number={5},
  note={artile no. 66},
  year={2022},
  publisher={Springer}
}

@article{dolgov2021functional,
  title={Functional {T}ucker approximation using {C}hebyshev interpolation},
  author={Dolgov, Sergey and Kressner, Daniel and Str\"{o}ssner, Christoph},
  journal={SIAM Journal on Scientific Computing},
  volume={43},
  number={3},
  pages={A2190--A2210},
  year={2021},
  publisher={SIAM}
}

@article{griebel2023low,
  title={Low-rank approximation of continuous functions in {S}obolev spaces with dominating mixed smoothness},
  author={Griebel, Michael and Harbrecht, Helmut and Schneider, Reinhold},
  journal={Mathematics of Computation},
  volume={92},
  number={342},
  pages={1729--1746},
  year={2023}
}

@article {oseledets2011tensor,
	AUTHOR = {Oseledets, I. V.},
	TITLE = {Tensor-train decomposition},
	JOURNAL = {SIAM Journal on Scientific Computing},
	VOLUME = {33},
	YEAR = {2011},
	NUMBER = {5},
	PAGES = {2295--2317}
}

@article{hashemi2017chebfun,
  title={Chebfun in three dimensions},
  author={Hashemi, Behnam and Trefethen, Lloyd N},
  journal={SIAM Journal on Scientific Computing},
  volume={39},
  number={5},
  pages={C341--C363},
  year={2017},
  publisher={SIAM}
}

@article{soley2021functional,
  title={Functional tensor-train {C}hebyshev method for multidimensional quantum dynamics simulations},
  author={Soley, Micheline B and Bergold, Paul and Gorodetsky, Alex A and Batista, Victor S},
  journal={Journal of {C}hemical {T}heory and {C}omputation},
  volume={18},
  number={1},
  pages={25--36},
  year={2021},
  publisher={ACS Publications}
}

@article{an2021lasso,
  title={Lasso hyperinterpolation over general regions},
  author={An, Congpei and Wu, Hao-Ning},
  journal={SIAM Journal on Scientific Computing},
  volume={43},
  number={6},
  pages={A3967--A3991},
  year={2021}
}

@article{hamm2021perturbations,
  title={Perturbations of {CUR} decompositions},
  author={Hamm, Keaton and Huang, Longxiu},
  journal={SIAM Journal on Matrix Analysis and Applications},
  volume={42},
  number={1},
  pages={351--375},
  year={2021},
  publisher={SIAM}
}

@article{che2025efficient-siam,
  title={Efficient Randomized Algorithms for Fixed Precision Problem of Approximate Tucker Decomposition},
  author={Che, Maolin and Wei, Yimin and Yan, Hong},
  journal={SIAM Journal on Matrix Analysis and Applications},
  volume={46},
  number={1},
  pages={256--297},
  year={2025},
  publisher={SIAM}
}

@article{hallman2022block,
  title={A block bidiagonalization method for fixed-accuracy low-rank matrix approximation},
  author={Hallman, Eric},
  journal={SIAM Journal on Matrix Analysis and Applications},
  volume={43},
  number={2},
  pages={661--680},
  year={2022},
  publisher={SIAM}
}

@article{che2019randomized,
  title={Randomized algorithms for the approximations of {T}ucker and the tensor train decompositions},
  author={Che, Maolin and Wei, Yimin},
  journal={Advances in Computational Mathematics},
  volume={45},
  number={1},
  pages={395--428},
  year={2019}}

@article{halko2011finding,
	title={Finding structure with randomness: Probabilistic algorithms for constructing approximate matrix decompositions},
	author={Halko, Nathan and Martinsson, Per-Gunnar and Tropp, Joel A},
	journal={SIAM Review},
	volume={53},
	number={2},
	pages={217--288},
	year={2011},
	publisher={SIAM}
}

@article{yu2018efficient,
  title={Efficient randomized algorithms for the fixed-precision low-rank matrix approximation},
  author={Yu, Wenjian and Gu, Yu and Li, Yaohang},
  journal={SIAM Journal on Matrix Analysis and Applications},
  volume={39},
  number={3},
  pages={1339--1359},
  year={2018},
  publisher={SIAM}
}

@article{halko2011algorithm,
  title={An algorithm for the principal component analysis of large data sets},
  author={Halko, Nathan and Martinsson, Per-Gunnar and Shkolnisky, Yoel and Tygert, Mark},
  journal={SIAM Journal on Scientific computing},
  volume={33},
  number={5},
  pages={2580--2594},
  year={2011},
  publisher={SIAM}
}

@article{ailon2009fast,
  title={The fast {J}ohnson--{L}indenstrauss transform and approximate nearest neighbors},
  author={Ailon, Nir and Chazelle, Bernard},
  journal={SIAM Journal on Computing},
  volume={39},
  number={1},
  pages={302--322},
  year={2009},
  publisher={SIAM}
}

@article{clarkson2017low,
  title={Low-rank approximation and regression in input sparsity time},
  author={Clarkson, Kenneth L and Woodruff, David P},
  journal={Journal of the ACM (JACM)},
  volume={63},
  number={6},
  pages={1--45},
  year={2017},
  publisher={ACM New York, NY, USA}
}

@article{bamberger2022johnson,
  title={{J}ohnson--{L}indenstrauss Embeddings with {K}ronecker Structure},
  author={Bamberger, Stefan and Krahmer, Felix and Ward, Rachel},
  journal={SIAM Journal on Matrix Analysis and Applications},
  volume={43},
  number={4},
  pages={1806--1850},
  year={2022},
  publisher={SIAM}
}

@article{jin2021faster,
  title={Faster {J}ohnson--{L}indenstrauss transforms with {K}ronecker products},
  author={Jin, Ruhui and Kolda, Tamara G and Ward, Rachel},
  journal={Information and Inference: A Journal of the IMA},
  volume={10},
  number={4},
  pages={1533--1562},
  year={2021},
  publisher={Oxford University Press}
}

@article{chen2020structured,
  title={Structured random sketching for {PDE} inverse problems},
  author={Chen, Ke and Li, Qin and Newton, Kit and Wright, Stephen J},
  journal={SIAM Journal on Matrix Analysis and Applications},
  volume={41},
  number={4},
  pages={1742--1770},
  year={2020},
  publisher={SIAM}
}

@article{chen2021tensor,
  title={Tensor-structured sketching for constrained least squares},
  author={Chen, Ke and Jin, Ruhui},
  journal={SIAM Journal on Matrix Analysis and Applications},
  volume={42},
  number={4},
  pages={1703--1731},
  year={2021},
  publisher={SIAM}
}

@inproceedings{diao2018sketching,
  title={Sketching for {K}ronecker product regression and p-splines},
  author={Diao, Huaian and Song, Zhao and Sun, Wen and Woodruff, David},
  booktitle={International Conference on Artificial Intelligence and Statistics},
  pages={1299--1308},
  year={2018},
  organization={PMLR}
}

@article{woodruff2014sketching,
  title={Sketching as a tool for numerical linear algebra},
  author={Woodruff, David P and others},
  journal={Foundations and Trends{\textregistered} in Theoretical Computer Science},
  volume={10},
  number={1--2},
  pages={1--157},
  year={2014},
  publisher={Now Publishers, Inc.}
}

@article{sobczyk2022pylspack,
  title={{P}ylspack: {P}arallel algorithms and data structures for sketching, column subset selection, regression, and leverage scores},
  author={Sobczyk, Aleksandros and Gallopoulos, Efstratios},
  journal={ACM Transactions on Mathematical Software},
  volume={48},
  number={4},
  pages={1--27},
  year={2022},
  publisher={ACM New York, NY}
}

@article{pagh2013compressed,
  title={Compressed matrix multiplication},
  author={Pagh, Rasmus},
  journal={ACM Transactions on Computation Theory (TOCT)},
  volume={5},
  number={3},
  pages={1--17},
  year={2013},
  publisher={ACM New York, NY, USA}
}

@article{boutsidis2013improved,
  title={Improved matrix algorithms via the subsampled randomized {H}adamard transform},
  author={Boutsidis, Christos and Gittens, Alex},
  journal={SIAM Journal on Matrix Analysis and Applications},
  volume={34},
  number={3},
  pages={1301--1340},
  year={2013},
  publisher={SIAM}
}

@article{woolfe2008fast,
  title={A fast randomized algorithm for the approximation of matrices},
  author={Woolfe, Franco and Liberty, Edo and Rokhlin, Vladimir and Tygert, Mark},
  journal={Applied and Computational Harmonic Analysis},
  volume={25},
  number={3},
  pages={335--366},
  year={2008},
  publisher={Elsevier}
}

@article{tropp2017practical,
  title={Practical sketching algorithms for low-rank matrix approximation},
  author={Tropp, Joel A and Yurtsever, Alp and Udell, Madeleine and Cevher, Volkan},
  journal={SIAM Journal on Matrix Analysis and Applications},
  volume={38},
  number={4},
  pages={1454--1485},
  year={2017},
  publisher={SIAM}
}

@article{tropp2023randomized,
  title={Randomized algorithms for low-rank matrix approximation: Design, analysis, and applications},
  author={Tropp, Joel A and Webber, Robert J},
  journal={arXiv preprint arXiv:2306.12418},
  year={2023}
}

@article{musco2015randomized,
  title={Randomized block {K}rylov methods for stronger and faster approximate singular value decomposition},
  author={Musco, Cameron and Musco, Christopher},
  journal={Advances in neural information processing systems},
  volume={28},
  year={2015}
}

@article{rokhlin2010randomized,
  title={A randomized algorithm for principal component analysis},
  author={Rokhlin, Vladimir and Szlam, Arthur and Tygert, Mark},
  journal={SIAM Journal on Matrix Analysis and Applications},
  volume={31},
  number={3},
  pages={1100--1124},
  year={2010},
  publisher={SIAM}
}

@article{bjarkason2019pass,
  title={Pass-efficient randomized algorithms for low-rank matrix approximation using any number of views},
  author={Bjarkason, Elvar K},
  journal={SIAM Journal on Scientific Computing},
  volume={41},
  number={4},
  pages={A2355--A2383},
  year={2019},
  publisher={SIAM}
}

@inproceedings{clarkson2009numerical,
  title={Numerical linear algebra in the streaming model},
  author={Clarkson, Kenneth L and Woodruff, David P},
  booktitle={Proceedings of the forty-first annual ACM symposium on Theory of computing},
  pages={205--214},
  year={2009}
}

@article{meier2024fast,
  title={Fast randomized numerical rank estimation for numerically low-rank matrices},
  author={Meier, Maike and Nakatsukasa, Yuji},
  journal={Linear Algebra and its Applications},
  volume={686},
  pages={1--32},
  year={2024},
  publisher={Elsevier}
}

@article{martinsson2016randomized,
  title={A randomized blocked algorithm for efficiently computing rank-revealing factorizations of matrices},
  author={Martinsson, Per-Gunnar and Voronin, Sergey},
  journal={SIAM Journal on Scientific Computing},
  volume={38},
  number={5},
  pages={S485--S507},
  year={2016},
  publisher={SIAM}
}

@book{canuto2006spectral,
  title={{S}pectral {M}ethods: {F}undamentals in {S}ingle {D}omains},
  author={Canuto, C. and Hussaini, M. and Quarteroni, A. and Zang, T.},
  year={2006},
  publisher={Springer, Berlin}
}

@article{stewart1993early,
  title={On the early history of the singular value decomposition},
  author={Stewart, Gilbert W},
  journal={SIAM Review},
  volume={35},
  number={4},
  pages={551--566},
  year={1993},
  publisher={SIAM}
}

@article{bigoni2016spectral,
  title={Spectral tensor-train decomposition},
  author={Bigoni, Daniele and Engsig-Karup, Allan P and Marzouk, Youssef M},
  journal={SIAM Journal on Scientific Computing},
  volume={38},
  number={4},
  pages={A2405--A2439},
  year={2016},
  publisher={SIAM}
}

@article{boutsidis2017optimal,
  title={Optimal {CUR} Matrix Decompositions},
  author={Boutsidis, Christos and Woodruff, David P},
  journal={SIAM Journal on Computing},
  volume={46},
  number={2},
  pages={543--589},
  year={2017},
  publisher={SIAM}
}

@article{mahoney2009cur,
  title={{CUR} matrix decompositions for improved data analysis},
  author={Mahoney, Michael W and Drineas, Petros},
  journal={Proceedings of the National Academy of Sciences},
  volume={106},
  number={3},
  pages={697--702},
  year={2009},
  publisher={National Acad Sciences}
}

@article{sorensen2016deim,
  title={A {DEIM} induced {CUR} factorization},
  author={Sorensen, Danny C and Embree, Mark},
  journal={SIAM Journal on Scientific Computing},
  volume={38},
  number={3},
  pages={A1454--A1482},
  year={2016},
  publisher={SIAM}
}

@book {golub2013matrix,
  AUTHOR = {Golub, G. and Van Loan, C.},
  TITLE = {Matrix {C}omputations},
  EDITION = {Fourth},
  PUBLISHER = {Johns Hopkins University Press, Baltimore, MD},
  YEAR = {2013},
  PAGES = {xiv+756},
  ISBN = {978-1-4214-0794-4; 1-4214-0794-9; 978-1-4214-0859-0},
  MRCLASS = {65-02 (65Fxx)},
  MRNUMBER = {3024913},
  MRREVIEWER = {J{\"o}rg Liesen},
}

@article{chaturantabut2010nonlinear,
  title={Nonlinear model reduction via discrete empirical interpolation},
  author={Chaturantabut, Saifon and Sorensen, Danny C},
  journal={SIAM Journal on Scientific Computing},
  volume={32},
  number={5},
  pages={2737--2764},
  year={2010},
  publisher={SIAM}
}

@article{duersch2020randomized,
  title={Randomized projection for rank-revealing matrix factorizations and low-rank approximations},
  author={Duersch, Jed A and Gu, Ming},
  journal={SIAM Review},
  volume={62},
  number={3},
  pages={661--682},
  year={2020},
  publisher={SIAM}
}

@article{voronin2017efficient,
  title={Efficient algorithms for {CUR} and interpolative matrix decompositions},
  author={Voronin, Sergey and Martinsson, Per-Gunnar},
  journal={Advances in Computational Mathematics},
  volume={43},
  pages={495--516},
  year={2017},
  publisher={Springer}
}

@article{bebendorf2000approximation,
  title={Approximation of boundary element matrices},
  author={Bebendorf, Mario},
  journal={Numerische Mathematik},
  volume={86},
  pages={565--589},
  year={2000},
  publisher={Springer}
}

@article{harbrecht2012low,
  title={On the low-rank approximation by the pivoted {C}holesky decomposition},
  author={Harbrecht, Helmut and Peters, Michael and Schneider, Reinhold},
  journal={Applied Numerical Mathematics},
  volume={62},
  number={4},
  pages={428--440},
  year={2012},
  publisher={Elsevier}
}

@article{gu1996efficient,
  title={Efficient algorithms for computing a strong rank-revealing {QR} factorization},
  author={Gu, Ming and Eisenstat, Stanley C},
  journal={SIAM Journal on Scientific Computing},
  volume={17},
  number={4},
  pages={848--869},
  year={1996},
  publisher={SIAM}
}

@article{drineas2008relative,
  title={Relative-error {CUR} matrix decompositions},
  author={Drineas, Petros and Mahoney, Michael W and Muthukrishnan, Shan},
  journal={SIAM Journal on Matrix Analysis and Applications},
  volume={30},
  number={2},
  pages={844--881},
  year={2008},
  publisher={SIAM}
}

@article{chiu2013sublinear,
  title={Sublinear randomized algorithms for skeleton decompositions},
  author={Chiu, Jiawei and Demanet, Laurent},
  journal={SIAM Journal on Matrix Analysis and Applications},
  volume={34},
  number={3},
  pages={1361--1383},
  year={2013},
  publisher={SIAM}
}

@article{cortinovis2020low,
  title={Low-rank approximation in the {F}robenius norm by column and row subset selection},
  author={Cortinovis, Alice and Kressner, Daniel},
  journal={SIAM Journal on Matrix Analysis and Applications},
  volume={41},
  number={4},
  pages={1651--1673},
  year={2020},
  publisher={SIAM}
}

@article{goreinov1997theory,
  title={A theory of pseudoskeleton approximations},
  author={Goreinov, Sergei A and Tyrtyshnikov, Eugene E and Zamarashkin, Nickolai L},
  journal={Linear Algebra and its Applications},
  volume={261},
  number={1-3},
  pages={1--21},
  year={1997}
}

@article{derezinski2021determinantal,
  title={Determinantal point processes in randomized numerical linear algebra},
  author={Derezinski, Micha{\l} and Mahoney, Michael W},
  journal={Notices of the American Mathematical Society},
  volume={68},
  number={1},
  pages={34--45},
  year={2021}
}

@article{boutsidis2014near,
  title={Near-optimal column-based matrix reconstruction},
  author={Boutsidis, Christos and Drineas, Petros and Magdon-Ismail, Malik},
  journal={SIAM Journal on Computing},
  volume={43},
  number={2},
  pages={687--717},
  year={2014},
  publisher={SIAM}
}

@incollection{gidisu2021hybrid,
  title={A hybrid {DEIM} and leverage scores based method for {CUR} index selection},
  author={Gidisu, Perfect Y and Hochstenbach, Michiel E},
  booktitle={European Consortium for Mathematics in Industry},
  pages={147--153},
  year={2021},
  publisher={Springer}
}

@article{xia2024making,
  title={Making the {N}ystr{\"o}m method highly accurate for low rank approximations},
  author={Xia, Jianlin},
  journal={SIAM Journal on Scientific Computing},
  volume={46},
  number={2},
  pages={A1076--A1101},
  year={2024},
  publisher={SIAM}
}

@article {An_ran_2025,
    AUTHOR = {An, Congpei and Ran, Jiashu},
     TITLE = {Hard thresholding hyperinterpolation over general regions},
   JOURNAL = {J. Sci. Comput.},
  FJOURNAL = {Journal of Scientific Computing},
    VOLUME = {102},
      YEAR = {2025},
    NUMBER = {2},
  PAGES = {26, Paper No. 37},
      ISSN = {0885-7474,1573-7691},
       DOI = {10.1007/s10915-024-02754-4},
}

@book{Foucart2013compressing,
	author = {Foucart, S. and Rauhut, H.},
	doi = {https://doi.org/10.1007/978-0-8176-4948-7},
	edition = {1},
	pages = {XVIII, 625},
	publisher = {Birkh{\"a}user New York, NY},
	series = {Applied and Numerical Harmonic Analysis},
	title = {A {M}athematical {I}ntroduction to {C}ompressive {S}ensing},
	year = {2013}}

@article{an2022exactness,
  title={On the quadrature exactness in hyperinterpolation},
  author={An, Congpei and Wu, Hao-Ning},
  journal={BIT Numerical Mathematics},
  volume={62},
  number={4},
  pages={1899--1919},
  year={2022},
  publisher={Springer}
}

@article{lin2021distributed,
  title={Distributed filtered hyperinterpolation for noisy data on the sphere},
  author={Lin, Shao-Bo and Wang, Yu Guang and Zhou, Ding-Xuan},
  journal={SIAM Journal on Numerical Analysis},
  volume={59},
  number={2},
  pages={634--659},
  year={2021},
  publisher={SIAM}
}

@article{LeGia2001uniform,
  title={The uniform norm of hyperinterpolation on the unit sphere in an arbitrary number of dimensions},
  author={Le Gia, T and Sloan, IH},
  journal={Constructive approximation},
  volume={17},
  number={2},
  pages={249--265},
  year={2001},
  publisher={Springer}
}

@article{an2024bypassing,
	author = {An, Congpei and Wu, Hao-Ning},
	doi = {https://doi.org/10.1016/j.jco.2023.101789},
	journal = {Journal of Complexity},
	pages = {101789},
	title = {{Bypassing the quadrature exactness assumption of hyperinterpolation on the sphere}},
	volume = {80},
	year = {2024}}

@article{dai2006hyperinterpolation,
  title={On generalized hyperinterpolation on the sphere},
  author={Dai, Feng},
  journal={Proceedings of the American Mathematical Society},
  volume={134},
  number={10},
  pages={2931--2941},
  year={2006}
}

@article{Wade2013hyperinterpolation,
  title={On hyperinterpolation on the unit ball},
  author={Wade, Jeremy},
  journal={Journal of Mathematical Analysis and Applications},
  volume={401},
  number={1},
  pages={140--145},
  year={2013},
  publisher={Elsevier}
}

@article{Wang2017needlet,
  title={Fully discrete needlet approximation on the sphere},
  author={Wang, Yu Guang and Le Gia, Quoc T and Sloan, Ian H and Womersley, Robert S},
  journal={Applied and Computational Harmonic Analysis},
  volume={43},
  number={2},
  pages={292--316},
  year={2017},
  publisher={Elsevier}
}

\end{document}